\documentclass[11pt,twoside,a4paper]{article}
\usepackage{color}
\usepackage[utf8]{inputenc}
\usepackage{amsmath}
\usepackage{amssymb}
\usepackage{amsthm}
\usepackage{latexsym}
\usepackage{stmaryrd}
\usepackage{shuffle}
\usepackage{tikz,tkz-tab, tikz-cd}

\tikzset{
  ncone/.pic={
	\draw (0,0)--(0,0.2);
  }
}

\tikzset{
  nctwo/.pic={
    \draw (0,0)--(0,0.2);
	\draw (0.1,0)--(0.1,0.2);
  }
}

\tikzset{
  nctwoW/.pic={
    \draw (0,0.2)--(0,0)--(0.1,0)--(0.1,0.2);
  }
}

\tikzset{
  nctwoWW/.pic={
    \draw (0,0.2)--(0,0)--(0.2,0)--(0.2,0.2);
  }
}

\tikzset{
  ncthreeWW/.pic={
    \draw (0,0.2)--(0,0)--(0.3,0)--(0.3,0.2);
	\draw (0.2,0)--(0.2,0.2);
  }
}

\tikzset{
  ncthree/.pic={
    \draw (0,0)--(0,0.2);
	\draw (0.1,0)--(0.1,0.2);
	\draw (0.2,0)--(0.2,0.2);
  }
}

\tikzset{
  ncthreeW/.pic={
    \draw (0,0.2)--(0,0)--(0.2,0)--(0.2,0.2);
	\draw (0.1,0)--(0.1,0.2);
  }
}

\tikzset{
  ncfour/.pic={
    \draw (0,0)--(0,0.2);
	\draw (0.1,0)--(0.1,0.2);
	\draw (0.2,0)--(0.2,0.2);
	\draw (0.3,0)--(0.3,0.2);
  }
}
\tikzset{
  ncfive/.pic={
    \draw (0,0)--(0,0.2);
	\draw (0.1,0)--(0.1,0.2);
	\draw (0.2,0)--(0.2,0.2);
	\draw (0.3,0)--(0.3,0.2);
	\draw (0.4,0)--(0.4,0.2);
  }
}

\tikzset{
  ncfourW/.pic={
    \draw (0,0.2)--(0,0)--(0.3,0)--(0.3,0.2);
	\draw (0.1,0)--(0.1,0.2);
	\draw (0.2,0)--(0.2,0.2);
  }
}
\tikzset{
  ncfiveW/.pic={
    \draw (0,0.2)--(0,0)--(0.4,0)--(0.4,0.2);
	\draw (0.1,0)--(0.1,0.2);
	\draw (0.2,0)--(0.2,0.2);
	\draw (0.3,0)--(0.3,0.2);
  }
}

\tikzset{
  nconeinsidetwoWW/.pic={
    \path (0,0) pic {nctwoWW}; \path (0.1,0.1) pic {ncone};
  }
}

\tikzset{
  nconeinsidethreeWW/.pic={
    \path (0,0) pic {ncthreeWW}; \path (0.1,0.1) pic {ncone};
  }
}
\tikzset{
  nconeinsidethreerightWW/.pic={
    \draw (0,0.2)--(0,0)--(0.3,0)--(0.3,0.2);
    \draw (0.1,0)--(0.1,0.2); \draw (0.2,0.1)--(0.2,0.3);
  }
}
\tikzset{
  nconeinsidethreeleftWW/.pic={
    \draw (0,0.2)--(0,0)--(0.3,0)--(0.3,0.2);
    \draw (0.2,0)--(0.2,0.2); \draw (0.1,0.1)--(0.1,0.3);
  }
}
\tikzset{
  nctwoWWW/.pic={
    \draw (0,0.2)--(0,0)--(0.3,0)--(0.3,0.2);
  }
}

\tikzset{
  nconeoneinsidetwoWW/.pic={
	\path (0,0) pic {ncone};
    \path (0.1,0) pic {nctwoWW}; 
	\path (0.2,0.1) pic {ncone};
  }
}

\newcommand{\ncone}{\,\tikz[x=1.3cm,y=1.3cm]{\path (0,0) pic {ncone}}\,}
\newcommand{\nctwo}{\,\tikz[x=1.3cm,y=1.3cm]{\path (0,0) pic {nctwo}}\,}
\newcommand{\nctwoW}{\,\tikz[x=1.3cm,y=1.3cm]{\path (0,0) pic {nctwoW}}\,}

\newcommand{\ncthree}{\,\tikz[x=1.3cm,y=1.3cm]{\path (0,0) pic {ncthree}}\,}
\newcommand{\ncthreeW}{\,\tikz[x=1.3cm,y=1.3cm]{\path (0,0) pic {ncthreeW}}\,}
\newcommand{\ncfour}{\,\tikz[x=1.3cm,y=1.3cm]{\path (0,0) pic {ncfour}}\,}
\newcommand{\ncfive}{\,\tikz[x=1.3cm,y=1.3cm]{\path (0,0) pic {ncfive}}\,}
\newcommand{\ncfourW}{\,\tikz[x=1.3cm,y=1.3cm]{\path (0,0) pic {ncfourW}}\,}
\newcommand{\ncfiveW}{\,\tikz[x=1.3cm,y=1.3cm]{\path (0,0) pic {ncfiveW}}\,}
\newcommand{\nconeinsidetwoWW}{\,\tikz[x=1.3cm,y=1.3cm]{
  \path (0,0) pic {nconeinsidetwoWW};
}\,}
\newcommand{\nconetwoW}{\,\tikz[x=1.3cm,y=1.3cm]{
  \path (0,0) pic {ncone}; \path (0.1,0) pic {nctwoW};
}\,}
\newcommand{\nctwoWone}{\,\tikz[x=1.3cm,y=1.3cm]{
  \path (0,0) pic {nctwoW}; \path (0.2,0) pic {ncone};
}\,}
\newcommand{\nctwoWtwoW}{\,\tikz[x=1.3cm,y=1.3cm]{
  \path (0,0) pic {nctwoW}; \path (0.2,0) pic {nctwoW};
}\,}
\newcommand{\nctwoWinsidetwoWWW}{\,\tikz[x=1.3cm,y=1.3cm]{
  \path (0,0) pic {nctwoWWW}; \path (0.1,0.1) pic {nctwoW};
}\,}
\newcommand{\nconeonetwoW}{\,\tikz[x=1.3cm,y=1.3cm]{
  \path (0,0) pic {ncone}; \path (0.1,0) pic {ncone}; \path (0.2,0) pic {nctwoW};
}\,}

\newcommand{\nconethreeW}{\,\tikz[x=1.3cm,y=1.3cm]{
  \path (0,0) pic {ncone}; \path (0.1,0) pic {ncthreeW};
}\,}

\newcommand{\nconeinsidethreeWW}{\,\tikz[x=1.3cm,y=1.3cm]{
  \path (0,0) pic {nconeinsidethreeWW};
}\,}
\newcommand{\nconeinsidethreerightWW}{\,\tikz[x=1.3cm,y=1.3cm]{
  \path (0,0) pic {nconeinsidethreerightWW};
}\,}
\newcommand{\nconeinsidethreeleftWW}{\,\tikz[x=1.3cm,y=1.3cm]{
  \path (0,0) pic {nconeinsidethreeleftWW};
}\,}

\newcommand{\nconeoneinsidetwoWW}{\,\tikz[x=1.3cm,y=1.3cm]{
  \path (0,0) pic {nconeoneinsidetwoWW};
}\,}

\usepackage[T1]{fontenc}   
\usepackage[english]{babel}

\theoremstyle{definition}
\newtheorem{defi}{\indent Definition}[subsection]
\newtheorem{rem}[defi]{\indent Remark}
\newtheorem{ex}[defi]{\indent Example}
\theoremstyle{theorem}
\newtheorem{lemma}[defi]{\indent Lemma}
\newtheorem{cor}[defi]{\indent Corollary}
\newtheorem{theo}[defi]{\indent Theorem}
\newtheorem{prop}[defi]{\indent Proposition}

\newcommand{\frontstick}{\,\raisebox{-1pt}{\begin{tikzpicture}
\draw [line width=1pt,] (0,0)--(0,0.25);
\end{tikzpicture}}\kern+2pt}
\newcommand{\bl}{\operatorname{bl}}
\newcommand{\ho}{\operatorname{ho}}

\newcommand{\g}{\mathfrak{g}}

\newcommand{\Id}{\operatorname{Id}}
\newcommand{\id}{\operatorname{id}}
\newcommand{\Vect}{\operatorname{Vect}}
\newcommand{\Conv}{\operatorname{Conv}}
\newcommand{\st}[1]{#1}
\newcommand{\mst}[1]{\vec{#1}}
\newcommand{\rmst}[1]{\overset{\ldots}{#1}}
\newcommand{\nc}{\operatorname{nc}}
\newcommand{\ind}{\operatorname{ind}}

\newcommand{\latNCC}{\operatorname{NCC}}
\newcommand{\latNP}{\operatorname{NCP}} 
\newcommand{\latNCP}{\operatorname{NCP}}

\newcommand{\latSP}{\operatorname{SP}}
\newcommand{\latSC}{\operatorname{SC}}
\newcommand{\latSMP}{\operatorname{SMP}}
\newcommand{\latNMP}{\operatorname{NMP}}

\newcommand{\cut}{\operatorname{cut}}

\newcommand{\Bgap}{\mathbf{S}_0} 
\newcommand{\Bgapnc}{\mathbf{N}_0} 
\newcommand{\Hgap}{\mathbf{S}} 
\newcommand{\Hgapnc}{\mathbf{N}} 
\newcommand{\Bblcomp}{\mathbf{C}^{\operatorname{comp}}} 
\newcommand{\Bblnccomp}{\mathbf{B}^{\operatorname{comp}}} 
\newcommand{\Bbl}{\mathbf{C}} 
\newcommand{\Bblnc}{\mathbf{B}} 

\newcommand{\Expl}{\mathcal{E}_\prec}
\newcommand{\Expr}{\mathcal{E}_\succ}

\newcommand{\isopil}{\stackrel{\raisebox{0.1ex}[0ex][0ex]{\(\sim\)}}%
			{\raisebox{-0.15ex}[0.28ex]{\(\rightarrow\)}}}

\newcommand{\Pto}{\raisebox{-8pt}{%
\!\begin{tikzpicture}
  \node at (0.0,0.0) {$\to$};
  \node at (-0.04,-0.14) {\tiny $P$};
\end{tikzpicture}\!%
}}

\newcommand{\ordsum}{\uplus}

\newcommand{\mob}{\mu}

\newcommand{\K}{\mathbb{K}}
\newcommand{\N}{\mathbb{N}}
\newcommand{\NCP}{\mathcal{NCP}}
\newcommand{\SC}{\mathcal{SC}}
\newcommand{\NCC}{\mathcal{NCC}}
\newcommand{\SP}{\mathcal{SP}}

\renewcommand{\P}{\mathcal{P}}

\definecolor{red}{rgb}{1.,0.,0.}
\definecolor{green}{rgb}{0.,1.,0.}
\definecolor{blue}{rgb}{0.,0.,1.}
\definecolor{orange}{rgb}{1.,0.8431372549019608,0.}

\title{Operads of (noncrossing) partitions, interacting bialgebras,\\ and moment-cumulant relations}

\date{April 14, 2020}

\author{Kurusch Ebrahimi-Fard$^a$, Lo\"\i c Foissy$^b$, Joachim Kock$^c$, Fr\'ed\'eric Patras$^d$\\ \\
{\small \it $^a$Department of Mathematical Sciences}\\
{\small \it Norwegian University of Science and Technology (NTNU)}\\
{\small \it 7491 Trondheim, Norway.}\\
{\small \it email: {\rm kurusch.ebrahimi-fard@ntnu.no}}\\[0.2cm]    
{\small \it $^b$F\'ed\'eration de Recherche Math\'ematique du Nord Pas de Calais FR 2956}\\
{\small \it Laboratoire de Math\'ematiques Pures et Appliqu\'ees Joseph Liouville}\\
{\small \it Universit\'e du Littoral C\^ote d'Opale-Centre Universitaire de la Mi-Voix}\\ 
{\small \it 50, rue Ferdinand Buisson, CS 80699,  62228 Calais Cedex, France}\\
{\small \it email: {\rm foissy@univ-littoral.fr}}\\[0.2cm]
{\small \it $^c$Departament de Matem\`atiques}\\
{\small \it Universitat Aut\`onoma de Barcelona}\\
{\small \it 08193 Bellaterra (Barcelona), Spain}\\
{\small \it email: {\rm kock@mat.uab.cat}}\\[0.2cm]
{\small \it $^d$Universit\'e C\^ote d'Azur, CNRS, UMR 7351}\\ 
{\small \it Laboratoire J.A.Dieudonn\'e}\\
{\small \it Parc Valrose, 06108 Nice Cedex 02, France.}\\
{\small \it email: {\rm patras@unice.fr}}}

\begin{document}

\maketitle

\begin{abstract}
  We establish and explore a relationship between two approaches to
  moment-cumulant relations in free probability theory: on one side the
  main approach, due to Speicher, given in terms of M\"obius inversion on
  the lattice of noncrossing partitions, and on the other side the more
  recent non-commutative shuffle-algebra approach, where the moment-cumulant relations
  take the form of certain exponential-logarithm relations. We achieve this
  by exhibiting two operad structures on (noncrossing) partitions,
  different in nature: one is an ordinary, non-symmetric operad whose
  composition law is given by insertion into gaps between elements, the
  other is a coloured, symmetric operad with composition law expressing 
  refinement of blocks. We show that these operad structures interact
  so as to make the corresponding incidence bialgebra of the former a
  comodule bialgebra for the latter. Furthermore, this interaction is compatible with
  the shuffle structure and thus unveils how the two approaches are intertwined.
  Moreover, the constructions and results are general enough to extend to
  ordinary set partitions.
\end{abstract}
  
\tableofcontents


\vspace{0.5cm}


\section{Introduction}

One of Rota's main motivations for developing the now ubiquitous machinery
of incidence algebras and M\"obius inversion~\cite{Rota0} was to provide a
clean combinatorial explanation of the moment-cumulant relations in
classical probability, known already to involve the Bell numbers (see for
example~\cite{Kendall:vol1}). The idea is that both moments $m$ and
cumulants $c$ (of a fixed random variable) are considered as linear
forms on the incidence coalgebra of the lattice of set partitions, and that
the relationships amount to convolution with the zeta function and its inverse, the
M\"obius function of the lattice (see also Speed~\cite{Speed}):
$$ 
	m = \zeta * c
	\qquad\qquad
	c = \mu * m . 
$$
More generally, Rota's seminal paper with Wallstrom~\cite{RotaWallstrom97}
showed that many of the basic aspects of stochastic analysis can be
rendered combinatorially in terms of the lattice of set partitions.

A beautiful extension of Rota's ideas was found by
Speicher~\cite{Speicher94, Speicher} to account for the relationship
between moments and free cumulants in the setting of Voiculescu's theory of free
probability~\cite{DVoi}. Speicher showed that the free moment-cumulant formulae follow the same
pattern as the classical case if just the lattice of set partitions is
replaced by the lattice of noncrossing partitions~\cite{SpeicherNica}.
Beyond the importance of this in the context of free probability (see also
Biane~\cite{Biane0, Biane}), it is as well a striking new application of the
combinatorial topic of noncrossing partitions\footnote{Kreweras' 1972 paper ``Sur les partitions non crois\'ees d'un cycle'' \cite{Kre} initiated the study of noncrossing partitions in combinatorics.}, which is also
important in exploring the interface between 
algebraic geometry and representation theory 
(see for example the recent survey \cite{Baumeister-et.al}),
and which comes up in numerous other contexts~\cite{MC, Simion}.

Recently, two of the present authors proposed a different approach to
moment-cumulant relations~\cite{EFP15,EFP16,EFP17,EFP18}, which is not
based on the lattice of noncrossing partitions, incidence algebras or
M\"obius inversion. Instead, it employs a certain non-commutative shuffle
algebra and considers moments and free cumulants as images of a particular
Hopf algebra character, respectively infinitesimal character.
Moment-cumulant relations are encoded through exponential-logarithm
correspondences implied by fixpoint equations involving so-called
half-shuffle products $\prec$ and $\succ$. For instance, in the free case the moment character $\phi$
and the free cumulant infinitesimal character $\kappa$ are related through
a half-shuffle fixpoint equation in non-commutative shuffle algebra
$$
	\phi = \varepsilon + \kappa \prec \phi,
$$
where $\varepsilon$ is the counit.
The shuffle-algebra approach supports a Lie theoretic perspective, which is augmented by the notions of pre-Lie and, more generally, brace algebra \cite{chapotonMM}.

\medskip

The motivation for the present work was to uncover the relationship,
hitherto completely mysterious, between the two approaches to
moment-cumulant formulae. The general framework found to accommodate the
two disparate approaches is the theory of operads and their incidence
bialgebras. In a nutshell we show that the two approaches are governed each
by their particular operad of noncrossing partitions; we then show how
these two operads interact at the level of the corresponding bialgebras,
and derive comparisons between the two settings in terms of this
interaction.

Somewhat surprisingly, the framework turns out to be general enough
to cover also the case of ordinary set partitions. Here, too, there exists 
a pair of interacting operads. As a consequence we obtain in this case as well a close
connection between the M\"obius-theoretic moment-cumulant relations \`a la 
Rota, and the (little studied) shuffle-algebraic approach in that context.
In the following summary we concentrate on the case of noncrossing partitions, 
since this was our original motivation and central aim.

\medskip

The first operad is an ordinary nonsymmetric operad of noncrossing
partitions, where each noncrossing partition of degree $n$ is considered an
$(n+1)$-ary operation. The composition law for this operad consists in inserting $n+1$
noncrossing partitions into the gaps of a noncrossing partition of degree
$n$. This includes the ``outer gaps'', i.e., in front of the whole partition and at the end of
it. To any operad is associated a bialgebra (see \cite{Foissy} and \cite{GKT-comb}), 
sometimes called its incidence
bialgebra, which always has the feature of being ``right-sided'', meaning
that the coproduct is linear in the left-hand tensor factor but gives monomials
in the right-hand tensor factor. We show that this bialgebra is an unshuffle
bialgebra (also called codendriform bialgebra), and that the resulting
fixpoint equations in the graded dual of the bialgebra are analogous to those of \cite{EFP15}. 

The second operad is a coloured symmetric operad, whose colours are the
positive integers. If a noncrossing partition has $k$ blocks, then it is
considered a $k$-ary operation. The colour of each input slot is the number
of elements in the block, and the output colour of the operation is the
total number of elements. The composition law for this operad works by
substituting a noncrossing partition 
for a block with the same number of elements. The resulting effect is
thus to refine partitions. The incidence bialgebra corresponding to this operad is shown
to be closely related to the incidence coalgebra of the lattice of 
noncrossing partitions
in an interesting way: there is a canonical coalgebra homomorphism $\Phi :
\Bblnc_{\operatorname{lat}} \to \Bblnc_{\operatorname{opd}}$ from the incidence
coalgebra of the lattice of noncrossing partitions to the incidence bialgebra
of the operad. Intervals in the lattice are type equivalent (in the sense
of Speicher~\cite{Speicher94}) precisely when they have the same image
under $\Phi$. M\"obius inversion in $\Bblnc_{\operatorname{lat}}$ is induced from that in
$\Bblnc_{\operatorname{opd}}$, but the latter is closer to what is actually used
in free probability, since it works with the noncrossing partitions
themselves instead of intervals of noncrossing partitions. Because of this
we can now forget about the lattice, and work directly with the incidence
bialgebra of the (block-substitution) operad.

The two bialgebras of noncrossing partitions induced by the two operad
structures have essentially the same multiplicative basis, both being free on the set
of noncrossing partitions. The key point of this work, which allows for the
comparison envisaged, is that these two bialgebra
structures together form a comodule bialgebra.  This is an intricate 
structure of two interacting bialgebras which recently has appeared in numerical analysis~\cite{CEFM}, local dynamical systems \cite{EFM17}, and stochastic integration and 
renormalization~\cite{Bruned-Hairer-Zambotti}.
Precisely, in Theorem~\ref{thm:comodulebialg} we establish that
the gap-insertion bialgebra is an
unshuffle-type bialgebra object in the symmetric monoidal category of comodules
for the block-substitution bialgebra. In particular, the block-substitution
bialgebra coacts on the gap-insertion bialgebra by bialgebra homomorphisms.
An immediate consequence of this is that the convolution algebra of the
block-substitution bialgebra acts on the shuffle algebra, i.e., the convolution
algebra of the gap-insertion bialgebra. An attractive feature is now that
this action is compatible with the splitting of the shuffle product into
half-shuffles (Theorem~\ref{thm:rightact-unshuffle}), and therefore with the fixpoint equations mentioned
above. Finally the distinguished character (solution to a fixpoint equation)
can be convolution-inverted (to provide shuffle logarithms). These can be
computed by M\"obius-inversion style formulae.

Throughout we emphasize noncrossing partitions, since this was our original
motivation, but in fact the theory we develop is general enough to cover
also the case of ordinary set partitions, so as to get analogous formulae
also for classical cumulants. A subtle
point for this to work is the Galois connection between the
lattice of set partitions and the lattice of noncrossing partitions, which is used to
transfer certain noncrossing features to the setting of ordinary
partitions, and which also leads to some direct comparisons between
classical and free cumulants. In particular this involves
the nesting preorder of a partition, and
the notions of lowersets and uppersets for partitions.

\smallskip

Impetus for this project came from more abstract work by two of the present
authors, concerned with general relationships between operads and
combinatorial Hopf algebras (and bialgebras). Foissy~\cite{FoissyOperads}
developed a theoretical framework allowing to understand how quite generally
operads equipped with certain $B_\infty$-actions induce comodule
bialgebras. In a different line of investigation, G\'alvez, Kock and Tonks, in a series of papers starting with \cite{GKT1}, developed
the notion of decomposition spaces, a general homotopical framework for
incidence algebras and M\"obius inversion, revealing many classical
combinatorial Hopf algebras to be incidence algebras (but not of posets).
(This framework covers also operads, via their two-sided bar constructions.)

The use of comultiplications to describe convolution
products (in combinatorics) goes back to Rota~\cite{Rota0}.  The
importance of Hopf algebra structure rather than just coalgebra structure
was stressed by Schmitt~\cite{Schmitt:1994}. 
In the more specific context of non-commutative probability
theory, various authors have recently exploited Hopf algebras
from different perspectives. We briefly mention
the works of Friedrich--McKay~\cite{Friedrich}, Hasebe--Lehner~\cite{HasLeh},
Mastnak--Nica~\cite{MastNica},
Gabriel~\cite{Gabriel15}, and Manzel--Sch\"urmann~\cite{MSch}.
The precise relations to~\cite{EFP15,EFP16,EFP17,EFP18} are not yet fully understood.

Operadic approaches to moment-cumulant formulae have been
exploited in other ways in recent years. Josuat-Verg\`es,
Menous, Novelli, and Thibon~\cite{Thibon} studied the (free) operad of
Schr\"oder trees to obtain an operadic version of the
half-shuffle equations of Ebrahimi-Fard and Patras~\cite{EFP15}, and
also related this to Speicher's original formulae~\cite{Speicher94}.
They do not, however, consider operad structures directly on noncrossing
partitions. Proposition~\ref{prop:genrels} below gives an explicit
presentation of our gap-insertion operad in terms of generators
(corolla trees) and relations, premonished by the way Schr\"oder
trees are shown to encode noncrossing partitions in \cite{Thibon}.
Their encoding becomes an operad map from the Schr\"oder operad
to our gap-insertion operad of noncrossing partitions.

A different operadic viewpoint on free moment-cumulant relations
was taken by Drummond-Cole~\cite{DC1, DC2}, who gave an
operadic interpretation of Speicher's multiplicativity~\cite{Speicher94}, and
exhibited the free moment-cumulant relations in terms of convolution
of maps from a cooperad to an endomorphism operad.
Again, the operads and cooperads considered are not directly on 
noncrossing partitions, but rather on auxiliary classes of decorated trees.

Finally we mention the traffic spaces of Male~\cite{Male}; these are 
defined as algebras for a certain operad of graph operations. While 
this is a very general framework, the operad of graph operations does 
not seem to specialize directly to any of the main constructions of the
present work.

\bigskip


{\bf{Acknowledgments}}: We would like to thank 
G.~Drummond-Cole,
F.~Lehner, and R.~Speicher for fruitful discussions. 
This research project received financial support from the CNRS Program PICS 07376: 
«Alg\`ebres de Hopf combinatoires et probabilit\'es non commutatives». 
We also thank the Centre de Recerca Matem\`atica (CRM) in Barcelona for its hospitality. 
J.K.~was supported by grants MTM2016-80439-P (AEI/FEDER, UE) of 
Spain and 2017-SGR-1725 of Catalonia. This work was partially supported by the project Pure Mathematics in Norway, funded by Bergen Research Foundation and Troms{\o} Research Foundation and by the European Research Council (ERC) under the European Union's Horizon 2020 research and innovation program (grant agreement No.670624).


\section{Partitions and noncrossing partitions --- notation and 
conventions}
\label{CompPart}

This section introduces notation and conventions and briefly recalls some
basic notions related to set partitions.


\subsection{Definitions}
\label{ssect:def}

\begin{enumerate}

\item We denote by $\K$ a commutative base field of characteristic zero.
All algebraic objects in this work, such as vector spaces, algebras, 
coalgebras, pre-Lie algebras, etc., will be defined over $\K$. 

For algebras with an augmentation $\varepsilon: A \to \K$, we write 
$A_+:=\operatorname{Ker}\varepsilon$ for the augmentation ideal.

\item By $\N$ we denote the linearly ordered set of non-negative integers, 
and by $\N^\ast$ the set of positive integers.
For $n \in \N^\ast$, we denote by $[n]$ the set 
$\{ 1,2,3,\ldots,n\}=\llbracket 1,n\rrbracket$.

For a finite subset $X=\{x_1,\dots,x_k\} \subset \N$,
we write $\sharp X=k$ for its cardinality, and
we write $X+n$ for the $n$-shifted subset $\{x_1+n,\ldots,x_k+n\}$. 

\end{enumerate}

\medskip

Let $X$ be a finite, linearly ordered set. A \emph{partition} of $X$ into
disjoint subsets $\pi_i,\ i=1,\ldots,k$ (called blocks) is written
$P=\{\pi_1,\dots,\pi_k\}$. If the order on the set of blocks
matters we will write it as a
sequence, $(\pi_1,\ldots,\pi_k)$, and call it a \emph{composition} instead of a
partition.\footnote{The terminology is by analogy with integer 
compositions: a composition of an integer $n$ is a sequence of non-zero integers
$(n_1,\ldots,n_k)$ such that $n=n_1+\dots +n_k$.} The degree of the 
partition or the composition is the cardinality of $X$. Except when
otherwise stated, partitions are ordered by coarsening: $P\leq Q$ means that $P$ is finer than $Q$.

\bigskip

For the purposes of this paper, it is important to require a
linear order on all underlying sets that are partitioned. Many notions
(in particular the notion of noncrossing partition)
depend crucially on this linear order, but do not depend on the
specific underlying set. This is to say we are only interested
in set partitions up to isomorphism. An \emph{isomorphism} between two
partitions is a monotone bijection of the underlying linearly
ordered sets compatible with the block structure (and for
compositions also with the order on the set of blocks). It
is clear that every partition (or 
composition) is uniquely isomorphic to one with underlying
linearly ordered set $[n]=\{1,2,\ldots,n\}$, for some $n\in \N$. These are called
{\it{standard partitions}}, and the process of replacing a partition
with this standard representative of its isomorphism class (iso-class) is
often called {\it{standardization}}. It is convenient most of the time
to work with the standard representatives of each class. When
non-standard representatives arise in the constructions (such as
when considering subsets of $[n]$), it must be remembered that it is only the
isomorphism class that matters.  
Working with iso-classes rather than 
restricting to standardized partitions gets us closer to
the combinatorial intuition conveyed
by the common pictorial representation of set partitions by block 
diagrams
exploited throughout.
(Standardization could always be
performed as desired for convenience, but does not affect the 
validity of the constructions.)

\medskip

For $k,n \in \N^\ast$, we denote by $\latSP(k,n)$ the set of iso-classes 
  of
  partitions of $n$-element sets into $k$ blocks. The set $\latSP(0,0)$ contains only
  the empty partition. We put:
\begin{align*}
	\latSP	&=\bigsqcup_{1\leq k\leq n} \latSP(k,n),
	&   \latSP(n)&=\bigsqcup_{k\leq n} \latSP(k,n),
		&   \latSP_0&=\latSP(0,0) \sqcup \latSP.
\end{align*}
When referring to a member of any of these sets, we shall always pick the
standard representative, having underlying set $[n]$.

\medskip

Given a monotone inclusion of linearly orderd sets $X \subset Y$,
and given a 
partition $P=\{\pi_1,\dots,\pi_k\}$ of $Y$, we write $P_{|X}$ for the
induced partition of $X$ (whose blocks are the non-empty intersections 
$\pi_i\cap X$).

\begin{defi}\label{conv}
  Let $X$ be a non-empty, finite subset of $\N$ (or of an arbitrary 
  linearly ordered set $Y$).
  The \emph{convex hull} of
  $X$ is by definition $\Conv(X)=\llbracket \min(X),\max(X)\rrbracket$. We
  shall say that $X$ is \emph{convex} if $\Conv(X)=X$. Any finite subset $X\subset
  \N$ decomposes uniquely as \allowdisplaybreaks
\begin{align*}
	X&=X_1\sqcup \cdots \sqcup X_k
\end{align*}
with each $X_i$ convex and each $X_i \sqcup X_j$ {\em not} convex for  $i\neq j$.
The $X_1,\ldots, X_k$ are called the \emph{convex components} of $X$.
\end{defi}

\begin{defi}[Noncrossing partitions] 
  A partition $P=\{\pi_1,\ldots,\pi_k\}$ is called \emph{noncrossing}
  when there are no $a,b\in \pi_i,\ c,d\in \pi_j$ with $i\not=j$ such 
  that $a<c<b<d$.  For example, $\{\{1\},\{2,4\},\{3\}\} =$
  \nconeoneinsidetwoWW\ is noncrossing,
    whereas  $\{\{1,3\},\{2,4\}\} =$
  \begin{tikzpicture}[scale=1.2]
  \draw (0,0.2)--(0,0)--(0.2,0)--(0.2,0.2);
  \draw (0.1,0.14)--(0.1,-0.06)--(0.3,-0.06)--(0.3,0.14);
\end{tikzpicture} is crossing.

  For $k,n \in \N^\ast$, we denote by $\latNCP(k,n)$ the set of iso-classes
  of noncrossing partitions of an $n$-element set into $k$ 
  blocks. The set $\latNCP(0,0)$ contains only the empty partition. We put:
\begin{align*}
	\latNCP	&=\bigsqcup_{1\leq k\leq n} \latNCP(k,n),
	&   \latNCP(n)&=\bigsqcup_{k\leq n} \latNCP(k,n),
		&   \latNCP_0&=\latNCP(0,0) \sqcup \latNCP.
\end{align*}
\end{defi}

Note that if $X \subset [n]$ and $P$ is a noncrossing partition of
$[n]$ then $P_{|X}$ is a noncrossing partition of $X$.

\begin{defi}[Noncrossing closure]\label{nc}
The \emph{noncrossing closure} of a set partition $P=\{\pi_1,\dots,\pi_k\}$ of
degree $n$, written $\nc(P)$, is the noncrossing partition defined by
$$
	\nc(P):=\min\limits_{\leq}\{Q\in \latNCP(n) \mid P\leq Q\}.
$$
Here $P\leq Q$ means that $P$ is finer than $Q$. In plain words, $\nc$ joins any two blocks that cross.
For example, $\nc(\,\begin{tikzpicture}[scale=1.2]
  \draw (0,0.2)--(0,0)--(0.2,0)--(0.2,0.2);
  \draw (0.1,0.14)--(0.1,-0.06)--(0.3,-0.06)--(0.3,0.14);
\end{tikzpicture}\,) = \ncfourW$.
\end{defi}

Anticipating the two operads that we will define on partitions, 
we note that $\nc$ preserves the degree of a partition. It will therefore be useful in connection with the gap-insertion operad. In contrast, it does not preserve the number of blocks, and for this reason will not play any role for the block-substitution operads.

\begin{rem}
  Throughout this paper we work with ``naked'' partitions, like the ones
  introduced above. However, we wish to stress that all the constructions
  and results go through  with partitions decorated by elements from an alphabet
  $A$. By this we mean that each partition --- with underlying linearly
  ordered set $X$, say --- is equipped with a map $a: X \to A$ (not
  required to be injective), pictured as
  $$
  
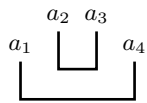
\begin{tikzpicture}
	  \draw [line width=1.0pt,] (0.0,0.5)--(0.0,0.0)--(1.5,0.0)--(1.5,0.5);
	  \draw [line width=1.0pt,] (0.5,0.9)--(0.5,0.4)--(1.0,0.4)--(1.0,0.9);
	  \node at (0.0,0.7) {\footnotesize $a_1$};
	  \node at (0.5,1.1) {\footnotesize $a_2$};
	  \node at (1.0,1.1) {\footnotesize $a_3$};
	  \node at (1.5,0.7) {\footnotesize $a_4$};
  \end{tikzpicture}
  $$
  These decorations do not interfere with any of the operations we shall
  perform on partitions.  For example, a decoration transfers 
  canonically along monotone bijections (such as standardization): if $X \to A$ is a decoration, and 
  $[n] \isopil X$ is the unique monotone bijection performing the 
  standardization, then the decoration of the standardized partition is simply 
  the composite map $[n] \to X \to A$.  The decorated situation is 
  relevant in probability theory where the $a_i$ will be random variables.
  Since such decorations do not interfere with the theory we develop,
  for the sake of simplicity
  we prefer to stick with ``naked'' partitions, which can be regarded
  the univariate case, that is, of a single random variable.  The only place where
  we shall need to refer to decorations is in Proposition~\ref{prop:sumofall}
  where we comment on relations with the double tensor algebra (over a 
  probability space $A$).
\end{rem}


\section{Gap-insertion operad and associated structures}
\label{1stoperad}

We describe here a series of algebraic structures on set partitions arising
from the idea of ``inserting a partition into another partition''.
We start with the basic insertion
operation, naturally encoded by an operad structure, and then turn to the
definition of further associated algebraic structures on partitions, such
as unshuffle coproducts and Hopf algebras.

In order to give an explicit combinatorial description of the bialgebra 
structure resulting from the operad, we exploit a preorder on the set of
blocks of a given partition, and the attendant notions of lowerset and 
upperset. These notions are quite standard in the context of noncrossing 
partitions, but less so for general partitions --- it is indeed tailored to 
study the relations between the two settings.

\smallskip

Before presenting the gap-insertion operad, we briefly recall that a
non-symmetric, single-coloured set operad $\mathcal P$ is a sequence of
sets ${\mathcal P}(n)$, $n \in \N$, equipped with a unit operation $I \in
{\mathcal P}(1)$ and a composition law: 
\allowdisplaybreaks
\begin{align*}
	{\mathcal P}(r)\ \times \ ({\mathcal P}(n_1)\times\cdots\times{\mathcal P}(n_r)) 
	& \to{\mathcal P}(n_1+\dots+n_r)\\
	(p, (q_1,\dots, q_r))&\mapsto p \circ (q_1,\dots,q_r),
\end{align*}
for $r\in \N$ and $n_1,\dots,n_r \in \N$, satisfying the axioms
$$
	I \circ (q) =q,
	\qquad 
	p \circ (I,\dots ,I) =p,
$$
and
$$
	\big ( p \circ 
	(q_1,\dots,q_r) \big) \circ ( h_1^1,\dots,h_{n_1}^1,\dots,h_1^r,\dots,h_{n_r}^r )
	= p \circ \big(q_1 \circ (h_1^1,\dots,h_{n_1}^1),\dots, q_r \circ 
	(h_1^r,\dots,h_{n_r}^r)\big).
$$ 
An operad $\mathcal{P}$ is \emph{reduced}\footnote{The term \emph{positive operad} is also used~\cite{Aguiar-Mahajan}.} if $\mathcal{P}(0) = \emptyset$, that is, it has no nullary operations. 

Equivalently, an operad can be defined using the notion of ``partial composition law''
$$
	\circ_i:{\mathcal P}(m)\times {\mathcal P}(n)\to {\mathcal P}(m+n-1),\quad i=1,\ldots,m,
$$
where
$$
	p\circ_i q:= p \circ (I,\dots,I,q,I,\dots,I),
$$
with $q$ located in position $i$. 

A set operad gives rise to an operad in the category of vector spaces by
replacing sets by the vector spaces they span, and cartesian products of
sets by tensor products of vector spaces.


\subsection{A non-symmetric single-colour operad of partitions}
\label{1stOperad}

The idea of the notion of \it gap-insertion \rm operad on partitions, which we proceed to define,
is easy (when displaying a partition pictorially): each partition $P \in \latSP(n)$ (or noncrossing partition $P\in
\latNCP(n)$) is considered an operation of arity $n+1$, its input slots
being the $n+1$ gaps between the elements, including the ``gap'' before $1$
and the ``gap'' after $n$. It can thus receive $n+1$ other partitions (or
noncrossing partitions) $Q_1,\ldots,Q_{n+1}$, which are simply inserted into
the gaps. For example, inserting the partition
$\{\{1,2\}\}$ between $2$ and $3$ into $\{\{1,2,3\}\}$ gives the partition
$\{\{1,2,5\},\{3,4\}\}$ and reads pictorially:
\begin{align*}
\begin{tikzpicture}[line cap=round,line join=round,>=triangle 45,x=0.3cm,y=0.3cm]
\clip(1.8,1.) rectangle (4.2,2.);
\draw [line width=0.8pt,] (2.,1.)-- (2.,2.);
\draw [line width=0.8pt,] (2.,1.)-- (4.,1.);
\draw [line width=0.8pt,] (3.,1.)-- (3.,2.);
\draw [line width=0.8pt,] (4.,1.)-- (4.,2.);
\end{tikzpicture}
\diamond_3
\begin{tikzpicture}[line cap=round,line join=round,>=triangle 45,x=0.3cm,y=0.3cm]
\clip(1.8,1.) rectangle (4.2,2.);
\draw [line width=0.8pt] (2.,1.)-- (2.,2.);
\draw [line width=0.8pt] (2.,1.)-- (3.,1.);
\draw [line width=0.8pt] (3.,1.)-- (3.,2.);
\end{tikzpicture}= \
\begin{tikzpicture}[line cap=round,line join=round,>=triangle 45,x=0.3cm,y=0.3cm]
\clip(3.8,1.) rectangle (14.2,3.);
\draw [line width=0.8pt] (4.,2.)-- (4.,1.);
\draw [line width=0.8pt] (4.,1.)-- (10.,1.);
\draw [line width=0.8pt] (10.,1.)-- (10.,2.);
\draw [line width=0.8pt] (7.,1.)-- (7.,2.);
\draw [line width=0.8pt] (8.,2.)-- (8.,3.);
\draw [line width=0.8pt] (9.,2.)-- (9.,3.);
\draw [line width=0.8pt] (8.,2.)-- (9.,2.);
\end{tikzpicture}
\end{align*}
where $\diamond_3$ denotes the partial composition law. Conceptually it is clear that
this defines an operad. Formalising it just requires appropriate reindexing
of the elements in the underlying sets. If $P$ is noncrossing and
$Q_1,\ldots,Q_{n+1}$ are all noncrossing, then it is clear that the result is again
noncrossing (as in the example just given).

It is simpler to describe the operad structure from the partial composition law viewpoint:

\begin{defi}[The gap-insertion operad]
We set $\SP(n):=\latSP(n-1)$, in particular, $\SP(0) = \emptyset$ (so that 
the operad will be reduced) and $\SP(1)=\{\emptyset\}$; the empty 
partition will be the operad unit.
This sequence of sets is given a unital non-symmetric operadic structure 
as follows: for $P=\{\pi_1,\dots,\pi_k\}\in \SP(m)$ 
and $Q=\{\rho_1,\dots,\rho_l\}\in \SP(n)$, we define the partial 
composition law (for $i=1,\ldots,m$) by:
\begin{equation*}
	P \diamond_i Q:=\{\chi(\pi_1),\dots,\chi(\pi_k),\rho_1+i-1,\dots,\rho_l+i-1\},
\end{equation*}
where $\chi(p):=p$ if $p<i$, $\chi(p):=p+n-1$ else.
Equivalently: given a sequence of partitions
$P\in \SP(m),Q_1\in \SP(n_1),\dots,Q_m\in\SP(n_m)$, the composition law is
\begin{equation}
\label{compositionlaw}
	P \diamond Q:=\gamma(P)\cup Q_1\cup Q_2+n_1\cup\dots\cup Q_m+n_1+\dots+n_{m-1},
\end{equation}
where $\gamma(i):=n_1+\dots+n_i$.
\end{defi}

\textbf{Example}.
1) \begin{align*}
	&\big\{\{1,2,3\}\big\} \diamond_3 \big(\big\{\{1\};\{2\}\big\}\big) 
	=\big\{\{1,2,5\};\{3\};\{4\}\big\}.
\end{align*}
Pictorially:
\begin{align*}
\begin{tikzpicture}[line cap=round,line join=round,>=triangle 45,x=0.3cm,y=0.3cm]
\clip(1.7,1.) rectangle (4.2,2.);
\draw [line width=0.8pt,] (2.,1.)-- (2.,2.);
\draw [line width=0.8pt,] (2.,1.)-- (4.,1.);
\draw [line width=0.8pt,] (3.,1.)-- (3.,2.);
\draw [line width=0.8pt,] (4.,1.)-- (4.,2.);
\end{tikzpicture}
\, \diamond_3 \;
\begin{tikzpicture}[line cap=round,line join=round,>=triangle 45,x=0.3cm,y=0.3cm]
\clip(1.8,1.) rectangle (4.2,2.);
\draw [line width=0.8pt,color=green] (2.,1.)-- (2.,2.);
\draw [line width=0.8pt,color=green] (3.,1.)-- (3.,2.);
\end{tikzpicture}
\! = \
\begin{tikzpicture}[line cap=round,line join=round,>=triangle 45,x=0.3cm,y=0.3cm]
\clip(3.8,1.) rectangle (14.2,3.);
\draw [line width=0.8pt] (4.,2.)-- (4.,1.);
\draw [line width=0.8pt] (4.,1.)-- (10.,1.);
\draw [line width=0.8pt] (10.,1.)-- (10.,2.);
\draw [line width=0.8pt] (7.,1.)-- (7.,2.);
\draw [line width=0.8pt,color=green] (8.,2.)-- (8.,3.);
\draw [line width=0.8pt,color=green] (9.,2.)-- (9.,3.);
\end{tikzpicture}
\end{align*}

2) 
\begin{align*}
&\big\{\{1,2,3\}\big\} \diamond 
\big(\big\{\{1\};\{2,3\}\big\}\;,\;\big\{\{1,2\}\big\}\;,\;\big\{\{1\};\{2\}\big\}\;,\;\big\{\{1,2,3,4\}\big\}\big)\\
&=\big\{\{1\};\{2,3\};\{4,7,10\};\{5,6\};\{8\};\{9\};\{11,12,13,14\}\big\}.
\end{align*}
Pictorially:
\begin{align*}
\begin{tikzpicture}[line cap=round,line join=round,>=triangle 45,x=0.3cm,y=0.3cm]
\clip(1.8,1.) rectangle (4.2,2.);
\draw [line width=0.8pt,] (2.,1.)-- (2.,2.);
\draw [line width=0.8pt,] (2.,1.)-- (4.,1.);
\draw [line width=0.8pt,] (3.,1.)-- (3.,2.);
\draw [line width=0.8pt,] (4.,1.)-- (4.,2.);
\end{tikzpicture}
\;\diamond\;
(\;\begin{tikzpicture}[line cap=round,line join=round,>=triangle 45,x=0.3cm,y=0.3cm]
\clip(1.8,1.) rectangle (4.2,2.);
\draw [line width=0.8pt,color=blue] (2.,1.)-- (2.,2.);
\draw [line width=0.8pt,color=blue] (4.,1.)-- (4.,2.);
\draw [line width=0.8pt,color=blue] (3.,1.)-- (3.,2.);
\draw [line width=0.8pt,color=blue] (3.,1.)-- (4.,1.);
\end{tikzpicture}
\;,\;
\begin{tikzpicture}[line cap=round,line join=round,>=triangle 45,x=0.3cm,y=0.3cm]
\clip(2.8,1.) rectangle (4.2,2.);
\draw [line width=0.8pt,color=red] (2.,1.)-- (2.,2.);
\draw [line width=0.8pt,color=red] (4.,1.)-- (4.,2.);
\draw [line width=0.8pt,color=red] (3.,1.)-- (3.,2.);
\draw [line width=0.8pt,color=red] (3.,1.)-- (4.,1.);
\end{tikzpicture}
\;,\;
\begin{tikzpicture}[line cap=round,line join=round,>=triangle 45,x=0.3cm,y=0.3cm]
\clip(2.8,1.) rectangle (4.2,2.);
\draw [line width=0.8pt,color=green] (2.,1.)-- (2.,2.);
\draw [line width=0.8pt,color=green] (4.,1.)-- (4.,2.);
\draw [line width=0.8pt,color=green] (3.,1.)-- (3.,2.);
\end{tikzpicture}
\;,\;
\begin{tikzpicture}[line cap=round,line join=round,>=triangle 45,x=0.3cm,y=0.3cm]
\clip(1.8,1.) rectangle (5.2,2.);
\draw [line width=0.8pt,color=orange] (2.,1.)-- (2.,2.);
\draw [line width=0.8pt,color=orange] (4.,1.)-- (4.,2.);
\draw [line width=0.8pt,color=orange] (3.,1.)-- (3.,2.);
\draw [line width=0.8pt,color=orange] (2.,1.)-- (5.,1.);
\draw [line width=0.8pt,color=orange] (5.,1.)-- (5.,2.);
\end{tikzpicture}\;)
\ = \
\begin{tikzpicture}[line cap=round,line join=round,>=triangle 45,x=0.3cm,y=0.3cm]
\clip(0.8,1.) rectangle (14.2,3.);
\draw [line width=0.8pt,color=blue] (1.,1.)-- (1.,2.);
\draw [line width=0.8pt,color=blue] (2.,1.)-- (2.,2.);
\draw [line width=0.8pt,color=blue] (2.,1.)-- (3.,1.);
\draw [line width=0.8pt,color=blue] (3.,1.)-- (3.,2.);
\draw [line width=0.8pt] (4.,2.)-- (4.,1.);
\draw [line width=0.8pt] (4.,1.)-- (10.,1.);
\draw [line width=0.8pt] (10.,1.)-- (10.,2.);
\draw [line width=0.8pt,color=red] (5.,2.)-- (6.,2.);
\draw [line width=0.8pt,color=red] (6.,2.)-- (6.,3.);
\draw [line width=0.8pt,color=red] (5.,2.)-- (5.,3.);
\draw [line width=0.8pt] (7.,1.)-- (7.,2.);
\draw [line width=0.8pt,color=green] (8.,2.)-- (8.,3.);
\draw [line width=0.8pt,color=green] (9.,2.)-- (9.,3.);
\draw [line width=0.8pt,color=orange] (11.,1.)-- (14.,1.);
\draw [line width=0.8pt,color=orange] (11.,1.)-- (11.,2.);
\draw [line width=0.8pt,color=orange] (12.,1.)-- (12.,2.);
\draw [line width=0.8pt,color=orange] (13.,1.)-- (13.,2.);
\draw [line width=0.8pt,color=orange] (14.,1.)-- (14.,2.);
\end{tikzpicture}
\end{align*}

As already observed, the set of noncrossing partitions is stable under the composition law of the operad $\SP$:

\begin{lemma}
  The sequence $\NCP(n):=\latNCP(n-1)$, equipped with the composition 
  law $\diamond$, defines a set operad called 
  the {\emph{noncrossing gap-insertion operad}}.
\end{lemma}

Since the composition laws $\diamond$ of the two operads $\SP$ and $\NCP$
work by insertion into gaps, and since neither the inclusion nor the 
noncrossing closure operator (see \ref{nc}) changes those gaps, the 
following is straightforward to check:

\begin{lemma}
  Both maps
\begin{tikzcd}
  \latNCP \ar[r, "i"', shift right] & 
  \latSP \ar[l, "\nc"', shift right]
\end{tikzcd}
induce morphisms of operads,
$$
\begin{tikzcd}
  \NCP \ar[r, "i"', shift right] & 
  \SP , \ar[l, "\nc"', shift right]
\end{tikzcd}
$$
exhibiting $\NCP$ as a retract of $\SP$.
\end{lemma}

The noncrossing gap-insertion operad admits the following simple
presentation in terms of generators and relations.

\begin{prop}\label{prop:genrels}
For any $n\geq 2$, we put $p_n=\{[n-1]\} \in \NCP(n)$. Then the operad $(\NCP,\diamond)$ is generated by the elements $p_n$, $n\geq 2$, with the relations:
\begin{equation*}
	\forall m,n\geq 2,\qquad p_m \diamond_m p_n=p_n \diamond_1 p_m. 
\end{equation*}
\end{prop}

This result (and its corollary below) is not essential
for later developments in this article, and we therefore limit ourselves to
a short proof assuming some familiarity with the theory of operads.
\begin{proof}
  Firstly, for any $m,n \geq 2$,
\begin{equation}\label{pn}
	p_m \diamond_m p_n=\{[m-1],[n-1]\}=p_n \diamond_1 p_m.
\end{equation}
We denote by $(\mathcal{Q},\circ)$ the non-symmetric operad generated by
the elements $q_m$, $m\geq 2$, with these relations. Classically, one
represents an $m$-ary operation $q_m$ by a corolla 
with $m$ leaves, and elements of the
non-symmetric operad freely generated by these elements by planar rooted
trees. Since (\ref{pn}) holds in $\NCP$, there
exists a unique operad morphism $\Theta:\mathcal{Q}\longrightarrow \NCP$
sending $q_m$ to $p_m$, for $m\geq 2$.

Let us prove that $\Theta$ is surjective. Let $\pi=\{\pi_1,\ldots ,
\pi_k\}\in \latNCP(k,n-1)$, we prove that it belongs to
$\Theta(\mathcal{Q}(n))$ using induction on $k$. If $k=1$, we have
$\pi=\Theta(q_n)$. Otherwise, assume that the block containing $1$
is $\pi_1$ and that it is of cardinality $l-1$. Then there exist
noncrossing partitions $\pi^{(2)},\ldots,\pi^{(l)}$ such that
\[
	\pi=p_l\diamond(I,\pi^{(2)},\ldots,\pi^{(l)}).
\]
The number of blocks of $\pi^{(i)}$ is strictly smaller than $k$, so for any $i$ there exist $q^{(i)}\in \mathcal{Q}$ such that $\pi^{(i)}=\Theta(q^{(i)})$. Hence:
\[
	\pi=\Theta(q_l)\diamond (I,\Theta(q^{(2)}),\ldots, \Theta(q^{(l)}))=\Theta(q_l\circ (I,q^{(2)},\ldots,q^{(l)})).
\]

Since for any $n$, the $n$th component of the free non-symmetric operad generated by the elements $q_n$ identifies with the set of planar rooted trees with $n$ leaves, to obtain $\mathcal{Q}$, one has to quotient by the relations
\[
\begin{tikzpicture}[line cap=round,line join=round,>=triangle 45,x=0.5cm,y=0.5cm]
\clip(2.,1.) rectangle (6.,4.);
\draw [line width=0.8pt] (4.,1.)-- (4.,2.);
\draw [line width=0.8pt] (4.,2.)-- (3.,3.);
\draw [line width=0.8pt] (3.,3.)-- (4.,4.);
\draw [line width=0.8pt] (3.,3.)-- (2.,4.);
\draw [line width=0.8pt] (4.,2.)-- (4.,3.);
\draw [line width=0.8pt] (4.,2.)-- (6.,3.);
\draw (2.5,4.2) node[anchor=north west] {{\small \textit{m}}};
\draw (4.,3.5) node[anchor=north west] {{\small \textit{n}-1}};
\end{tikzpicture}
=\begin{tikzpicture}[line cap=round,line join=round,>=triangle 45,x=0.5cm,y=0.5cm]
\clip(2.,1.) rectangle (6.,4.);
\draw [line width=0.8pt] (4.,1.)-- (4.,2.);
\draw [line width=0.8pt] (4.,2.)-- (5.,3.);
\draw [line width=0.8pt] (5.,3.)-- (6.,4.);
\draw [line width=0.8pt] (5.,3.)-- (4.,4.);
\draw [line width=0.8pt] (4.,2.)-- (2.,3.);
\draw [line width=0.8pt] (4.,2.)-- (4.,3.);
\draw (4.5,4.2) node[anchor=north west] {{\small \textit{n}}};
\draw (2.4,3.6) node[anchor=north west] {{\small \textit{m}-1}};
\end{tikzpicture},
\]
so (by a simple recursion) the quotient $\mathcal{Q}(n)$ identifies with
the set of planar rooted trees with $n$ leaves such that for any internal
vertex $v$, the leftmost child of $v$ is a leaf. The number of such trees
is the Catalan number $cat_n =\frac{1}{n}{2n-2\choose n-1}$, so 
$\sharp(\mathcal{Q}(n))\leq
cat_n=\sharp(\NCP(n))$. Therefore, $\Theta$ is bijective.
\end{proof}

\begin{cor}
  $\NCP$-algebras are vector spaces $V$ carrying for any $n\geq 2$ an
  $n$-multilinear map $\langle-,\ldots,-\rangle$ such that for any $m,n \geq 2$,
  and $x_1,\ldots,x_{m+n-1} \in V$:
\[
	\langle x_1,\ldots,x_{m-1},\langle x_m,\ldots,x_{m+n-1}\rangle\rangle
	=\langle \langle x_1,\ldots,x_m \rangle,x_{m+1},\ldots,x_{m+n-1}\rangle.
\]
\end{cor}


\subsection{Induced brace algebra structures}
\label{ssect:algstructures}

The operad structure on $\SP$ and $\NCP$ induces various algebraic
structures on the free vector spaces. The ones in this section
are not directly used in this paper, but we mention them
briefly, for the sake of completeness, and since a brace algebra structure
on fundamental objects such as partitions or noncrossing partitions can be
expected to have meaningful applications (recall that brace algebras are
useful in the study of algebraic structures up to homotopy, for example
on the Hochschild complex \cite{gv}). The reader is referred to
\cite{FoissyOperads} for details, proofs and further references.

\bigskip

Denote by $T(V):=\bigoplus_{n\in\N} V^{\otimes n}$ the tensor 
algebra over $V$, and by $T^+(V):=\bigoplus_{n\in\N^\ast} V^{\otimes n}$ 
its augmentation ideal. We use word notation $v_1\dots v_n$ for 
the tensors $v_1\otimes \cdots \otimes v_n$ in $T(V)$, and write
$\mathbf{1} \in V^{\otimes 0}$ for the empty word.

\begin{defi} 
  A \emph{brace algebra} is a vector space $V$ equipped with a linear map: $\{ -,-
  \}$ from $V\otimes T(V)$ to $V$ such that:
  \begin{itemize}

  \item $\forall v \in V,\ \{v,\mathbf{1}\}=v$.

  \item $\forall v,y_1,\dots,y_k\in V, \forall w \in T(V):$
\[
	  \{\{v,y_1\dots y_k\},w\}=\sum\limits_{w=w_1\dots w_{2k+1}}
	  \{v,w_1\{y_1,w_2\}w_3\dots w_{2k-1}\{y_k,w_{2k}\}w_{2k+1}\},
\]
  where the sum on the right runs over all decompositions of the word $w$
  as a concatenation product of (possibly empty) subwords.
\end{itemize}
\end{defi}

Brace algebras were introduced to encode the properties of composition laws 
in operads. In particular (see \cite[Chap.~3]{FoissyOperads} for a proof):

\begin{lemma}
The vector spaces $\SP$ and $\NCP$ spanned by partitions respectively 
noncrossing partitions, carry a brace algebra structure by the following operations: $\forall p\in {\SP}(n),\ p_1,\dots,p_k \in {\SP}$:
$$
	\{p,p_1\dots p_k\}:=\sum\limits_{1\leq i_1<\cdots < i_k\leq n}p \circ (I,\dots, I,p_1, I,\dots, I,p_k, I,\dots, I), 
$$
where on the right-hand side $p_l$ is located in position $i_l$.
\end{lemma}

For example, for  $p=\{\{1,2\}\}$, $p_1=\{\{1\}\}$, $p_2=\{\{1,2,3\}\}$, we get
$$
	\{p,p_1 p_2\}=\{\{2,6\},\{1\},\{3,4,5\}\}+\{\{2,3\},\{1\},\{4,5,6\}\}+\{\{1,3\},\{2\},\{4,5,6\}\}.
$$

The deconcatenation coproduct of a word
$$
	\Delta(v_1\dots v_n):= v_1\dots v_n \otimes \mathbf{1} + \mathbf{1} \otimes v_1\dots v_n 
	+ \sum\limits_{i=1}^{n-1} v_1\dots v_i\otimes v_{i+1}\dots v_n
$$
equips $T(V)$ with the structure of a connected graded coalgebra.

If an associative product $\ast : T(V) \otimes T(V) \to T(V)$ with unit
$\mathbf 1$ equips $(T(V),\ast,\Delta)$ with the structure of a bialgebra
and the restrictions of $\ast$ (on the image) to maps $\{\ ,\ 
\}_{n,m} : V^{\otimes n}\otimes V^{\otimes m} \to V$ are null when $n>1$ and $m
> 0$, then the maps $\{\ ,\ \}_{1,m}: V \otimes V^{\otimes m} \to V$ define a brace
algebra structure on $V$. The converse assertion is also true. Indeed, a
brace algebra structure on $V$ defines uniquely a bialgebra structure on
$T(V)$ with these properties (the two categories are equivalent, as
follows from the fact that $(T(V),\Delta)$ is the cofree (coaugmented 
conilpotent) coassociative coalgebra over $V$ in the category of connected 
graded coalgebras).

\begin{prop}
  The product $\ast$ on $T(\SP )$ (and, by restriction, on $T(\NCP)$) is
  obtained from the brace operations as follows. Given a non-empty sequence
  of set partitions $Q_1,\dots ,Q_n$ written in word notation $W=Q_1\dots
  Q_n$, then for all $P_1,\dots,P_k\in \SP$:
  $$
	P_1 \dots P_k \ast W:=\sum\limits_{W=W_1 \dots W_{2k+1}}
	W_1 \{P_1,W_2\}W_3 \dots W_{2k-1}\{P_k,W_{2k}\}W_{2k+1}.
  $$
\end{prop}

\begin{rem}
The product $\ast$ is a non-commutative shuffle product (also called 
dendriform product), meaning that it splits into two half-shuffle products ($\ast =\prec +\succ$):
\begin{align*}
	P_1\dots P_k\prec W 
	&:=\sum\limits_{W=W_1\dots W_{2k}}\{P_1,W_1\}W_2\dots W_{2k-2}\{P_k,W_{2k-1}\}W_{2k},\\
	P_1\dots P_k\succ W 
	&:=\sum\limits_{W=W_1\dots W_{2k+1} \atop W_1\not
	=\emptyset}W_1\{P_1,W_2\}W_3\dots W_{2k-1}\{P_k,W_{2k}\}W_{2k+1}
\end{align*}
that satisfy the Eilenberg--Mac Lane shuffle relations:
\begin{align}	
	(U\prec V)\prec W&=U\prec (V\ast W) 	\label{shuf1}\\
	U\succ (V\prec W)&=(U\succ V)\prec W 	\label{shuf2}\\
	U\succ (V\succ W)&=(U\ast V)\succ W. 	\label{shuf3}
\end{align}
\end{rem}

These constructions are functorial. In particular, the embedding of $\NCP$
into $\SP$ and its left inverse, the noncrossing closure, induce 
homomorphisms of
brace algebras (and of the associated algebraic structures).


\subsection{Induced bialgebras and Hopf algebras}
\label{ssect:Hopfstruct}

To any (non-symmetric) operad $\mathcal{P}$ (subject to suitable finiteness
conditions), there is associated a canonical bialgebra, whose algebra
structure is the free algebra on the set of all operations. (This is
sometimes called the incidence bialgebra of the operad.) The coproduct of
an operation $R$ is defined as
$$
	\Delta(R) = \sum_{R=P \circ (Q_1,\ldots, Q_k)} P \otimes Q_1 \cdots Q_k.
$$
In plain words, the coproduct is given by summing over all the ways $R$
could have arisen from the composition law, and then putting the receiving
operation $P$ on the left and the monomial of all the operations fed into 
$P$ on the right. This bialgebra is graded (by arity-minus-one), but is never
connected, and hence not Hopf.  Indeed, degree zero contains all monomials
in the operad unit.
But one can obtain a Hopf algebra by passing to the quotient 
defined by dividing out by the coideal generated by $u-\mathbf{1}$ for all
unary operations $u$.  (In the cases of interest here, the only unary 
operation is the empty partition.)

These constructions can be described in various formal ways, for example
following from a general process
described in \cite{FoissyOperads}, applied to the symmetrization of 
$\mathcal{P}$ (the constructions are denoted $\mathbf{D}_\mathcal{P}^*$
and $\mathbf{B}_\mathcal{P}^*$ in that article), or via the so-called
two-sided bar construction, as exploited in~\cite{Kock-Weber}.

There will thus be four bialgebras, of which two are Hopf:
\begin{itemize}
  
  \item For general set partitions:
we shall denote by $\Bgap$ the bialgebra induced by the operad $\SP$,
and by $\Hgap$ its connected quotient.

\item For noncrossing partitions: we shall denote by $\Bgapnc$
the sub-bialgebra of $\Bgap$ induced by the sub-operad $\NCP \subset \SP$,
and by $\Hgapnc$ its connected quotient (a sub Hopf algebra of $\Hgap$).
\end{itemize}
Altogether these four bialgebras fit together by bialgebra homomorphisms like this:
$$
\begin{tikzcd}
	\Bgapnc \ar[r, hookrightarrow] \ar[d, twoheadrightarrow] & \Bgap \ar[d, twoheadrightarrow] \\
	\Hgapnc \ar[r, hookrightarrow] & \Hgap
\end{tikzcd}
$$
The Hopf algebra $\Hgapnc$ associated to $\NCP$ was introduced in \cite{EFP16}.
The arguments given there could be adapted to establish the bialgebra
axioms also for the other three cases.

Our aim here is to describe these bialgebras and Hopf algebras in purely
elementary and combinatorial terms for the particular cases of the
set-partitions operad $\SP$ and the noncrossing partitions operad $\NCP$
(in both cases with the gap-insertion composition law).

\bigskip

  As an algebra, $\Bgap$ is the free associative algebra generated by
  $\latSP_0$: it is spanned linearly by sequences of partitions, some of
  which are possibly empty. The algebra $\Hgap$ is the free associative algebra
  generated by $\latSP$, identified with the quotient of $\Bgap$ by the ideal
  generated by $\emptyset-\mathbf{1}$.

  As an algebra, $\Bgapnc$ is the free associative algebra generated by
  $\latNCP_0$: it is spanned linearly by sequences of noncrossing
  partitions, some of which are possibly empty. The algebra $\Hgapnc$ is the free
  associative algebra generated by $\latNCP$, identified with the quotient
  of $\Bgapnc$ by the ideal generated by $\emptyset-\mathbf{1}$.

The products on $\Bgap$, $\Bgapnc$, $\Hgapnc$ and $\Hgap$ are denoted by $\cdot$. 
In order to describe the coproducts $\Delta_0$ and $\Delta$ of
$\Bgap$ respectively $\Hgap$ (they restrict to the coproducts on the 
bialgebra $\Bgapnc$ and the Hopf algebra $\Hgapnc$ associated to $\NCP$), we need 
some further preliminaries regarding partitions.

\bigskip

Let $P = \{\pi_1,\ldots,\pi_k\}$ be a partition of a linearly ordered set $X$. The
set of blocks of $P$ carries a reflexive relation defined by declaring $\pi \Pto
\rho$ to mean that $\Conv(\pi) \cap \rho \neq \emptyset$. In plain words,
$\pi \Pto \rho$ means that either $\rho$ is nested inside $\pi$ or that the
two blocks cross. In the latter case we have also $\rho \Pto \pi$, which
shows that in general the relation $\Pto$ is not antisymmetric. We also write abusively $\Pto$ for the transitive closure of the relation (which
does not always define a poset). This preorder will play an essential role
for allowing the noncrossing and general partitions to be treated on equal
footing in the following algebraic constructions.

Note immediately that

\begin{lemma}
  A partition $P$ is noncrossing if and only if the associated preorder
  $\Pto$ is actually a poset (i.e.~is an antisymmetric relation).
\end{lemma}

The standard notions of lowerset and upperset for a preorder will be important for $\Pto$, so let us recall:

\begin{defi}[Upperset and lowerset]\label{upperset-lowerset}
An \emph{upperset} of a partition $P$ is a subset $U$ of the
set of blocks such that if $\pi \in U$ and $\pi\Pto \rho$ in $P$ then also
$\rho\in U$. In plain words, if a block is in $U$ then all nested blocks
and all blocks crossing it are also in $U$. 

Similarly, a \emph{lowerset} of
$P$ is a subset $L$ of the set of blocks such that if $\pi \in U$ and
$\sigma\Pto \pi$ in $P$ then also $\sigma\in U$. In plain words, if a
block is in $L$ then all englobing blocks and all blocks crossing it are
also in $L$.
\end{defi}

\begin{rem}
If a partition $P$ is noncrossing, then all its lowersets and uppersets are again noncrossing.
\end{rem}

The following is obvious (and holds in any preorder).
\begin{lemma}\label{UcLc}
  If $U$ is an upperset of $P$ then the complement set of blocks $U^c$ is a lowerset.
  If $L$ is a lowerset of $P$ then the complement set of blocks $L^c$ is an upperset. 
\end{lemma}

\begin{defi}[Cut]
  A \emph{cut} of a partition $P$ is a splitting of the set of blocks into a lowerset $L$
  and an upperset $U$. We write $(L,U) \in \cut(P)$ to express this
  situation.  Note that by Lemma~\ref{UcLc} a cut is completely determined
  by specifying either a lowerset or an upperset.
\end{defi}
The following picture illustrates the notion of cut:
  
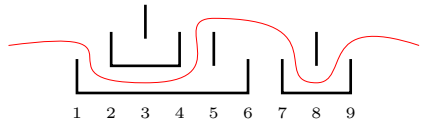
\begin{center}
	\begin{tikzpicture}[scale=0.90]
	  \draw [line width=1.0pt,] (0.0,0.5)--(0.0,0.0)--(2.5,0.0)--(2.5,0.5);
	  \draw [line width=1.0pt,] (3.0,0.5)--(3.0,0.0)--(4.0,0.0)--(4.0,0.5);
	  
	  \draw [line width=1.0pt,] (0.5,0.9)--(0.5,0.4)--(1.5,0.4)--(1.5,0.9);
	  \draw [line width=1.0pt,] (2.0,0.9)--(2.0,0.4);
	  \draw [line width=1.0pt,] (3.5,0.9)--(3.5,0.4);

	  \draw [line width=1.0pt,] (1.0,1.3)--(1.0,0.8);
	    
	  \draw [line width=0.3pt, color=red] 
		(-1.0,0.7) .. controls (1.2,1.0) and (-0.7,0.15) ..
		(1.0,0.15) .. controls (2.4,0.15) and (1.3,1.1) ..
		(2.0,1.1) .. controls (3.8,1.1) and (2.8,0.15) ..
		(3.5,0.15) .. controls (4.0,0.15) and (3.5,1.2) ..
		(5.0,0.7);
			 			 
	  \node at (0.0,-0.3) {\tiny $1$};
	  \node at (0.5,-0.3) {\tiny $2$};
	  \node at (1.0,-0.3) {\tiny $3$};
	  \node at (1.5,-0.3) {\tiny $4$};
	  \node at (2.0,-0.3) {\tiny $5$};
	  \node at (2.5,-0.3) {\tiny $6$};
	  \node at (3.0,-0.3) {\tiny $7$};
	  \node at (3.5,-0.3) {\tiny $8$};
	  \node at (4.0,-0.3) {\tiny $9$};

	  \end{tikzpicture}
  \end{center}

\begin{rem}
\label{cor:fundam1}
  If $P$ is a partition with a cut $(L,U)$, and if $\pi$ and $\pi'$
  are two crossing blocks, then either they both belong to $L$ or they
  both belong to $U$.
\end{rem}

\begin{cor}
\label{cor:samecuts}
If two partitions have the same noncrossing closure (as in \ref{nc}), then there is a natural bijection between their sets of cuts.
\end{cor}

The following more refined version of the complement $L^c$ is specific to partitions, and is the
key to giving a direct combinatorial description of the coproduct $\Delta_0$.

\begin{lemma}
\label{Ui}
A lowerset $L$ of $P$ defines canonically a list of uppersets $U_0,\ldots,U_k$ whose union is $L^c$.  (Here $k$ is the cardinality of the underlying set of $\cup_{\pi\in L} \pi$.)
\end{lemma}

For the picture above, the resulting list of uppersets is
$$
	\emptyset \ \cdot 
  \raisebox{-15pt}{
	\begin{tikzpicture}[scale=0.70]
	  \draw [line width=1.0pt,] (0.0,0.5)--(0.0,0.0)--(1.0,0.0)--(1.0,0.5);
	  \draw [line width=1.0pt,] (0.5,0.9)--(0.5,0.4);
	  \node at (0.0,-0.3) {\tiny $2$};
	  \node at (0.5,-0.3) {\tiny $3$};
	  \node at (1.0,-0.3) {\tiny $4$};
	\end{tikzpicture}
	}
	\cdot\
		\emptyset \ \cdot \
	\emptyset \ \cdot 
  \raisebox{-15pt}{
	\begin{tikzpicture}[scale=0.70]
	  \draw [line width=1.0pt,] (0.0,0.5)--(0.0,0.0);
	  \node at (0.0,-0.3) {\tiny $8$};
	\end{tikzpicture}
	}
\cdot
		\ \emptyset 
$$
(with monomial ``dot'' notation, as will be used later).

\begin{proof}
  Let $\{x_1,\ldots,x_k\}$ denote the underlying set of the lowerset $L$.
  Since the underlying set $[n]$ of $P$ is linearly ordered, these $k$ points
  define $k+1$ (possibly empty) intervals in $[n]$, denoted $D_0,\ldots,D_k$, where
  $D_0 := \{ y \in [n] \mid y < x_1 \}$ and $D_k := \{ y \in [n] \mid
  y > x_k\}$. The intervening 
  intervals are $D_i := \{ y \in [n] \mid x_i < y < x_{i+1} \}$.
  The promised uppersets are simply the restrictions $U_i := P_{|D_i}$, for $i=0,\ldots,k$.
  To see that this is meaningful, note first that
  each $D_i$ is a union of blocks of $P$, because the lowerset condition on 
  $L$ prevents the blocks in the complement from straddling any of the 
  points $x_i$. Second, since we are inside the complement  
  $U := L^c$, we have $U_i = P_{|D_i} = U_{|D_i}$, and the 
  restriction of an upperset is an upperset.
\end{proof}

\begin{rem}
  Note that the $U_i$ are either empty or are the convex 
  components of the underlying set of $U$ in the sense of Definition \ref{conv}.
\end{rem}

\begin{defi}[Gap monomial]\label{gap-monomial}
  For $U$ an upperset, we define the 
  \emph{gap monomial}
  to be the monomial in $\Bgap$ given by
  $$
  	\mst{U} := \st{U_0} \cdots \st{U_k}
  $$
  Here, $U_0,\ldots,U_k$ is the list of uppersets defined by the lowerset
  $U^c$ as in Lemma~\ref{Ui}.
  
  We also define the \emph{reduced gap monomial} $\rmst{U}$
  to be the class of $\mst{U}$ in $\Hgap$, that is, the monomial
  obtained from $\mst{U}$ by omitting those 
  entries in the list that are empty partitions.
  This can also be characterized as the monomial of
  the convex components of the underlying set of $U$.
\end{defi}

With the concepts and notations just introduced, we can now 
give the following interpretation of the operadic composition law:
\begin{quote}
  $P \diamond (Q_0,\ldots,Q_k) = R$
  \\[4pt]
  if and only if
  \\[4pt]
  $P$ is a lowerset of $R$ with complement list of uppersets 
  $Q_0,\ldots,Q_k$
\end{quote}
This leads to a description of the coproduct in explicit combinatorial terms:

\begin{prop}
  The coproduct of the incidence bialgebra $\Bgap$ of the gap-insertion operad is 
  given by
  $$ 
  	\Delta_0(P) = \sum_{(L,U) \in \cut(P)} \st{L} \otimes \mst{U}.
  $$
\end{prop}

\noindent
Here the right-hand tensor factor is  the gap monomial of 
Definition~\ref{gap-monomial}.

\begin{cor}
\label{coro:subHopfalg}
  The partitions $J_n:=\{\{1\},\ldots,\{n\}\} \in \latSP(n,n)$ generate
  the sub-bialgebra $\Bgap^|$ in $\Bgap$.
\end{cor}
We refer to the finest partitions $J_n$ colloquially as ``forests of sticks''.

\begin{proof}
  This is clear, as the lowerset and upperset corresponding to a cut in a
  partition $J_n$ are again such a partition, respectively a monomial of
  such partitions. Hence, the coproduct can be written
$$
	\Delta_0(J_n)=\sum_{L \subseteq [n]} 
	\st{J_{\sharp  L}} \otimes \st{J_{n_0}} \cdot\  \cdots \cdot\ \st{J_{n_{\sharp  L}}}. 
$$
Here $J_0$ is identified with the empty set. As an illustration:
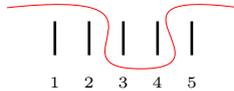
\begin{center}
 \begin{tikzpicture}[scale=0.90]
	\draw [line width=1.0pt,] (0.0,0.5)--(0.0,0.0);
	\draw [line width=1.0pt,] (0.5,0.5)--(0.5,0.0);
	\draw [line width=1.0pt,] (1.0,0.5)--(1.0,0.0);
	\draw [line width=1.0pt,] (1.5,0.5)--(1.5,0.0);
	\draw [line width=1.0pt,] (2.0,0.5)--(2.0,0.0);
	\draw [line width=0.3pt, color=red]
	  (-0.7,0.7) .. controls (1.25+0.8,1.0) and (1.25-1.4,-0.2) ..
	  (1.25,-0.2) .. controls (1.25+1.4,-0.2) and (1.25-0.8,1.0) ..
	  (2.7,0.7);	
	\node at (0.0,-0.4) {\tiny $1$};
	\node at (0.5,-0.4) {\tiny $2$};
	\node at (1.0,-0.4) {\tiny $3$};
	\node at (1.5,-0.4) {\tiny $4$};
	\node at (2.0,-0.4) {\tiny $5$};
 \end{tikzpicture}
\end{center}
corresponds to the tensor product  
$$
 \raisebox{-12pt}{
	\begin{tikzpicture}[scale=0.70]
	  \draw [line width=1.0pt,] (0.0,0.5)--(0.0,0.0);
	  \draw [line width=1.0pt,] (0.5,0.5)--(0.5,0.0);
	  \draw [line width=1.0pt,] (1.3,0.5)--(1.3,0.0);
	  \node at (0.0,-0.3) {\tiny $1$};
	  \node at (0.5,-0.3) {\tiny $2$};
	  \node at (1.3,-0.3) {\tiny $5$};
	\end{tikzpicture}
 }
 \quad
 \bigotimes
 \quad
 \ \emptyset \ \cdot \ \emptyset \ \cdot
 \raisebox{-12pt}{
	\begin{tikzpicture}[scale=0.70]
	  \draw [line width=1.0pt,] (0.0,0.5)--(0.0,0.0);
	  \draw [line width=1.0pt,] (0.5,0.5)--(0.5,0.0);
	  \node at (0.0,-0.3) {\tiny $3$};
	  \node at (0.5,-0.3) {\tiny $4$};
	\end{tikzpicture}
 }
 \cdot \ \emptyset
$$
in the coproduct $\Delta_0(J_5)$.
\end{proof}

Recall that the Hopf algebra $\Hgap$ is obtained from $\Bgap$ by 
quotiening by 
the ideal generated by $\emptyset - \mathbf{1}$. 

\begin{prop}
  The coproduct of the reduced incidence Hopf algebra $\Hgap$ of the gap-insertion operad is 
  given (for $P$ a nonempty partition) by:
  $$
  	\Delta(P) = \sum_{(L,U) \in \cut(P)} \st{L} \otimes \rmst{U}.
  $$
  The forest of sticks, $\Hgap^|$, form a sub Hopf algebra in $\Hgap$. 
\end{prop}

\noindent
Here the right-hand tensor factor $\rmst{U}$ is the reduced gap monomial of
Definition \ref{gap-monomial}. The exact same formulae hold for the bialgebras $\Bgapnc$ and 
$\Hgapnc$ of noncrossing partitions. The sub bi- and Hopf algebras of forests 
of sticks are denoted $\Bgapnc^|$ respectively $\Hgapnc^|$. 

\begin{rem}
  With the cut formulation of the coproduct,
  a tight analogy with the Butcher--Connes--Kreimer Hopf
  algebra of rooted trees becomes clear: the way the upperset
  is split into a monomial is analogous to the way the crown of
  a tree with a cut is interpreted as a forest. In fact this is
  more than just an analogy: the underlying poset of a
  noncrossing partition is actually a forest (in the sense that
  for any block $\pi$, the set $\{\sigma \mid \sigma \Pto \pi\}$
  is a linear order), and the notion of cut is the same as that
  for forests. In fact:
\end{rem}

\begin{prop}
  The assignment sending a noncrossing
  partition to its underlying forest defines a bialgebra
  homomorphism from $\Hgapnc$ to the Butcher--Connes--Kreimer
  Hopf algebra.
\end{prop}

\noindent
In fact, this bialgebra homomorphism factors through the
  non-commutative Hopf algebra of planar forests of \cite{Foissy:trees}. 

\bigskip

These four bialgebras are $\N^2$-graded: the bidegree of a partition $P$ 
of $[n]$ with $k$ blocks is $(k,n)$. (Note that the arity of a degree-$n$ partition is $n+1$, but that the bialgebra construction lowers the degree 
by $1$.) Note that the bialgebras $\Bgap$ and $\Bgapnc$ are not connected: the degree zero component is spanned by all the monomials in $\emptyset$. The coarsest partitions are `skew-primitive', meaning for example
$$
	\Delta_0 ( \ncthreeW ) = 
	\ncthreeW \otimes \emptyset{\cdot}\emptyset{\cdot}\emptyset{\cdot}\emptyset
	\ + \ \emptyset \otimes \ncthreeW .
$$
On the other hand, $\Hgap$ and $\Hgapnc$ are connected by construction (and therefore
automatically Hopf algebras). For each of the four bialgebras, the counit is given by $\varepsilon(P)=0$ 
for any nonempty partition (or noncrossing partition).

\bigskip

Applications in free probability of the theory of partitions and noncrossing partitions suggest to introduce another Hopf algebra map from $\Hgapnc$ to $\Hgap$ than the obvious embedding.

\begin{prop}\label{prop:ncstar}
  The linear map
  \begin{align*}
	\nc^\ast : 	&\left\{\begin{array}{rcl}
	\NCP 	&\longrightarrow	& \SP \\
	P		&\longmapsto		& \sum\limits_{\nc(P')=P} P'
			   \end{array}\right.
\end{align*}
  defines a bialgebra homomorphism from $\Bgapnc$ to $\Bgap$ (and from $\Hgapnc$ to 
  $\Hgap$).
\end{prop}

\begin{proof}
  This follows since partitions with common noncrossing closure have 
  isomorphic sets of cuts, cf.~Corollary~\ref{cor:samecuts}. 
\end{proof}


\subsection{Unshuffling coproducts on partitions}
\label{ssec:unshuffle}

A (non-commutative) monomial of partitions will also be called a
\emph{multipartition} (resp.~\emph{noncrossing multipartition}): these are
the elements $P=P_1\cdots P_k\in \Hgap$ (resp.~in $\Hgapnc$), where for
each $i \in [k]$, $P_i$ is a partition (resp.~a noncrossing partition) in
$\latSP(n_i)$ (resp.~in $\latNCP(n_i))$, $n_i\not= 0$.

The set of multipartitions (resp.~noncrossing multipartitions) is denoted
by $\latSMP$ (resp.~$\latNMP$). These elements form a linear basis of
$\Hgap$ (resp.~$\Hgapnc$). By multiplicativity, we define bigradings
$\latSMP(k,n)$ ($\latNMP(k,n)$) for all $k,n\geq 0$, where $k$ stands for
the total number of blocks. The notions of upperset, lowerset, and cut
defined in  Subsection~\ref{ssect:Hopfstruct}, 
extend multiplicatively to multipartitions in a
straightforward manner. 
For example, a lowerset of a multipartition
$P=P_1\cdots P_k$ is a sequence of lowersets $L_i$ of each $P_i$, written
as a monomial $L_1\cdots L_k$.

Moreover, shifting the elements of $P_2$ by $\deg(P_1)$, $\ldots$, the
elements of $P_k$ by $\deg(P_1)+\cdots+\deg(P_{k-1})$, such a $P$ can be
seen as a family of sets of subsets of $[n_1+\cdots+n_k]=[n]$, these
subsets forming a partition of $[n]$. With these conventions and this
identification, which will be used systematically in this section, for any
$P\in \latSMP$ or $P \in \latNMP$ as above,
the coproducts can be written
\begin{align}
	\Delta_0(P)&=\sum_{(L,U) \in \cut(P)} L \otimes \mst{U} \nonumber
	\\
	\Delta(P)&=\sum_{(L,U) \in \cut(P)} L \otimes \rmst{U}. \label{thecoprod}
\end{align}
Here $\Delta_0$ and $\Delta$ are extended multiplicatively 
from single partitions to multipartitions (monomials in partitions)
in the usual way.
As a result, since $P$ is a multipartition (i.e.~a monomial),
also $L$ is a monomial.
In the right-hand tensor factors, $U$ is itself a monomial for the 
same reason, and $\mst{U}$ (resp.~$\rmst{U}$) are furthermore 
``monomials of monomials'', 
namely the product of the
gap monomials (resp.~reduced gap monomials), as in
Definition \ref{gap-monomial}. 

In the following we work only with $\Delta$.

\begin{defi}[Unshuffle structure]
\label{def:shuffle}
  For any non-empty $P\in \latSMP$ we put:
\begin{align*}
	\Delta_\prec(P)&=\sum_{(L,U) \in \cut(P) \atop 1\in L} 
				L \otimes \rmst{U}, \\
	\Delta_\succ(P)&=\sum_{(L,U) \in \cut(P) \atop 1\in U} 
				 L \otimes \rmst{U}.
\end{align*}
Here the $L$ are the lowerset monomials of the multipartition $P$,
whereas the $\rmst{U}$ are the reduced gap monomials of each cut, as in 
Definition~\ref{gap-monomial}.

These definitions restrict to (non-empty) noncrossing partitions.  
\end{defi}

\begin{rem}
  It follows from Remark~\ref{cor:fundam1} that the multiplicative extension to
  a map from $\latNCP$ to $ \latSMP$ of the map $\nc^*$ from
  Proposition~\ref{prop:ncstar} commutes with these three operations.
\end{rem}

\begin{prop}[Unshuffle Hopf algebra structures]\label{prop:halfshuffles}
  The two coproducts just defined, on the augmentation ideals $\Hgap_+ \subset \Hgap$
  and $\Hgapnc_+ \subset \Hgapnc$, turn $\Hgap$ and $\Hgapnc$ into unshuffle Hopf algebras (also 
  called codendriform Hopf algebras~\cite{Foissy}).  
\end{prop}

In detail, these coproducts satisfy the dual of the shuffle relations \eqref{shuf1}--\eqref{shuf3}:
\begin{align}
	(\Delta_\prec \otimes \Id)\circ \Delta_\prec&=(\Id \otimes 
	\Delta)\circ \Delta_\prec, \label{Dprec} \\
	(\Delta_\succ \otimes \Id)\circ \Delta_\prec&=(\Id \otimes 
	\Delta_\prec)\circ \Delta_\succ, \notag\\
	(\Delta \otimes \Id)\circ \Delta_\succ&=(\Id \otimes 
	\Delta_\succ)\circ \Delta_\succ. \notag
\end{align}
Moreover, for any $x,y\in \Hgap_+$, introducing Sweedler-like notation for 
these three coproducts: 
\begin{align*}
	\Delta_\prec(x) &=x\otimes \mathbf{1} + x'_\prec\otimes x''_\prec\\
	\Delta_\succ(x) &= \mathbf{1}\otimes x+x'_\succ \otimes x''_\succ\\
	\Delta(y)	     &=y\otimes  \mathbf{1} +  \mathbf{1}\otimes y+y'\otimes y'',
\end{align*}
we have
\begin{align*}
	\Delta_\prec(x\cdot y)&=x\cdot y\otimes  \mathbf{1}+x\otimes y+x\cdot y'\otimes y''
				+x'_\prec\cdot y\otimes x''_\prec+x'_\prec\otimes x''_\prec\cdot y
				+x'_\prec\cdot y'\otimes x''_\prec\cdot y'',\\
	\Delta_\succ(x\cdot y)&= \mathbf{1}\otimes x\cdot y+y\otimes x+y'\otimes x\cdot y''
				+x'_\succ\cdot y\otimes x''_\succ+x'_\succ\otimes x''_\succ\cdot y
				+x'_\succ\cdot y'\otimes x''_\succ\cdot y''.
\end{align*}

We have used Sweedler-type notations for the reduced coproducts, i.e., $\Delta'_\prec(x):=x'_\prec\otimes x''_\prec$, $\Delta'_\succ(x):=x'_\succ \otimes x''_\succ$ and $\Delta'(x):=x' \otimes x''$. 

\medskip

\begin{proof} 
  We briefly comment on the coproduct \eqref{thecoprod} and its splitting,
  $\Delta=\Delta_\prec + \Delta_\succ$, satisfying the shuffle relations.
  For more details the reader is referred to \cite{EFP16}. Starting from $P
  \in \latSMP(k,n)$, with $k \geq 1$ each cut $(L,U) \in \cut(P)$
  in the definition of the
  coproduct \eqref{thecoprod} corresponds to a decomposition of $P$ into
  a lowerset $L$ and an upperset $U$, which are the complements of each 
  other. Recall that 
  $U$ and $L$ are partitions themselves.
  We shall need \emph{compatible pairs of cuts}, writing $(L,M,U)  \in
  \cut_2(P)$ for the situation where $L$ is a lowerset of $P$ and $U$
  is an upperset of $P$, with complements $L^c = M \sqcup U$ and $U^c = L \sqcup M$.
  (Coassociativity amounts to saying that the last two conditions are equivalent.) 
  The shuffle identities are then checked by keeping track of the first element $1 \in P$:
\begin{align*}
	(\Delta_\prec \otimes \Id)\circ \Delta_\prec(P)&
	=\sum_{(U,M,L) \in \cut_2(P) \atop 1 \in L} L \otimes \rmst{M}\otimes \rmst{U}
	= \ (\Id \otimes \Delta)\circ \Delta_\prec(P),\\
	(\Delta_\succ \otimes \Id)\circ \Delta_\prec(P)&
	=\sum_{(U,M,L) \in \cut_2(P) \atop 1 \in M} L \otimes \rmst{M}\otimes \rmst{U}
	= \ (\Id \otimes \Delta_\prec)\circ \Delta_\succ(P),\\
	(\Delta \otimes \Id)\circ \Delta_\succ(P)&
	=\sum_{(U,M,L) \in \cut_2(P) \atop 1 \in U} L \otimes \rmst{M}\otimes \rmst{U}
	= \ (\Id \otimes \Delta_\succ)\circ \Delta_\succ(P).
\end{align*}
(Here $L$ is a monomial just because $P$ is, whereas $\rmst{M}$ and 
$\rmst{U}$ are reduced gap monomials, as in Definition~\ref{gap-monomial}.)
The compatibility between $\Delta_\prec$, $\Delta_\succ$ and the product $\cdot$ follows by the multiplicative extension of $\Delta_\prec$, $\Delta_\succ$ implied by Definition \ref{def:shuffle}.
\end{proof}
Notice that the three coproducts $\Delta,\Delta_\prec,\Delta_\succ$ dualize respectively to the products denoted $\ast,\prec,\succ$ on $\Hgap_+^\ast$ and $\Hgapnc_+^\ast$. It follows from Proposition~\ref{prop:halfshuffles} that the latter satisfy the shuffle identities \eqref{shuf1}--\eqref{shuf3}.


\subsection{Moments and cumulants}
\label{ssect:momcum}

In this subsection we briefly revisit the relations between classical
cumulants and free cumulants in the context of free probability, a point
usually addressed using the properties of the inclusion of the lattice
of noncrossing partitions into the one of partitions \cite{Lehner}. We use
freely the results of \cite{EFP16} and restrict the study to the case of a
single free random variable; this allows us to consider solely computations
with formal power series. The multivariate case can be addressed with
exactly the same tools using the techniques in \cite{EFP16}.

Recall that a classical probability space is a pair $(A,\varphi)$ where $A$
is a commutative ring of random variables and $\varphi$ a unital linear
form on it called the expectation. The moments of a random variable $a \in A$ are
defined as $m_n:=\varphi(a^n),\ n\in\N^\ast$. The moment-generating
function is the associated exponential series $E(z):=1+\sum_{n\geq
1}m_n\frac{z^n}{n!}$, and the classical cumulants $\{c_n\}_{n\geq 1}$ are
defined as the coefficients of the cumulant-generating function
$C(z)=\sum_{n\geq 1}c_n\frac{z^n}{n!}$ given by the formulae
$$
	C(z) := \log (E(z)), \qquad E(z)=\exp(C(z)).
$$

In free probability \cite{SpeicherNica}, the ring $A$ is not necessarily commutative, and an
element $a\in A$ is interpreted as a non-commutative random variable.
The moments are defined as in the classical case, but the \emph{free cumulants}
$\{k_n\}_{n\geq 1}$
are defined instead using the {\em ordinary}
generating series 
$$
	M(z):=1+\sum_{n\geq 1}m_nz^n \quad \text{ and } \quad K(z):=\sum_{n\geq 1}k_n z^n ,
$$
related through the fixpoint formula
$$
	M(z) = 1+ K(zM(z)).
$$

For various reasons (see for example \cite{SpeicherNica}), it is fruitful
to lift these relations to the framework of set partitions. As far as
classical cumulants are concerned, the moment-cumulant relations translate
into
\begin{equation}
\label{clascumu}
	m_n=\sum\limits_{P\in \latSP(n)}c_{P},
\end{equation}
where, for $P:=\{\pi_1,\dots, \pi_k\}\in \latSP(n)$,
$$
	c_{P}:=\prod\limits_{i\leq k}c_{\sharp \pi_i}.
$$
For free cumulants, Speicher showed that one has instead
\begin{equation}
\label{freecumu}
	m_n=\sum\limits_{Q\in \latNCP(n)}k_{Q},
\end{equation}
where, for $Q=\{\tau_1,\dots, \tau_k\}\in \latNCP(n)$,
$$
	k_{Q}:=\prod\limits_{i\leq k}k_{\sharp \tau_i}.
$$
From the properties of the embedding $\latNCP \subset \latSP$, one can 
derive (see \cite{Lehner}) that
$$
	k_Q=\sum\limits_{\nc(P)=Q}c_P.
$$
In other words, if $k$ and $c$ are extended to linear functions on
$\latNCP$, respectively $\latSP$, we arrive at the following
fundamental relation between free and classical cumulants

\begin{equation}\label{Lehfla}
	k=c\circ \nc^\ast,
\end{equation}
where $\nc^\ast$ was defined in Proposition \ref{prop:ncstar}.

As mentioned in the introduction, the papers \cite{EFP15,EFP18,EFP17} 
proposed an approach to the relations between moments and free cumulants, which is
based on a non-commutative shuffle bialgebra structure \cite{Foissy} on the
dual of a particular connected, graded, non-commutative, non-cocommutative
Hopf algebra $\mathbf{H}$ of words constructed as the double tensor algebra
from $(A,\varphi)$. The group of characters and the corresponding Lie
algebra of infinitesimal characters over this Hopf algebra can then be
shown to be related by three naturally defined exponential-type maps and
the corresponding logarithms. The unital linear map $\varphi$, which is
part of the data of a non-commutative probability space $(A,\varphi)$,
induces a particular Hopf algebra character on $\mathbf{H}$. One can show
that the three logarithms applied to this character encode the three
families of monotone, free, and boolean cumulants as infinitesimal Hopf
algebra characters. This algebraic setting allows to recover many of the
results and formulas in Speicher's approach, although the lattice of
noncrossing partitions plays a rather marginal role. Indeed, noncrossing
partitions appear only after evaluating the Hopf algebra character
corresponding to $\varphi$ on elements of $\mathbf{H}$.

In \cite{EFP16} these results were transferred to the Hopf algebra
$\Hgapnc$ defined on noncrossing partitions in the following way. Define the 
linear form $\kappa$ on $\Hgapnc$ to be zero on all noncrossing
multipartitions in $\latNMP(k,n)$ for $k>1$ and by $\kappa([n]):=k_n$ else.
Solving the half-shuffle fixpoint equation in the graded dual
$\Hgapnc^\ast$:
\begin{equation}
\label{fixedpt}
	\phi=\varepsilon_N+\kappa\prec\phi,
\end{equation}
where $\varepsilon_N$ is the counit of $\Hgapnc$ (it is 
essentially the Kronecker delta $\varepsilon_N(P):=\delta_{P,\emptyset}$), one
obtains a character which evaluates any noncrossing partition $Q$ to
$\phi(Q)=k_Q$, i.e., it maps $Q$ to the product of cumulants for each block
in $Q$. Therefore, the $n$th moment is obtained as $m_n=\sum_{Q\in
\latNCP(n)}\phi(Q)$. We shall see in Proposition~\ref{prop:doubletensor} that \eqref{fixedpt} 
and the sum formula for $m_n$ find a more natural expression
in terms of the bialgebra interaction we introduce in this paper.
We can thus interpret Equation~\eqref{fixedpt} as
an algebraic lift of (\ref{freecumu}).
 
Define now a linear form $\gamma$ on $\Hgap$ by $\gamma(Q):=c_Q$ if $Q$ is a
partition of $[n]$ such that $\nc(Q)=[n]$, and to be zero on all other
partitions and multipartitions. Since
$$
	k_n=\sum_{P\in\latSP(n) \atop \nc(P)=[n]}c_P,
$$ 
we obtain that
\begin{equation*}
	\kappa=\gamma\circ \nc^\ast .
\end{equation*}

\begin{prop}
  With $\gamma$ defined as above, the solution $\psi$ to the half-shuffle
  fixpoint equation in the graded dual $\Hgap^\ast$
  \begin{equation}
  \label{newcmfla}
	  \psi=1+\gamma\prec\psi
  \end{equation}
  is such that 
  \begin{equation}
  \label{newcmfla2}
	  \phi = \psi \circ \nc^\ast.
  \end{equation}
\end{prop}

  In particular, $m_n=\sum_{P\in \latSP(n)}\psi(P)$ and
  equations~(\ref{newcmfla}) and (\ref{newcmfla2}) can be interpreted as
  algebraic and operadic lifts of the classical moment-cumulant formula,
  respectively formula (\ref{Lehfla}).

\begin{proof} 
  The proposition follows from the identity $\kappa=\gamma\circ \nc^\ast$
  and the fact that $\nc^\ast$ commutes with $\Delta_\prec$
  defined in Proposition~\ref{prop:halfshuffles}.
\end{proof}


\section{Block-substitution operad and associated structures}
\label{2ndOperad}

We introduce now a second family of operads and algebraic structures on
partitions and noncrossing partitions. We shall see that they are closely
related to the lattice properties of the two families of partitions.

\smallskip

Before presenting the block-substitution operad, we briefly recall that a
coloured symmetric set operad $\mathcal{P}$ is given by:
\begin{itemize}
  
  \item A set of \emph{colours} $C$;
  
  \item For each $n\in \N$, and each $(n+1)$-tuple of colours 
  $(c_1,\ldots,c_n;c) \in C^{n+1}$, a set of \emph{$n$-ary operations}
  $\mathcal{P}(c_1,\ldots,c_n; c)$; 
  
  \item For each colour $c\in C$, an \emph{identity operation} $\id_c \in 
  \mathcal{P}(c;c)$;
  
  \item An operadic \emph{composition law}
  $$
  \begin{tikzcd}
	\mathcal{P}(c_1,\ldots,c_k; c)
	\ \times \
    \mathcal{P}(d_{1,1},\ldots,d_{1,i_1}; c_1) \times \cdots \times
	\mathcal{P}(d_{k,1},\ldots,d_{k,i_k}; c_k) 
	\ar[d] \\
	\mathcal{P}(d_{1,1},\ldots,d_{1,i_1},\ldots,d_{k,1},\ldots,d_{k,i_k}; c) 
  \end{tikzcd}
  $$
  denoted $(p,(q_1,\ldots,q_k)) \longmapsto p \circ (q_1,\ldots,q_k)$,
 
  \item For each $n\in \N$, for each 
  $(c_1,\ldots,c_n;c) \in C^{n+1}$, and for each permutation $\sigma\in 
  \mathfrak{S}_n$, a map
  $$
  \mathcal{P}(c_1,\ldots,c_n; c) \to  
  \mathcal{P}(c_{\sigma(1)},\ldots,c_{\sigma(n)}; c)
  $$
  denoted $P \mapsto P^\sigma$.
\end{itemize}

This data is subject to axioms completely analogous to the ordinary operad
axioms: there is an associative axiom, a unit axiom (for each identity
operation), and a symmetry axiom. It is simply a typed version of the
notion of operad, where the composition law requires strict type-checking
in terms of colours: an operation can be substituted into an input slot of
another operation if and only if its output colour matches the colour of
the receiving input slot.

A coloured operad $\mathcal{P}$ is \emph{reduced} if it has no nullary 
operations.

\begin{ex}
  A small category is the same thing as a coloured operad with only unary 
  operations. The objects are then the colours, and the arrows are the 
  operations (with domain as input colour and codomain as output colour).
\end{ex}


\subsection{A coloured symmetric operad of compositions}
\label{ssec:coloured}

Central objects in this work are the block-substitution operads $\latSC$ of set
partitions and $\latNCC$ of noncrossing partitions, which we define in this
subsection. They are both coloured symmetric operads with colour set
$\N^\ast$. We concentrate here on the ordinary partitions, since the
noncrossing condition is orthogonal to the constructions, 
cf.~Proposition~\ref{prop:opd-nc} below. 

The idea is simple: a (non-empty) partition $P = \{\pi_1,\ldots,\pi_k\}$
with $k$ blocks is considered a $k$-ary operation. Its output colour is the
number of elements of the underlying set $[n]$. Each block $\pi_i$ is
considered to be an input slot of arity $n_i = \sharp \pi_i$, and into it
one can substitute any partition $Q_i$ of $[n_i]$ by replacing the block
with the partition $Q_i$ (under the unique order-preserving bijection
between $[n_i]$ and the underlying set of $\pi_i$). The effect of the
substitution is thus to refine the partition, while leaving fixed the total
number of elements $n$. The identity operations are clearly the single-block
partitions $I_n$, since substituting such a partition into a block of size $n$
does not refine it further.

To actually implement this idea and make it fit the formal definition,
it is necessary to number the blocks, so as to know in which order the input 
slots come, and where to substitute given inputs. For this reason,
the operations will have to be set compositions rather than set partitions,
and the actions of the symmetric groups required for symmetric operads
will then simply be renumbering of the blocks. This is a standard 
procedure in the construction of symmetric operads.

\begin{defi}
  A \emph{set composition} is a set partition equipped with a numbering of
  its blocks.  Equivalently, it is given by a list of blocks 
  $P=(\pi_1,\ldots\pi_k)$ rather than a set of blocks.
We denote by $\latSC(k,n)$ the set of iso-classes of
compositions of an $n$-element set into $k$ blocks, that is, sequences
$C=(\pi_1,\ldots,\pi_k)$ such that $P=\{\pi_1,\ldots,\pi_k\}$ is a
partition. This includes the case $\latSC(0,0) = \{\emptyset\}$, consisting
only of the empty composition.
We also put:
\begin{align*}
	\latSC&=\bigsqcup_{1\leq k\leq n} \latSC(k,n),
	&   \latSC(n)&=\bigsqcup_{k\leq n} \latSC(k,n),
		&   \latSC_0&=\latSC(0,0) \sqcup \latSC.
\end{align*}
  
  Similarly, a \emph{noncrossing composition} is
  a noncrossing partition with a numbering of its blocks.  
We denote by $\latNCC(k,n)$ the set of iso-classes  of noncrossing
compositions of an $n$-element set into $k$ blocks, that is, sequences
$C=(\pi_1,\ldots,\pi_k)$ such that $P=\{\pi_1,\ldots,\pi_k\}$ is a
noncrossing partition.
We put:
\begin{align*}
	\latNCC&=\bigsqcup_{1\leq k\leq n} \latNCC(k,n),
	&   \latNCC(n)&=\bigsqcup_{k\leq n} \latNCC(k,n),
		&   \latNCC_0&=\latNCC(0,0) \sqcup \latNCC.
\end{align*}

\end{defi}

We can now specify the data of the operad $\SC$ of set partitions (which 
should more precisely be called of set compositions):

\begin{itemize}
  
  \item
  The set of colours is the set of positive integers $\N^\ast$;
  
  \item For fixed $n_1,\ldots,n_k,n \in \N^\ast$, we denote by 
  $\SC(n_1,\ldots,n_k;n)$ the set of set 
  compositions of $[n]$, say
  $C=(\pi_1,\ldots,\pi_k)$, such that $\sharp
  \pi_i=n_i$ for each $i\in [k]$.  (Note that if $n\neq n_1+\cdots+n_k$,
  then $\SC(n_1,\ldots,n_k;n)=\emptyset$.)
  
  A given composition $P=(\pi_1,\ldots,\pi_k)$ of $[n]$ is thus regarded as a 
  $k$-ary operation of output colour $n$, and input colour-list $(\sharp 
  \pi_1,\ldots,\sharp \pi_k)$.
  
  \item The identity operations are by definition the coarsest 
  compositions, i.e., single blocks of $n$ elements, $I_n \in \SC(n;n)$,
  for each $n \geq 1$.
  
  \item The actions of the symmetric groups $\mathfrak{S}_k$ are given
  by permutation of blocks:
 for $P=(\pi_1,\ldots,\pi_k) \in \SC(n_1,\ldots,n_k;n)$ and $\sigma \in 
 \mathfrak{S}_k$:
  \begin{align*}
	P^\sigma&=(\pi_{\sigma(1)},\ldots,\pi_{\sigma(k)})
	\in
	\latSC(n_{\sigma(1)},\ldots,n_{\sigma(k)} ; n) .
  \end{align*}
	(That is, $\sigma$ acts by renumbering blocks.)

\item 
  The composition law on $\SC$ is given as follows.  
  For any 
  $P=(\pi_1,\ldots,\pi_k) \in \SC(n_1,\ldots,n_k;n)$ 
  and a list of compositions $(Q_1,\ldots,Q_k)$  with 
  $Q_i=(\tau_{i,1},\ldots,\tau_{i,l_i})\in 
 \SC( m_{i,1},\ldots,m_{i,l_i} ; n_i)$, we define:
\begin{align*}
  P \circ (Q_1,\ldots,Q_k) &:= 
  (\tau_{1,1},\ldots, \tau_{1,l_1},\ \ldots \ , \ 
  \tau_{k,1},\ldots, \tau_{k,l_k} ).
  \end{align*}
  This requires reindexing: implicitly we are transporting the set-composition
  structure of each $Q_i$ on $[n_i]$ along the unique monotone bijection
  $[n_i] \isopil \pi_i \subset [n]$.

\end{itemize}
Checking that these data constitute indeed a coloured symmetric operad $\SC$
is a straightforward, but rather cumbersome, routine exercise. The only 
difficulty is index bookkeeping.

\begin{prop}
\label{prop:opd-nc}
  The composition law preserves noncrossing compositions. In particular,
  the same definitions as above, but using only noncrossing compositions,
  defines a (coloured symmetric) operad $\NCC$ of noncrossing compositions.
\end{prop}

\begin{proof}
  For any application of the composition law, say $P \circ
  (Q_1,\ldots,Q_k)$, assume that all of $P$ and $Q_i$ are noncrossing. We
  assume the notation from above. Consider two blocks $\tau$ and $\tau'$ of
  $P\circ (Q_1,\ldots,Q_k)$. These blocks $\tau$ and $\tau'$ are
  essentially blocks of some of the $Q_i$, but shifted around according to
  the identifications made. If $\tau$ and $\tau'$ originate in the same $Q_i$,
  then they were noncrossing in $Q_i$, and since the shifting into block
  $\pi_i$ is effectuated by a monotone bijection $[n_i] \to \pi_i \subset
  [n]$, this clearly preserves the noncrossing condition. If the two 
  blocks $\tau$ and $\tau'$
  originate in distinct compositions $Q_i$ and $Q_j$, then they are mapped
  into distinct blocks $\pi_i$ and $\pi_j$ of $P$, and since $\pi_i$ and
  $\pi_j$ do not cross in $P$, a block of a refinement of $\pi_i$ cannot 
  cross a block of a refinement of $\pi_j$. In either case we see that
  $\tau$ and $\tau'$ do not cross in $P\circ (Q_1,\ldots,Q_k)$.
\end{proof}

\medskip

\textbf{Example}.
\begin{align*}
	&\big(\{1,5,6\};\{2,3,4\};\{7,8,9,10,11\}\big) 
	\circ \big((\{\{1\};\{2,3\}),(\{1,3\};\{2\}),(\{1,2,5\};\{3,4\})\big)\\
	&=(\{1\};\{5,6\};\{2,4\};\{3\};\{7,8,11\};\{9,10\}).
\end{align*}
Graphically:
\begin{align*}
\begin{tikzpicture}[line cap=round,line join=round,>=triangle 45,x=0.3cm,y=0.3cm]
\clip(0.8,1.) rectangle (11.2,4.);
\draw [line width=0.8pt,color=red] (1.,1.)-- (5.,1.);
\draw [line width=0.8pt,color=red] (5.,1.)-- (6.,1.);
\draw [line width=0.8pt,color=red] (6.,1.)-- (6.,2.);
\draw [line width=0.8pt,color=red] (5.,1.)-- (5.,2.);
\draw [line width=0.8pt,color=red] (1.,1.)-- (1.,2.);
\draw [line width=0.8pt,color=green] (2.,3.)-- (2.,2.);
\draw [line width=0.8pt,color=green] (2.,2.)-- (4.,2.);
\draw [line width=0.8pt,color=green] (4.,2.)-- (4.,3.);
\draw [line width=0.8pt,color=green] (3.,2.)-- (3.,3.);
\draw [line width=0.8pt,color=blue] (7.,2.)-- (7.,1.);
\draw [line width=0.8pt,color=blue] (7.,1.)-- (11.,1.);
\draw [line width=0.8pt,color=blue] (11.,1.)-- (11.,2.);
\draw [line width=0.8pt,color=blue] (10.,1.)-- (10.,2.);
\draw [line width=0.8pt,color=blue] (8.,1.)-- (8.,2.);
\draw [line width=0.8pt,color=blue] (9.,1.)-- (9.,2.);
\draw (0.2,3.4) node[anchor=north west,color=red] {\footnotesize $1$};
\draw (1.2,4.4) node[anchor=north west,color=green] {\footnotesize $2$};
\draw (6.2,3.4) node[anchor=north west,color=blue] {\footnotesize $3$};
\end{tikzpicture}
\circ
\bigg(\;
\begin{tikzpicture}[line cap=round,line join=round,>=triangle 45,x=0.3cm,y=0.3cm]
\clip(3.8,1.) rectangle (6.2,3.);
\draw [line width=0.8pt,color=red] (5.,1.)-- (6.,1.);
\draw [line width=0.8pt,color=red] (6.,1.)-- (6.,2.);
\draw [line width=0.8pt,color=red] (5.,1.)-- (5.,2.);
\draw [line width=0.8pt,color=red] (4.,1.)-- (4.,2.);
\draw (3.2,3.4) node[anchor=north west,color=red] {\footnotesize $1$};
\draw (4.2,3.4) node[anchor=north west,color=red] {\footnotesize $2$};
\end{tikzpicture}
\ , \
\begin{tikzpicture}[line cap=round,line join=round,>=triangle 45,x=0.3cm,y=0.3cm]
\clip(1.8,1.) rectangle (4.2,4.);
\draw [line width=0.8pt,color=green] (2.,2.)-- (2.,1.);
\draw [line width=0.8pt,color=green] (4.,1.)-- (4.,2.);
\draw [line width=0.8pt,color=green] (3.,2.)-- (3.,3.);
\draw [line width=0.8pt,color=green] (2.,1.)-- (4.,1.);
\draw (1.2,3.4) node[anchor=north west,color=green] {$1$};
\draw (2.2,4.4) node[anchor=north west,color=green] {$2$};
\end{tikzpicture}
\ , \
\begin{tikzpicture}[line cap=round,line join=round,>=triangle 45,x=0.3cm,y=0.3cm]
\clip(6.8,1.) rectangle (11.2,4.);
\draw [line width=0.8pt,color=blue] (7.,2.)-- (7.,1.);
\draw [line width=0.8pt,color=blue] (7.,1.)-- (11.,1.);
\draw [line width=0.8pt,color=blue] (11.,1.)-- (11.,2.);
\draw [line width=0.8pt,color=blue] (8.,1.)-- (8.,2.);
\draw [line width=0.8pt,color=blue] (9.,3.)-- (9.,2.);
\draw [line width=0.8pt,color=blue] (9.,2.)-- (10.,2.);
\draw [line width=0.8pt,color=blue] (10.,2.)-- (10.,3.);
\draw (6.2,3.4) node[anchor=north west,color=blue] {\footnotesize $1$};
\draw (8.2,4.4) node[anchor=north west,color=blue] {\footnotesize $2$};
\end{tikzpicture}
\bigg)
\ = \
\begin{tikzpicture}[line cap=round,line join=round,>=triangle 45,x=0.3cm,y=0.3cm]
\clip(0.8,1.) rectangle (11.2,5.);
\draw [line width=0.8pt,color=red] (5.,1.)-- (6.,1.);
\draw [line width=0.8pt,color=red] (6.,1.)-- (6.,2.);
\draw [line width=0.8pt,color=red] (5.,1.)-- (5.,2.);
\draw [line width=0.8pt,color=red] (1.,1.)-- (1.,2.);
\draw [line width=0.8pt,color=green] (2.,3.)-- (2.,2.);
\draw [line width=0.8pt,color=green] (2.,2.)-- (4.,2.);
\draw [line width=0.8pt,color=green] (4.,2.)-- (4.,3.);
\draw [line width=0.8pt,color=blue] (7.,2.)-- (7.,1.);
\draw [line width=0.8pt,color=blue] (7.,1.)-- (11.,1.);
\draw [line width=0.8pt,color=blue] (11.,1.)-- (11.,2.);
\draw [line width=0.8pt,color=blue] (8.,1.)-- (8.,2.);
\draw [line width=0.8pt,color=green] (3.,3.)-- (3.,4.);
\draw [line width=0.8pt,color=blue] (9.,3.)-- (9.,2.);
\draw [line width=0.8pt,color=blue] (9.,2.)-- (10.,2.);
\draw [line width=0.8pt,color=blue] (10.,2.)-- (10.,3.);
\draw (0.2,3.4) node[anchor=north west,color=red] {\footnotesize $1$};
\draw (4.2,3.4) node[anchor=north west,color=red] {\footnotesize $2$};
\draw (1.2,4.4) node[anchor=north west,color=green] {\footnotesize $3$};
\draw (2.2,5.4) node[anchor=north west,color=green] {\footnotesize $4$};
\draw (6.2,3.4) node[anchor=north west,color=blue] {\footnotesize $5$};
\draw (8.2,4.4) node[anchor=north west,color=blue] {\footnotesize $6$};
\end{tikzpicture}
\ = \
\begin{tikzpicture}[line cap=round,line join=round,>=triangle 45,x=0.3cm,y=0.3cm]
\clip(0.8,1.) rectangle (11.2,4.);
\draw [line width=0.8pt] (5.,1.)-- (6.,1.);
\draw [line width=0.8pt] (6.,1.)-- (6.,2.);
\draw [line width=0.8pt] (5.,1.)-- (5.,2.);
\draw [line width=0.8pt] (1.,1.)-- (1.,2.);
\draw [line width=0.8pt] (2.,2.)-- (2.,1.);
\draw [line width=0.8pt] (4.,1.)-- (4.,2.);
\draw [line width=0.8pt] (7.,2.)-- (7.,1.);
\draw [line width=0.8pt] (7.,1.)-- (11.,1.);
\draw [line width=0.8pt] (11.,1.)-- (11.,2.);
\draw [line width=0.8pt] (8.,1.)-- (8.,2.);
\draw [line width=0.8pt] (3.,2.)-- (3.,3.);
\draw [line width=0.8pt] (9.,3.)-- (9.,2.);
\draw [line width=0.8pt] (9.,2.)-- (10.,2.);
\draw [line width=0.8pt] (10.,2.)-- (10.,3.);
\draw [line width=0.8pt] (2.,1.)-- (4.,1.);
\draw (0.2,3.4) node[anchor=north west] {\footnotesize $1$};
\draw (4.2,3.4) node[anchor=north west] {\footnotesize $2$};
\draw (1.2,3.4) node[anchor=north west] {\footnotesize $3$};
\draw (2.2,4.4) node[anchor=north west] {\footnotesize $4$};
\draw (6.2,3.4) node[anchor=north west] {\footnotesize $5$};
\draw (8.2,4.4) node[anchor=north west] {\footnotesize $6$};
\end{tikzpicture} \ .
\end{align*}
Here the numbers indicate the block numbering that makes a partition into a 
composition.

\begin{rem}
  The axioms force us to exclude the colour $0$ and the empty composition.
  The only composition 
  of $\emptyset$ is of course the empty composition, and it has arity $0$,
  not $1$.  Were we to insist on including the empty composition (as a
  nullary operation) we would have to give up the unit axiom.  We shall
  return to this issue when we come to comodule bialgebras in 
  Section~\ref{ssect:comod}.
  
  In particular, since every non-empty composition has at least one block,
  there are no nullary operations, which is to say that the operad is
  reduced.
\end{rem}

\medskip

The following is a general construction of an ordinary operad from a 
coloured one, but it requires the setting of linear operads, i.e.,~operads 
in the symmetric monoidal category $(\Vect, \otimes, \K)$ instead of
the category of sets.

\begin{prop}
Let $C$ be a set and let $\P$ be a $C$-coloured linear operad. 
For any $n\geq 1$, we put:
\begin{align*}
	\widehat{\P}(n)&=\prod_{c_1,\ldots,c_n,c\in C}\P(c_1,\ldots,c_n;c).
\end{align*}
Let us assume that for any $(c_1,\ldots,c_n)\in C^n$, the set
$\ind(c_1,\ldots,c_n)=\{c\in C\mid \P(c_1,\ldots,c_n;c)\neq 0 \}$ is
finite. We define a composition law on $\widehat{\P}$ in the following way: we
consider elements $p\in \widehat{\P}(n)$, $q_i\in \widehat{\P}(k_i)$ for any $i\in
[n]$ and put:
\begin{align*}
	   p=\prod_{c_1,\ldots,c_n,c\in C} p_{c_1,\ldots,c_n;c}, \qquad
	q_i=\prod_{c_{i,1},\ldots,c_{i,k_i},c_i\in C} q_{c_{i,1},\ldots,c_{i,k_i};c_i}.
\end{align*} 
Then:
\begin{align*}
	p\bullet (q_1,\ldots,q_n)
	&=\prod_{\substack{c_{1,1},\ldots,c_{n,k_n},c\in C}}\Bigg(\sum_{\substack{c_1\in \ind(c_{1,1},\ldots,c_{1,k_1}),\\ 
	\vdots\\ 
	c_n\in \ind(c_{n,1},\ldots,c_{n,k_n})}} p_{c_1,\ldots,c_n;c}\circ (q_{c_{1,1},\ldots,c_{1,k_1};c_1},
	\ldots,q_{c_{n,1},\ldots,c_{n,k_n};c_n})\Bigg).
\end{align*}
The unit of $\widehat{\P}$ is:
\begin{align*}
	I&=\prod_{c\in C} I_c\in \prod_{c\in C}\P(c;c)\subseteq \widehat\P(1).
\end{align*}
\end{prop}

\begin{proof}
  By hypothesis on $\ind(c_1,\ldots,c_n)$, the sum defining $p\bullet
  (q_1,\ldots,q_n)_{c_{1,1},\ldots,c_{n,k_n};c}$ is finite, so the law $\bullet$ is
  well-defined. The associativity of $\circ$ and its compatibility with the
  action of the symmetric groups imply that $\bullet$ is associative and
  compatible with the action of the symmetric groups. For any $p\in
  \widehat{\P}(n)$:
\begin{align*}
	I\bullet p
	&=\prod_{c_1,\ldots,c_n,c\in C}(I_c\circ p_{c_1,\ldots,c_n;c}+0)
	  =\prod_{c_1,\ldots,c_n,c\in C}p_{c_1,\ldots,c_n;c}
	  =p,\\ 
	p\bullet(I,\ldots,I)
	&=\prod_{c_1,\ldots,c_n,c\in C}p_{c_1,\ldots,c_n;c}\circ (I_{c_1},\ldots,I_{c_n})
	  =\prod_{c_1,\ldots,c_n,c\in C}p_{c_1,\ldots,c_n;c}
	  =p.
\end{align*}
So $I$ is a unit of $\P$. (Notice that in all the above constructions,
products could be replaced by direct sums except for the fact that then
the operad $\widehat\P$ would not have a unit since $I\notin \bigoplus_{c\in C}
I_c$.)
\end{proof}

\begin{defi}[Operads of compositions]
  Let $\P$ be the linearization of $\SC$ or $\NCC$, with set of
  colours $C=\N^\ast$.
  Then, for any $c_1,\ldots,c_n\in C$:
\begin{align*}
	\ind(c_1,\ldots,c_n)&=\{c_1+\cdots+c_n\}.
\end{align*}
The construction of the previous proposition can thus be applied to $\SC$
and $\NCC$ to obtain two ordinary (linear) operads $\widehat{\SC}$ and
$\widehat{\NCC}$.
\end{defi}


\subsection{Induced bialgebras}
\label{indbialg}

The general construction of a bialgebra from an operad (already invoked in
Subsection~\ref{ssect:Hopfstruct}) applies equally well to coloured
symmetric operads $\mathcal{P}$ (subject to suitable finiteness
conditions).  As an algebra, the associated bialgebra is free on the set of all operations.
The coproduct of
an operation $R$ is defined as
$$
	\Delta(R) = \sum_{R=P \circ (Q_1,\ldots, Q_k)} P \otimes Q_1 \cdots Q_k.
$$
In plain words, the coproduct is given by summing over all the ways $R$
could have arisen from the composition law, and then putting the receiving
operation $P$ on the left and the monomial of all the operations fed into
$P$ on the right. Like in the one-colour case, this bialgebra is graded (by
arity-minus-one), but is never connected.
In the present case, degree zero is spanned by the identity operations $I_n$, 
the single-block partitions,
and for most purposes it is too drastic to divide out by those.  We shall 
therefore mostly accept that the resulting bialgebras are not Hopf.
Another ``drawback'' is that the general bialgebra construction yields a
bialgebra spanned by compositions rather than partitions.  This is an 
artifact of having introduced the numbering of blocks, and it is
corrected easily by passing to coinvariants, meaning dividing out by the
actions of the symmetric groups.  Since these actions are free (because
the numberings were introduced freely), this quotiening is well behaved.

Although it is possible to describe these construction by hand, it is
usually better to invoke general theorems to the effect that the
construction works. One general framework for the construction is described
in \cite{FoissyOperads}, applied to single-coloured operads $\widehat{\SC}$
and $\widehat{\NCC}$. (The construction is denoted
$\mathbf{D}_\mathcal{P}^\ast$ in that article.)
With the previous notation for coloured linear operads,
$\prod_{c_1,\ldots,c_n,c\in C}\P(c_1,\ldots,c_n;c)$ is the linear dual of
$\oplus_{c_1,\ldots,c_n,c\in C}\P(c_1,\ldots,c_n;c)^\ast$. These
constructions allow to define bialgebras $\Bblcomp$ and $\Bblnccomp$ of
set compositions and noncrossing compositions, and pass to coinvariants to
get bialgebras $\Bbl$ and $\Bblnc$ of set partitions and noncrossing
partitions. (This composite construction is denoted $D_{\P}^\ast$ in 
\cite{FoissyOperads}.)

Another general construction goes via the so-called two-sided bar
construction, as exploited in \cite{Kock-Weber}: to any (coloured) operad
one can first assign its two-sided bar construction, which is a simplicial
groupoid (see~\cite{Kock-Weber}). This simplicial groupoid is in fact a
symmetric monoidal decomposition space in the sense of \cite{GKT1}, and
therefore admits an incidence bialgebra, which in the present two cases are
$\Bbl$ and $\Bblnc$ directly ($\Bblcomp$ and $\Bblnccomp$ cannot be
obtained in this way). (Incidence coalgebras of posets and categories are
also special cases of this, and the setting allows for precise comparison
results. We shall briefly touch upon this below when we  compare the 
operad approach with the classical lattice approach (see
Remark~\ref{decalage}).)

\bigskip

We summarize the outcome of these constructions in the following proposition, and
proceed to give explicit descriptions of the structures.

\begin{prop}
  The coloured symmetric operads $\SC$ and $\NCC$ induce four bialgebra structures
  and bialgebra homomorphisms:
  $$
\begin{tikzcd}[sep=small]
   &  \text{\footnotesize noncrossing} && \text{\footnotesize general} 
   & \phantom{xxxxxxx} \\
	 \text{\footnotesize compositions} & 
	 \Bblnccomp  \ar[dd, twoheadrightarrow] && 
	 \Bblcomp \ar[dd, twoheadrightarrow] & \\
	 &&&& \\
	\text{\footnotesize partitions} & \Bblnc  && 
	\Bbl &
\end{tikzcd}
$$
\end{prop}

The bialgebra $\Bblcomp$ is freely generated as an algebra by $\latSC$,
and $\Bblnccomp$ is freely generated by $\latNCC$.
It is the coproduct, denoted
$\delta_0$, which is the most interesting structure. We proceed to describe
it in elementary terms.

For any $P$, $Q\in \latSC(n)$, recall that $P\leq Q$ if there exist $R_1,\ldots,R_k\in \latSC$ such that: 
\begin{align*}
	P=Q\circ (R_1,\ldots,R_k).
\end{align*}
This is an adaptation to compositions of the usual ordering by
coarsening of partitions ($P \leq Q$ when $Q$ is coarser than $P$). 
For fixed $P$ and $Q$, such $R_1,\ldots,R_k$ are unique if they 
exist;  in that case we put $P/Q:=R_1\cdots R_k$.
We can now describe the coproduct 
$\delta_0 \colon \Bblcomp \to \Bblcomp \otimes \Bblcomp$
explicitly. For $P \in \latSC$ we have
\begin{align*}
	\latSC \ni P\mapsto \delta_0(P) =\sum_{P\leq Q} Q\otimes P/Q.
\end{align*}

The space of invariants under the action of the symmetric groups of
$\Vect(\latSC)$ is naturally identified with $\Vect(\latSP)$.
Accordingly, $\Bbl$ is the free {\em commutative} algebra generated by 
$\latSP$, and  similarly $\Bblnc$ is the free {\em commutative}
algebra generated by $\latNCP$.

The coproduct $\delta$ on $\Bbl$ is induced from $\delta_0$.
To describe it, further notation is required:

For $P$, $Q\in \latSP(n)$, we shall say that $P\leq Q$ if each block of
$Q$ is the union of the blocks of $P$ it contains. This is the usual poset
of partitions of degree $n$. Its minimal element is the finest partition
$J_n=\{\{1\},\ldots,\{n\}\}$, the unique element of $\latSP(n,n)$, and its
maximal element is the coarsest partition $I_n=\{[n]\}$, the unique element
of $\latSP(1,n)$.

Denoting $Q=\{\tau_1,\ldots,\tau_l\}$, we put
\begin{align*}
	P/Q&\ =\ \st{P_{\mid \tau_1}}\cdots \st{P_{\mid \tau_l}} \ \in \Bbl .
\end{align*}
Finally we can describe the coproduct explicitly. We have $\delta \colon \Bbl \to \Bbl \otimes \Bbl$, and:
\begin{align}
	\latSP \ni P\mapsto \delta(P) =\sum_{P\leq Q} Q\otimes P/Q \label{refinecoprod}.
\end{align}

The counit of $\Bbl$ is denoted by $\varepsilon_\Bbl$; for 
$P\in \latSP(k,n)$, we have $\varepsilon_\Bbl(P)=\delta_{k,1}$.

\begin{rem}
For the noncrossing case the description is exactly the same. The sum runs over noncrossing partitions $Q$ coarser than $P$.
This is meaningful because if $P$ is noncrossing then for any coarser noncrossing partition 
$Q$, the restrictions $P_{\mid \tau_i}$ forming the monomial
are noncrossing as well.
\end{rem}

\begin{ex}
We give some example computation in the case of noncrossing partitions.
Since the sum defining $\delta$ is over all coarser noncrossing 
partitions, the coarsest partitions $I_n$ are all group-like:
\allowdisplaybreaks
\begin{eqnarray*}
  \delta( \ncone )  =  
  \ncone \otimes \ncone
  \qquad
  \delta( \nctwoW )  =  
  \nctwoW \otimes \nctwoW
  \qquad
  \delta( \ncthreeW )  = 
  \ncthreeW \otimes \ncthreeW
\end{eqnarray*}
and 
the finest partitions $J_n$ have the longest formulas for coproduct:
\begin{eqnarray*}
  \delta( \nctwo ) & = & 
   \nctwoW\otimes\nctwo 
  +  \nctwo\otimes\ncone {\cdot} \ncone 
  \\
  \delta( \ncthree ) & = &\ncthreeW\otimes 
  \ncthree 
  +  
  \big( \nconeinsidetwoWW  +  \nconetwoW  +  \nctwoWone \big) \otimes\ncone {\cdot} \nctwo  
  + \ncthree  \otimes \ncone {\cdot} \ncone {\cdot} \ncone  
  \\
  \delta(\ncfour) & = &
  \ncfourW \otimes \ncfour  
   +  \nctwoWtwoW \otimes \nctwo {\cdot} \nctwo 
   +  \nctwoWinsidetwoWWW \otimes \nctwo {\cdot} \nctwo 
   +   \nconeonetwoW \otimes \ncone {\cdot} \ncone {\cdot} \nctwo
   +  \nconeoneinsidetwoWW \otimes \ncone {\cdot} \ncone {\cdot} \nctwo \\ 
   & &
   + \ncone \nctwoW \ncone   \otimes  \ncone {\cdot} \ncone {\cdot} \nctwo
   +  \tikz[x=1.3cm,y=1.3cm]{ 
\draw (0.5,0.2)--(0.5,0)--(0.8,0)--(0.8,0.2);
\path (0.6,0.1) pic {nctwo};}\otimes
   \ncone {\cdot} \ncone {\cdot} \nctwo 
   + \nconeinsidetwoWW \ncone  \otimes   \ncone {\cdot} \ncone {\cdot} \nctwo 
   +  \nctwoW \ncone \ncone \otimes
   \ncone {\cdot} \ncone {\cdot} \nctwo  \\
   & &
   + \nconethreeW  \otimes  \ncone {\cdot} \ncthree 
   + \nconeinsidethreeWW \otimes\ncone {\cdot} \ncthree 
   + \nconeinsidethreerightWW \otimes\ncone {\cdot} \ncthree 
   + \ncthreeW \ncone \otimes  \ncone {\cdot} \ncthree 
   +  \ncfour \otimes\ncone {\cdot} \ncone {\cdot} \ncone {\cdot} \ncone .
\end{eqnarray*}
Here are a few other examples, which will be referenced later:
\label{examples:sticksandcups}
\begin{eqnarray*}
  \delta(\nconetwoW) & = & 
   \nconetwoW \otimes\ncone {\cdot} \nctwoW  + \ncthreeW\otimes  \nconetwoW   
  \\
  \delta(\nctwoWtwoW) & = & 
   \nctwoWtwoW \otimes \nctwoW{\cdot} \nctwoW 
  + \ncfourW \otimes \nctwoWtwoW 
  \\
  \delta( \nctwoWinsidetwoWWW ) & = & 
   \nctwoWinsidetwoWWW \otimes \nctwoW{\cdot} \nctwoW 
  +  \ncfourW \otimes\nctwoWinsidetwoWWW 
  \\
  \delta(\nconeonetwoW) & = &
   \nconeonetwoW \otimes\ncone {\cdot} \ncone {\cdot} \nctwoW 
  + 
   \nctwoWtwoW \otimes \nctwo {\cdot} \nctwoW 
  +
   \big( \nconethreeW + 
  \nconeinsidethreeWW \big)\otimes \ncone {\cdot} \nconetwoW 
  +  \ncfourW\otimes\nconeonetwoW 
  \\
  \delta( \nconeoneinsidetwoWW ) & = &
   \nconeoneinsidetwoWW \otimes\ncone {\cdot} \ncone {\cdot} \nctwoW 
  +
   \nconeinsidethreerightWW \otimes\ncone {\cdot} \nconetwoW 
  + 
   \nconethreeW \otimes\ncone {\cdot}  \nconeinsidetwoWW 
  + \ncfourW \otimes \nconeoneinsidetwoWW .
\end{eqnarray*}
\end{ex}


\subsection{Comparison with the lattice approach and M\"obius inversion}
\label{subsec:Phi}

The monomial $P/Q$ appearing on the right-hand side in the coproduct
formula \eqref{refinecoprod} plays an important role in relating the
bialgebras $\Bbl$ and $\Bblnc$ with the incidence coalgebras of the
lattices of partitions and noncrossing partitions. In this subsection,
for simplicity 
we treat only the case of noncrossing partitions, but the whole discussion
(though not the example computations) carries over to the case of general
set partitions.

We shall see, for example, that Speicher's basic
moment-cumulant formula given in terms of M\"obius inversion in the
incidence algebra of the lattice of noncrossing partitions can be rendered
elegantly in the bialgebra $\Bblnc$.

\begin{defi}[Fibre]
\label{def:fibre}
Given two (noncrossing) partitions, $P,Q \in \latNCP$, one finer than the other, $P\leq Q,$
the {\em fibre} map is defined by: 
$$
	\Phi(P,Q) := P/Q.
$$ 
\end{defi}

{\bf Example:}
$$
	\Phi\big(  \
\tikz[x=1.3cm,y=1.3cm]{\path (0,0) pic {ncone}; \path (0.1,0.1) pic {nctwoW};
\path (0.3,0.1) pic {nctwo};\path (0.5,0) pic {ncone}; \path (0.6,0) pic {nconeinsidetwoWW};}
\ , \
\tikz[x=1.3cm,y=1.3cm]{\draw (0,0.2)--(0,0)--(0.5,0)--(0.5,0.2); \path (0.1,0.1) pic {ncfourW}; 
\path (0.6,0) pic {ncthreeW};}
\ \big)
\ \ = \ \ 
\nctwo \cdot \tikz[x=1.3cm,y=1.3cm]{\path (0,0) pic {nctwoW}; \path (0.2,0) pic 
{nctwo};} \cdot \nconeinsidetwoWW.
$$

The linear dual $\Bblnc^\ast$ is an algebra under the corresponding convolution product:
for two linear maps $\phi, \psi: \Bblnc \to \K$, the convolution product is given by
$$
	(\phi*\psi)(P) = \sum_{P\leq Q} \phi(Q) \psi(P/Q)  .
$$
The neutral element for the convolution product is the bialgebra counit
$\varepsilon_\Bblnc$, which takes value $1$ on all $I_n$ (as
well as on monomials in the $I_n$) and value $0$ on all other noncrossing
(multi)partitions.

Recall that for any bialgebra $T$, the characters (algebra homomorphisms 
$T \to \K$) form a
monoid for the convolution product: the product of two characters
$\phi,\psi$ is given by $\phi \ast \psi := m_\K\circ (\phi\otimes\psi
)\circ\Delta_T$. For the case of $\Bblnc$, a character $\psi$ is thus multiplicative, in
the sense that if $P$ is a commutative monomial of partitions, $P= P_1\cdots P_r$, then
$\psi(P) = \psi(P_1) \cdots \psi(P_r)$ (and we also have
$\psi(\mathbf{1})=1$, where the argument $\mathbf{1}$ denotes the empty product, the unit
element in the algebra).

A basic character on
$\Bblnc$ is the {\em zeta function}
\begin{align*}
	\zeta:&\left\{\begin{array}{rcl}
	\Bblnc&\longrightarrow& \K \\
	P&\longmapsto& 1.
\end{array}\right.
\end{align*}
It is a general fact that the zeta function is convolution invertible;
its inverse is by definition the M\"obius function $\mu$, which is again a 
character, determined by the following recursive formula:
\begin{eqnarray*}
  \mu(I_n) & = & 1   \qquad\qquad\qquad\ \text{ for all $n>0$},\\
  \mu(P) & = & - \sum_{P<Q} \mu(Q) \qquad \text{for $P\neq I_n$}.
\end{eqnarray*}
The sum is over all {\em strictly} coarser partitions or noncrossing partitions.

\begin{lemma}
In the dual $\Bblnc^\ast$, given a character $\kappa$,
define 
$$
	\phi :=  \zeta *\kappa .
$$
Then: 
\begin{equation}\label{eqB}
	\phi(J_n)=\sum\limits_{P=\{\pi_1,\dots,\pi_k\}\in \latNCP(n)}\kappa(J_{\sharp\pi_1})\cdots \kappa(J_{\sharp\pi_k}).
\end{equation}
\end{lemma}

\begin{cor}
  With
\begin{align*}
	k_n &:= \kappa( J_n ) = \kappa(\ncone {\cdots} \nctwo) &\text{ and }& &
		m_n &:= \phi( J_n ) = \phi(\ncone {\cdots} \nctwo) ,
\end{align*}
Equation~\eqref{eqB} recovers the familiar free moment-cumulant relation
$m_n=\!\!\!\!\!\sum\limits_{Q\in \latNCP(n)} \!\! \!\!k_{Q}$ from \eqref{freecumu}.
\end{cor}

\begin{rem}
  Note that the moments and cumulants here are defined using the finest
  partitions, not the coarsest, which may surprise readers familiar with
  the work of Speicher~\cite{Speicher94,SpeicherNica}. The contrast will be
  fully resolved below.
\end{rem}

\begin{ex}
Using the above coproduct formulae and the monomial multiplicativity, we
calculate for instance the second, third and fourth moments in the 
noncrossing case:
\begin{eqnarray*}
	m_2 
	&=& \phi( \nctwo ) = ( \zeta *\kappa )(\nctwo) 
	  =\kappa(\ncone {\cdot} \ncone) + \kappa( \nctwo )\\
	&=& k_1 k_1 + k_2,\\
	m_3 
	&=& \phi(\ncthree) = (\zeta *\kappa )(\ncthree) 
	  = \kappa(\ncone{\cdot}\ncone{\cdot}\ncone) + 3 \kappa(\ncone {\cdot} \nctwo) + \kappa(\ncthree ) \\
	&=& k_1 k_1 k_1 + 3 k_1 k_2 + k_3, \\
  m_4 &=& \kappa(\ncone{\cdot}\ncone{\cdot}\ncone{\cdot}\ncone) 
  \ + \
  6 \kappa(\ncone {\cdot} \ncone {\cdot} \nctwo)
  \ + \
  4 \kappa(\ncone {\cdot} \ncthree)
  \ + \
  2 \kappa (\nctwo {\cdot} \nctwo) 
  \ + \ \kappa(\ncfour)
  \\
  & = & k_1 k_1 k_1 k_1 + 6 k_1 k_1 
  k_2 + 4 k_1 k_3 +2 k_2 k_2 + k_4.
\end{eqnarray*}
(Note that the coefficient $2$ in front of $k_2 k_2$ is the first that
distinguishes the case of noncrossing partitions from the case of general
set partitions, where this coefficient would be $3$.)

Conversely, using the recursive formula for $\mu$, one finds quickly that 
\begin{align*}
	\mu(\ncone) &=\mu(\nctwoW) = \mu(\ncthreeW) = \mu(\ncfourW) = 1\\
	\mu(\nctwo) &= \mu(\nconeinsidetwoWW) 
				= \mu(\nconetwoW) 
				= \mu(\nctwoWone)
				= \mu(\ncone \ncthreeW)  
				= \mu(\nconeinsidethreeleftWW)
				= \mu(\nconeinsidethreerightWW)
				= \mu(\ncthreeW \ncone)
				= -1,\\
	\mu(\ncthree) &= \mu( \nctwoW \nctwo ) 
				= \mu( \ncone \nctwoW \ncone) 
				= \mu( \nctwo \nctwoW )
				= 2,
\end{align*}
but
\begin{align*}
\mu(\nconeoneinsidetwoWW) &= \mu(\nconeinsidetwoWW \ncone) = 1,&
	\mu(\ncfour) &= -5.
\end{align*}
Combining with the basic comultiplication table (see computations in 
Example~\ref{examples:sticksandcups}), we get for the second, third
and fourth cumulants
\begin{eqnarray*}
	k_2 
	&=& \kappa( \nctwo ) = (\mu *\phi)(\nctwo) 
	= \mu(\nctwo)\phi(\ncone {\cdot} \ncone) + \mu(\nctwoW)\phi( \nctwo )\\
	&=& - m_1 m_1 + m_2, \\
	k_3 
	&=& \kappa(\ncthree) = (\mu *\phi)(\ncthree) 
	= \mu(\ncthreeW)\phi(\ncthree ) + 3 \mu(\nctwoW\ncone) \phi(\ncone {\cdot} \nctwo)
	+ \mu(\ncthree) \phi(\ncone{\cdot}\ncone{\cdot}\ncone) \\
	& = & m_3   -3 m_1 m_2 + 2 m_1^3, \\
  k_4 
  & = & m_4 - 2 m_2^2 + 10 m_1^2 m_2 - 4 m_1 m_3 + m_1^4.
\end{eqnarray*}
\end{ex}

These computations are completely analogous to the calculations usually
done to relate moments and free cumulants in the lattice of noncrossing
partitions (see for instance~\cite{SpeicherNica}), but they are simpler in
the present framework since one can work with noncrossing partitions directly
instead of working with {\em intervals} of noncrossing partitions. As
observed above, it may seem disturbing that we put $k_n: =
\kappa(\ncfive)$, whereas Speicher~\cite{Speicher94} puts $k_n: =
\kappa(\ncfiveW)$, which is really shorthand for $\kappa( \llbracket J_n,
I_n\rrbracket)$ in that context (the one of the incidence algebra of the
lattice). The reason for this is that the relationship between the noncrossing
partition lattice and the bialgebra $\Bblnc$ is actually given by the assignment
$\Phi$ introduced in Definition \ref{def:fibre}, which to any interval associates its fibre.

Recall, e.g., from \cite{SpeicherNica} that the incidence coalgebra 
of a lattice is the free vector space spanned by the intervals $\llbracket 
P,Q \rrbracket$ in the lattice, and that the coproduct is given by
$$
	\delta_{\operatorname{lat}} (\llbracket P,Q \rrbracket) = \sum_{M\in \llbracket P,Q 
	\rrbracket} \llbracket P,M\rrbracket \otimes \llbracket M,Q \rrbracket.
$$
We shall actually use the opposite coalgebra, meaning that the order of the
tensor factors is reversed. This opposite coalgebra we call
$\Bblnc_{\operatorname{lat}}$.
For contrast, the operad bialgebra will 
be denoted 
$\Bblnc_{\operatorname{opd}}$ in the 
following discussion.  

\begin{rem}
We emphasize that the ``opposite'' 
occurring here is unrelated to the question we are addressing, i.e., ``finer vs.~coarser''.  Indeed, working with the 
opposite coalgebra is purely a matter of convention of order of the tensor factors in the coproducts.
\end{rem}

The following result resolves the relation between the two coalgebras $\Bblnc_{\operatorname{lat}}$ and $\Bblnc_{\operatorname{opd}}$.

\begin{prop}
  The assignment $\Phi$ constitutes a coalgebra homomorphism 
  $\Bblnc_{\operatorname{lat}} \to \Bblnc_{\operatorname{opd}}$.
\end{prop}

\begin{proof}
  This follows from the next two lemmas. The first is a standard property
  of noncrossing partitions, the second follows from the
  definition of $\Phi$.
\end{proof}

\begin{lemma}
  Every interval $\llbracket P,Q \rrbracket$ is isomorphic as a poset to a
  product of intervals ending in a coarsest partition $I_n$. Precisely,
$$
  	\llbracket P,Q \rrbracket \simeq \llbracket P_1, I_{n_1} 
	\rrbracket \times \cdots \times  \llbracket P_k, I_{n_k} \rrbracket,
$$
where $k$ is the number of blocks in $Q$, the $i$th block of $Q$ is of size $n_i$, and $P_i$ is the corresponding partition of $[n_i]$.
\end{lemma}

\begin{lemma}
For any interval $\llbracket P,Q \rrbracket$, we have
$$
  	\Phi(\llbracket P,Q \rrbracket) = \Phi( \llbracket P_1,I_{n_1} \rrbracket)
	\cdots \Phi(\llbracket P_k,I_{n_k} \rrbracket),
$$
where again $k$ is the number of blocks in $Q$, and the $i$th block  of  $Q$ 
is of size $n_i$, and $P_i$ is the corresponding finer partition.
\end{lemma}

\begin{cor}
The dual map $\Bblnc^\ast_{\operatorname{opd}}\to 
\Bblnc^\ast_{\operatorname{lat}}$ is a homomorphism of convolution algebras.
\end{cor}

\begin{rem}
Characters on $\Bblnc_{\operatorname{opd}}$, say $\kappa_{\operatorname{opd}} :
\Bblnc_{\operatorname{opd}} \to \K$, induce linear forms on
$\Bblnc_{\operatorname{lat}}$, say $\kappa_{\operatorname{lat}} :
{\Bblnc_{\operatorname{lat}}} \to \K$ such that
$$
	\kappa_{\operatorname{lat}}( \llbracket P,Q \rrbracket ) \ = \ 
	\kappa_{\operatorname{opd}}(\Phi( \llbracket P,Q \rrbracket)) = 
	\kappa_{\operatorname{opd}}(P/Q) .
$$
These linear forms have the special property (actually desired) that their
value on an interval only depends on its canonical product splitting 
(i.e.~its type, in the sense of Speicher).

This correspondence explains the apparent discrepancy we noticed: when in
the lattice setting one writes $\kappa(I_n)$ it is really shorthand for
$\kappa(\llbracket J_n,I_n \rrbracket)$, and the fibre of the interval $\llbracket J_n,I_n \rrbracket$ is clearly
the partition $J_n$, so that $
	\kappa_{\operatorname{lat}}(I_n)
	=
	\kappa_{\operatorname{opd}}(J_n) ,$
and both can be called $k_n$ without conflict.
\end{rem}

\begin{rem}\label{decalage}
  The coalgebra homomorphism $\Phi$ should not be considered as something strange or ad hoc. In
  fact, it is an instance of a general phenomenon, visible in the simplicial
  viewpoints hinted at in the preliminary discussion in 
  Subsection~\ref{indbialg}. 
  There it was mentioned
  that $\Bblnc_{\operatorname{opd}}$ is the incidence bialgebra of a 
  certain simplicial groupoid $X$
  (more precisely a monoidal decomposition space), obtained as the
  two-sided bar construction of the operad $\NCP$. In this simplicial
  viewpoint, one can take upper decalage of $X$ to find essentially the
  nerve of the noncrossing partitions lattice, and then the dec map
  $\operatorname{Dec}_\top X \to X$ induces precisely $\Phi$. It is a
  coalgebra homomorphism for general reasons.  Many examples of
  coalgebra homomorphisms from incidence coalgebras of posets to incidence 
  bialgebras of operads which fit this pattern are given in \cite{GKT-comb}. 
\end{rem}

\begin{rem}
  We stress again that all the discussion in this subsection applies 
  equally to the case of general set partitions, and in particular: {\em the
  assignment $\Phi$ constitutes a coalgebra homomorphism}
  $$
  	\Bbl_{\operatorname{lat}} \to \Bbl_{\operatorname{opd}},
  $$
  from the --opposite-- incidence coalgebra of the lattice of partitions to the
  incidence bialgebra of the (block substitution) operad of set partitions.
\end{rem}


\section{Comodule bialgebra, half-shuffle exponentials and M\"obius inversion}
\label{ssect:comod}

In the previous sections we have shown how the gap-insertion operad encodes the shuffle 
Hopf algebra approach to moment-cumulants relations, and that the 
block-substitution operad encodes the lattice viewpoint, reproducing the
M\"obius inversion.  

In this section we address the
relationship between the two operads, at the level of their 
bialgebras. We show that these constitute a pair of interacting
bialgebras in the sense of \cite{CEFM}, and more precisely a comodule
bialgebra (see \cite{Manchon18} for a review). In fact we show 
that $\Hgapnc$ is a comodule shuffle Hopf algebra over
$\Bblnc$.

We shall then explore this relationship to establish formulae for the
half-shuffle exponentials in terms of universal characters, which are
analogues of zeta functions in the three cases of free, monotone, and
boolean moment-cumulant relations. The comodule-bialgebra viewpoint
also allows to understand the inverses of these universal characters,
and relate them to M\"obius inversion.


\subsection{Comodule bialgebra}
\label{ssec:comodbi}

Briefly, the notion of comodule bialgebra is the following: for any coalgebra $B$ one has the
notion of right-$B$-comodule. If $B$ is furthermore a bialgebra, the
category $\operatorname{CoMod}(B)$ of right-$B$-comodules is naturally
monoidal, and if $B$ is a commutative bialgebra, this monoidal structure
acquires a braiding. It hence makes sense to talk about bialgebras in
$\operatorname{CoMod}(B)$. These are the comodule bialgebras. In a
nutshell a comodule bialgebra is a bialgebra with a coaction by $B$,
compatible with the bialgebra structure.

In the case at hand, the promised
right coaction $\rho: \Bgapnc \to \Bgapnc \otimes \Bblnc$ is essentially the 
coproduct $\delta$ of the commutative bialgebra $\Bblnc$, but
extended to include the empty partition, which is an element in $\Bgapnc$.
\begin{align*}
	\rho:&\left\{\begin{array}{rcl}
	\Bgapnc&\longrightarrow& \Bgapnc \otimes \Bblnc \\
	P&\longmapsto& \sum\limits_{P \leq Q} Q\otimes P/Q.
\end{array}\right.
\end{align*}

Note in particular that $\rho(\emptyset) = \emptyset \otimes 1$
(because there are no other coarser partitions of $\emptyset$ than
$\emptyset$ itself, and this has no blocks, so the fibre monomial is
just the algebra unit). Coassociativity of this coaction $\rho$ is clear 
from the fact that it is essentially the coproduct $\delta$.

\begin{theo}\label{thm:comodulebialg}
  The coaction $\rho: \Bgapnc \to \Bgapnc \otimes \Bblnc$ makes $\Bgapnc$ into a
  comodule bialgebra over $\Bblnc$.
  That is, the following relations 
  hold for all $P,Q \in \Bgapnc$:
  \allowdisplaybreaks
  \begin{align}
  &&\rho(\mathbf{1}_{\Bgapnc})&=\mathbf{1}_{\Bgapnc}\otimes \mathbf{1}_{\Bblnc}, \nonumber\\
  &&\rho(P \cdot Q)&=\rho(P)\rho(Q), \nonumber\\
  &&(\varepsilon_{\Bgapnc}\otimes \Id)\circ \rho(P)&=\varepsilon_{\Bgapnc}(P)1, \nonumber\\
  &&(\Delta_0\otimes \Id)\circ \rho(P)&=(\Id \otimes \Id\otimes m_{\Bblnc})\circ (\Id \otimes \tau \otimes \Id)\circ (\rho\otimes\rho)\circ \Delta_0(P), \label{theEq}
  \end{align}
  where $\tau:\Bblnc\otimes \Bgapnc\longrightarrow \Bgapnc\otimes \Bblnc$ is the swap map 
  $Q\otimes P \mapsto P\otimes Q$.
\end{theo}

This theorem can be proved using the general machinery of
operads and brace algebras from \cite{FoissyOperads}.  Here we wish
to present an elementary argument, exhibiting the main ingredients that
make it work.

\begin{proof}
  The most interesting part is the fourth statement, which states that 
  $\Delta_0$  is a $\Bblnc$-comodule homomorphism. Diagrammatically, this looks as follows
  $$\begin{tikzcd}
  \Bgapnc \ar[ddd, "\rho"'] \ar[rr, "\Delta_0"] 
  && \Bgapnc \otimes \Bgapnc \ar[d, "\rho\otimes\rho"] 
  \\
  && \Bgapnc \otimes \Bblnc \otimes \Bgapnc \otimes \Bblnc 
  \ar[d, "\Id\otimes\tau\otimes\Id"]
  \\
  && \Bgapnc \otimes \Bgapnc \otimes \Bblnc \otimes \Bblnc
  \ar[d, "\Id\otimes\Id\otimes m_{\Bblnc}"]
  \\
  \Bgapnc \otimes \Bblnc \ar[rr, "\Delta_0 \otimes \Id"']
  && \Bgapnc \otimes \Bgapnc \otimes \Bblnc 
  \end{tikzcd}
  $$
  The two sides of the equation \eqref{theEq} are both elements in the tensor product
  $\Bgapnc \otimes \Bgapnc \otimes \Bblnc$, that is, sums of tensors. To establish the
  equation (for each fixed partition $P$), we first establish a bijection
  between the indexing sets for the two sums, and then show that under
  this bijection, the three corresponding tensor terms agree.
  
  Step 1: the left-hand side of  \eqref{theEq}  (corresponding to the
  down-right path in the diagram) sends $P$ to a sum over the set of ways
  of first coarsening $P$ and then cutting the coarsened partition.  
  The right-hand side of the equation (given by the right-down path in the
  diagram) sends a partition $P$ to a sum indexed by the set of all ways to
  cut $P$ into upper- and lowersets and then coarsen each layer. The natural
  bijection between these to sets is most easily given by establishing
  for each set a natural bijection with the set of possible \emph{compatible
  cuts and coarsenings}, or \emph{coarse-cuts}, as we shall say for short.
  A coarse-cut on a partition $P$ consists of both a cut and a coarsening,
  required to be compatible in the sense that the cut is not allowed to
  separate two blocks that are joined in the coarsening. Here is a picture
  to illustrate this:
  \begin{center}
	\begin{tikzpicture}[scale=0.90]
	  \draw [line width=1.0pt,] (0.0,0.5)--(0.0,0.0)--(2.5,0.0)--(2.5,0.5);
	  \draw [line width=1.0pt,] (3.0,0.5)--(3.0,0.0)--(4.0,0.0)--(4.0,0.5);
	  
	  \draw [line width=1.0pt,] (0.5,0.9)--(0.5,0.4)--(1.5,0.4)--(1.5,0.9);
	  \draw [line width=1.0pt,] (2.0,0.9)--(2.0,0.4);
	  \draw [line width=1.0pt,] (3.5,0.9)--(3.5,0.4);

	  \draw [line width=1.0pt,] (1.0,1.3)--(1.0,0.8);
	  
	  \draw [dotted, line width=1.1pt, color=blue] (1.0,0.8)--(1.0,0.4);
	  \draw [dotted, line width=1.1pt, color=blue] (2.5,0.0)--(3.0,0.0);
	  
	  \draw [line width=0.3pt, color=red] 
		(-1.0,0.7) .. controls (1.2,1.0) and (-0.7,0.15) ..
		(1.0,0.15) .. controls (2.4,0.15) and (1.3,1.1) ..
		(2.0,1.1) .. controls (3.8,1.1) and (2.8,0.15) ..
		(3.5,0.15) .. controls (4.0,0.15) and (3.5,1.2) ..
		(5.0,0.7);

	  \node at (0.0,-0.3) {\tiny $1$};
	  \node at (0.5,-0.3) {\tiny $2$};
	  \node at (1.0,-0.3) {\tiny $3$};
	  \node at (1.5,-0.3) {\tiny $4$};
	  \node at (2.0,-0.3) {\tiny $5$};
	  \node at (2.5,-0.3) {\tiny $6$};
	  \node at (3.0,-0.3) {\tiny $7$};
	  \node at (3.5,-0.3) {\tiny $8$};
	  \node at (4.0,-0.3) {\tiny $9$};

	  \end{tikzpicture}
  \end{center}
  The blue dotted lines indicate which blocks are joined in the coarsening.
  Graphically, the compatibility condition says simply that the red line 
  expressing the cut is not allowed to cross the blue lines expressing
  the coarsening. The required bijections are clear.
  
  Step 2: for each fixed partition $P$, and each fixed coarse-cut, we must
  identify the corresponding two terms in the triple tensor product
  $\Bgapnc \otimes \Bgapnc \otimes \Bblnc$. For the example given
  above, these three tensor factors are
  $$
  \raisebox{-15pt}{
	\begin{tikzpicture}[scale=0.70]
	  \draw [line width=1.0pt,] (0.0,0.5)--(0.0,0.0)--(3.0,0.0)--(3.0,0.5);
	  \draw [line width=1.0pt,] (2.0,0.5)--(2.0,0.0);
	  \draw [line width=1.0pt,] (1.5,0.5)--(1.5,0.0);
	  	  
	  \draw [line width=1.0pt,] (1.0,0.9)--(1.0,0.4);

	  \node at (0.0,-0.3) {\tiny $1$};
	  \node at (1.0,-0.3) {\tiny $5$};
	  \node at (1.5,-0.3) {\tiny $6$};
	  \node at (2.0,-0.3) {\tiny $7$};
	  \node at (3.0,-0.3) {\tiny $9$};
	\end{tikzpicture}
	}
	\qquad
	\bigotimes
	\qquad
	\emptyset \ \cdot 
  \raisebox{-15pt}{
	\begin{tikzpicture}[scale=0.70]
	  \draw [line width=1.0pt,] (0.0,0.5)--(0.0,0.0)--(1.0,0.0)--(1.0,0.5);
	  \draw [line width=1.0pt,] (0.5,0.5)--(0.5,0.0);
	  \node at (0.0,-0.3) {\tiny $2$};
	  \node at (0.5,-0.3) {\tiny $3$};
	  \node at (1.0,-0.3) {\tiny $4$};
	\end{tikzpicture}
	}
	\cdot\
		\emptyset \ \cdot \
	\emptyset \ \cdot 
  \raisebox{-15pt}{
	\begin{tikzpicture}[scale=0.70]
	  \draw [line width=1.0pt,] (0.0,0.5)--(0.0,0.0);
	  \node at (0.0,-0.3) {\tiny $8$};
	\end{tikzpicture}
	}
\cdot
		\ \emptyset 
	\qquad
	\bigotimes
	\qquad
  \raisebox{-15pt}{
	\begin{tikzpicture}[scale=0.70]
	  \draw [line width=1.0pt,] (0.0,0.5)--(0.0,0.0)--(0.5,0.0)--(0.5,0.5);
	  \draw [line width=1.0pt,] (1.0,0.5)--(1.0,0.0)--(1.5,0.0)--(1.5,0.5);
	  \node at (0.0,-0.3) {\tiny $1$};
	  \node at (0.5,-0.3) {\tiny $6$};
	  \node at (1.0,-0.3) {\tiny $7$};
	  \node at (1.5,-0.3) {\tiny $9$};
	\end{tikzpicture}
	}
	\cdot
  \raisebox{-15pt}{
	\begin{tikzpicture}[scale=0.70]
	  \draw [line width=1.0pt,] (0.0,0.5)--(0.0,0.0)--(1.0,0.0)--(1.0,0.5);
	  \draw [line width=1.0pt,] (0.5,0.9)--(0.5,0.4);
	  \node at (0.0,-0.3) {\tiny $2$};
	  \node at (0.5,-0.3) {\tiny $3$};
	  \node at (1.0,-0.3) {\tiny $4$};
	\end{tikzpicture}
	}
\cdot
  \raisebox{-15pt}{
	\begin{tikzpicture}[scale=0.70]
	  \draw [line width=1.0pt,] (0.0,0.5)--(0.0,0.0);
	  \node at (0.0,-0.3) {\tiny $5$};
	\end{tikzpicture}
	}
\cdot
  \raisebox{-15pt}{
	\begin{tikzpicture}[scale=0.70]
	  \draw [line width=1.0pt,] (0.0,0.5)--(0.0,0.0);
	  \node at (0.0,-0.3) {\tiny $8$};
	\end{tikzpicture}
	}
$$
  as we shall now see.
  
  The first tensor factor (in $\Bgapnc$) is simply the coarsening of the
  lower layer.  Indeed, this is obtained via the left-hand side  by first coarsening, 
  and then cutting the coarse partition. It is obtained via the right-hand side  by 
  first cutting, and then coarsening the lower layer. So the first tensor
  factor matches up as required.
  
  The second tensor factor (also in $\Bgapnc$) is described as the upperset 
  monomial of the coarse version of the cut. This immediately matches
  the description obtained from the left-hand side. Via the composite 
  map of the right-hand side,
  the second tensor factor arises from looking 
  first in the upper layer (that's a monomial, not a single partition), 
  and then coarsening each of the factors in the monomial. Since the
  lowerset partition and its coarsening have precisely the same elements 
  in $[n]$, the factors in the uppersets also have the same elements.
  On the left-hand side  the coarsening of them takes place before cutting, and
  in the right-hand side  the coarsening takes place after cutting, and clearly the
  order of these two steps does not affect the result, so also the second 
  tensor factor matches up as required.
  
  Finally the third factor (which belongs to $\Bblnc$) has nothing to do with
  the cut: it is simply the fibre of the total coarsening (as in
  Subsection~\ref{subsec:Phi}, i.e.~for each block in the coarse partition,
  list the corresponding finer partition).
  What the equation says for this tensor factor (and which is clearly true)
  is that to compute the fibre of the whole coarsening (without even taking
  the cut into account) can be done in two steps: first compute fibres of
  the blocks in the bottom layer, then compute fibres of the blocks in the
  upper layer, and then join the result (concatenation, multiplication of
  monomials). (A subtlety to note here is that the monomial in the right-hand
  tensor factor of $\Bgapnc \otimes \Bgapnc$ may contain empty partitions. But we 
  have $\rho(\emptyset) = \emptyset \otimes 1$ (where $1$ is the algebra 
  unit). This is just to say that the fibre of the trivial coarsening of
  $\emptyset$ is the trivial monomial. Put in another way, since the
  empty partition has no blocks, there are no fibres.  The 
  $\emptyset$-factors therefore disappear (what the third tensor factor $\Bblnc$
  is concerned).)  
\end{proof}

\begin{theo}\label{thm:rightact-unshuffle}
  The right coaction $\rho : \Hgapnc \to \Hgapnc \otimes \Bblnc$  
  makes $\Hgapnc$ an unshuffle (also called codendriform)
Hopf algebra in the category of right comodules over $\Bblnc$.
That is, the comodule Hopf algebra structure and the unshuffle Hopf 
algebra structure are compatible, meaning that
for all $P \in \Hgapnc_+$ we have
\begin{align*}
&&(\Delta_\prec\otimes \Id)\circ \rho(P)&=(\Id \otimes \Id\otimes m_\Bblnc)\circ (\Id \otimes \tau \otimes \Id)\circ (\rho\otimes\rho)\circ \Delta_\prec(P),\\
&&(\Delta_\succ\otimes \Id)\circ \rho(P)&=(\Id \otimes \Id\otimes m_\Bblnc)\circ (\Id \otimes \tau \otimes \Id)\circ (\rho\otimes\rho)\circ \Delta_\succ(P),
\end{align*}
where $\tau:\Bblnc\otimes \Hgapnc\longrightarrow \Hgapnc\otimes \Bblnc$ is the swap map 
$Q\otimes P \mapsto P\otimes Q$.
\end{theo}

\begin{proof}
  It is clear that the coaction $\rho: \Bgapnc \to \Bgapnc \otimes \Bblnc$ restricts 
  to a coaction $\rho: \Hgapnc \to \Hgapnc\otimes \Bblnc$, and that this makes $\Hgapnc$ a 
  comodule Hopf algebra. To check the compatibility with the splitting of the unshuffle coproduct 
  into half-unshuffles, it remains to notice that the coaction works by 
  refinement of blocks, and that this does not interfere with whether or 
  not
  the block containing the element $1$ belongs to the lowerset of a cut.
\end{proof}


\subsection{Four bijections from infinitesimal characters to characters}
\label{ssec:fourbij}

We focus in this subsection on the structure of noncrossing partitions and
relations between moments and free cumulants, having in mind applications
of the previous results to free, boolean and monotone probabilities (on the
latter two in the context of the present article, see \cite{EFP18}). We
show indeed that the comodule bialgebra structure of $\Hgapnc$ over $\Bblnc$ leads to
a new approach to these phenomena.

Recall first from \cite{EFP17} that groups of characters on unshuffle (also
called codendriform) Hopf algebras $H$ have a rather rich structure with,
in particular, three exponentials and logarithms putting in bijection the
Lie algebra of infinitesimal characters of $H$ and the group of characters.
When applied in the context of free probability, these maps together with
associated properties of the group and Lie algebras lead to a new
understanding of the three natural families of cumulants in this context
(free, monotone and boolean), of their relationships, and of related
notions such as additive convolution of free distributions (see also
\cite{EFP18}). In the present section we explore these bijections using the action
of the monoid $M_\Bblnc$ of
characters of $\Bblnc$, and relate the resulting formulae with
the M\"obius function of the
lattice of noncrossing partitions.

The product in $M_\Bblnc$, induced by the coproduct
$\delta$, is denoted by $*$. We denote by $G_\Bblnc$ the group of invertible
elements of $M_\Bblnc$.

\begin{lemma} \label{lemmeinversible}
  Let $\phi \in M_\Bblnc$. Then $\phi\in G_\Bblnc$ if and only if $\phi(I_n)\neq 0$ for all $n\geq 1$.
\end{lemma}

\begin{proof} 
  Let $\phi\in G_\Bblnc$ and $\psi$ be the inverse of $\phi$ for the convolution
  $*$. For any $n\geq 1$, since $I_n$ is a group-like element, we have
  $1=\phi*\psi(I_n)=\phi(I_n)\psi(I_n)$, and hence $\phi(I_n)\neq 0$.

Conversely, the bialgebra $\Bblnc$ is graded by the number of blocks (minus one), and its
component of degree zero has a basis of group-like elements, namely the
products of $I_n$'s. Hence, $\Bblnc'=\Bblnc[I_1^{-1},\dots,I_n^{-1},\dots]$ is a Hopf
algebra, with the extension $\delta'$ of the coproduct $\delta$ defined by
$\delta'(I_n^{-1})=I_n^{-1}\otimes I_n^{-1}$. If $\phi\in M_\Bblnc$ satisfies
$\phi(I_n)\neq 0$ for all $n\geq 1$, it can be extended as a character
$\phi'$ to $\Bblnc'$. Denote by $\psi'$ the inverse of $\phi'$ in the group of
characters of $\Bblnc'$ and by $\psi$ its restriction to $\Bblnc$. 
For any $P \in \Bblnc$:
\[
	(\phi*\psi)(P)
	=(\phi\otimes \psi)\circ \delta(P)
	=(\phi'\otimes \psi')\circ 
	\delta'(P)=\varepsilon_{\Bblnc'}(P)=\varepsilon_\Bblnc (P).
\]
Similarly, $(\psi*\phi)(P)=\varepsilon_\Bblnc (P)$, so $\phi\in G_\Bblnc$. \end{proof}

We also write $\g_\Hgapnc$ for the Lie algebra of infinitesimal characters of $\Hgapnc$
(linear forms that vanish on $(\Hgapnc_+)^2$; the Lie bracket is induced 
by the associative convolution product of linear forms). 
We furthermore write $G_\Hgapnc$ for the group of characters of $\Hgapnc$. Its product, induced by
$\Delta$, is denoted by $\star$.

As already observed, the coproducts $\Delta_\prec$ and $\Delta_\succ$
induce a shuffle (also called dendriform) algebra structure on the dual
$\Hgapnc^*_+$. It is extended to products of elements of $\Hgapnc^*$ with the unit
$\varepsilon_\Hgapnc$ of the convolution product as follows: for any $f\in
\Hgapnc^*_+$,
\begin{align*}
	f\prec \varepsilon_\Hgapnc&=f,&f\succ \varepsilon_\Hgapnc&=0,\\
	\varepsilon_\Hgapnc \prec f&=0,&\varepsilon_\Hgapnc\succ f&=f.
\end{align*} 

The coaction $\rho$ induces a right action of $M_\Bblnc$ on linear forms 
on $\Hgapnc$, denoted by
$\curvearrowleft$: for $\alpha$ an arbitrary linear
form on $\Hgapnc$ and $\phi\in M_\Bblnc$, it is given by
$$
	\alpha \curvearrowleft \phi:=(\alpha\otimes\phi)\circ \rho.
$$
It restricts to an action by group endomorphisms on $G_\Hgapnc$ and by Lie
algebra endomorphisms on $\g_\Hgapnc$.

Note that as an algebra $\Bblnc$ is the abelianization of $\Hgapnc$, so, as sets,
$M_\Bblnc$ and $G_\Hgapnc$ can be identified. Through this identification, we see that
the action $\curvearrowleft$ coincides with the product $*$. In the following, to avoid
ambiguities, when $\phi\in G_\Hgapnc$, we write $\overline\phi$ for the
corresponding element of $M_\Bblnc$. Conversely, when $\psi\in M_\Bblnc$, we write
$\widehat\psi$ for the corresponding element of $G_\Hgapnc$ (so that in particular
for $\phi\in G_\Hgapnc$ and $\gamma\in M_\Bblnc$, we have $\overline{\phi \curvearrowleft
\gamma}=\overline{\phi}\ast\gamma$).\\

First, since $\Hgapnc$ is freely generated by $\latNCP$ as an algebra, characters and
infinitesimal characters on $\Hgapnc$ are entirely determined by their
restriction to $\latNCP$, and the following maps are bijections:
\begin{align*}
	\theta_{\g}:&\left\{\begin{array}{rcl}
	\g_\Hgapnc&\longrightarrow&\Vect(\latNCP)^*\\
	\kappa&\longmapsto&\kappa_{\mid \Vect(\latNCP)},
	\end{array}\right.& \theta_G:&\left\{\begin{array}{rcl}
G_\Hgapnc&\longrightarrow&\Vect(\latNCP)^*\\
\phi&\longmapsto&\phi_{\mid \Vect(\latNCP)}.
\end{array}\right.&
\end{align*}
The composite bijection
$\Theta=\theta_G^{-1}\circ \theta_{\g}:\g_\Hgapnc\stackrel{\sim}\longrightarrow G_\Hgapnc$
sends any $\kappa\in \g_\Hgapnc$ to the unique character $\Theta(\kappa) : \Hgapnc \to
\K$ such that it agrees with $\kappa$ on degree-$1$ monomials. That is,
$\kappa(P)=\Theta(\kappa)(P)$ (for any $P\in \latNCP$).

What cumulants are concerned we shall be especially interested in those
elements in $\g_\Hgapnc$ that vanish on $\latNCP(k,n)$ for $k\geq 2$. 
In particular, we shall consider $e=\Theta^{-1}(\widehat{\varepsilon}_\Bblnc)$.
In other words, $e$ is the unique infinitesimal character of $\Hgapnc$ such that  
for any $P\in \latNCP(k,n)$, we have $e(P)=\delta_{k,1}$. As a consequence, for all $P\in \latNMP(k,n)$:
\begin{align*}
	 e(P)=\delta_{k,1}.
\end{align*}

The first natural bijection $\Theta: \g_\Hgapnc \to G_\Hgapnc$ 
rewrites as follows, in terms of the coaction of $\Bblnc$ on $\Hgapnc$:

\begin{prop}
\label{lemmae} 
  For any $\kappa\in \g_\Hgapnc$, 
  put $K := \overline{\Theta(\kappa)} \in M_\Bblnc$.
  Then we have
$$
	 \kappa = e \curvearrowleft K.
$$
\end{prop}

\begin{proof}
For any $P\in \latNCP$ we have
\begin{align*}
	e\curvearrowleft K(P)
	&=\sum_{P\leq Q}e(Q)K(P/Q)
	  =K(P)=\kappa(P).
\end{align*}
Since $\kappa$ and $e \curvearrowleft K$
therefore agree on all noncrossing partitions, and are both 
infinitesimal characters over $\Hgapnc$, they must be equal.
\end{proof}

\begin{prop}[Left half-shuffle exponential isomorphism]
  Let $\kappa \in \g_\Hgapnc$. There exists a unique $\phi\in \Hgapnc^*$, such that
  $\phi=\varepsilon_\Hgapnc+\kappa\prec \phi$; moreover, $\phi \in G_\Hgapnc$.
  Conversely, let $\phi\in G_\Hgapnc$. There exists a unique $\kappa\in \Hgapnc^*$,
  such that $\phi=\varepsilon_\Hgapnc+\kappa\prec \phi$; moreover, $\kappa \in
  \g_\Hgapnc$. In particular, the following map is a bijection:
\begin{align*}
	\Expl :&\left\{\begin{array}{rcl} 
	          \g_\Hgapnc&\longrightarrow&G_\Hgapnc\\
              \kappa&\longmapsto&\phi\ \mbox{ such that }\ 
			  \phi =\varepsilon_\Hgapnc+\kappa\prec \phi.
\end{array}
\right.
\end{align*}
\end{prop}

\begin{proof}
  Let us define $\phi(P)$ for any $P \in \latNMP(k,n)$ by induction on $n$.
  If $n=0$, then $P=1$ and $\phi(1)=1$. If $n\geq 1$, we put
  $\phi(P)=\kappa(P)+\kappa(P'_\prec)\phi(P''_\prec)$ (with Sweedler-like notation 
  for $\Delta_\prec$ as in \eqref{Dprec}). We obtain in this way a map
  $\phi\in \Hgapnc^*$, such that $\phi=\varepsilon_\Delta+\kappa\prec \phi$. Let
  $x\in \latNMP(k,m)$ and $y\in \latNMP(l,n)$; let us prove that
  $\phi(xy)=\phi(x)\phi(y)$ by induction on $m+n$. If $m=0$ or $n=0$, then
  $x=1$ or $y=1$ and the result is immediate. We now assume that $m,n\geq
  1$. Then, as $\kappa$ is an infinitesimal character, and by the induction
  hypothesis:
\begin{align*}
	\phi(PQ)	&=\kappa(PQ)\phi(1)+\kappa(P)\phi(Q)+\kappa(PQ')\phi(Q'')+\kappa(P'_\prec Q)\phi(P''_\prec)\\
			&\quad +\kappa(P'_\prec)\phi(P''_\prec Q)+\kappa(P'_\prec Q')\phi(P''_\prec Q'')\\
		    	&=\kappa(P)\phi(Q)+\kappa(P'_\prec)\phi(P''_\prec)\phi(Q)\\
                     	&=\phi(P)\phi(Q).
\end{align*}
So $\phi \in G_\Hgapnc$. The proof of the second statement is similar and the
last is an easy consequence of the first two points.
\end{proof}

The following character on $\Hgapnc$ plays a special role: $\psi_\prec := \Expl (e).$
We shall see in Theorem~\ref{thm:psiprec} that it corresponds to the zeta function,
$\psi_\prec = \widehat \zeta$: we have $\psi_\prec(P)=1$ for every noncrossing partition $P$.

The left half-shuffle exponential isomorphism rewrites as follows.

\begin{theo}\label{thm:left}
  For any $\kappa\in \g_\Hgapnc$, put $K := \overline{\Theta(\kappa)} \in M_\Bblnc$ 
  as usual. For
  $\phi \in G_\Hgapnc$ we have
  $$
	\phi=\varepsilon_\Hgapnc+\kappa\prec \phi
	\ \Longleftrightarrow \
	\phi =\psi_\prec  \curvearrowleft K.
  $$ 
 In other words,  we have
   $$
  \Expl(\kappa) = \psi_\prec \curvearrowleft K. 
  $$
\end{theo}

\begin{proof}
  By definition, $\psi_\prec=\varepsilon_\Hgapnc+e\prec\psi_\prec$. Since $\Hgapnc$ is an
  unshuffle bialgebra in the category of right $\Bblnc$-comodules, we have:
\begin{align*}
	\psi_\prec \curvearrowleft K
	&=(\varepsilon_\Hgapnc+e\prec \psi_\prec) \curvearrowleft K\\
	&=\varepsilon_\Hgapnc \curvearrowleft K+(e \curvearrowleft K)\prec (\psi_\prec  \curvearrowleft K)\\
	&=\varepsilon_\Hgapnc+(e \curvearrowleft K)\prec (\psi_\prec \curvearrowleft K)\\
	&=\varepsilon_\Hgapnc+\kappa \prec (\psi_\prec \curvearrowleft K)
	& \text{(by \ref{lemmae}).}
\end{align*}
This shows that $\psi_\prec  \curvearrowleft K$ satisfies the equation 
characterizing $\phi$, and hence establishes the bi-implication.
To establish the last equation, note that the character
$\Expl(\kappa)$ is by definition the unique solution $\phi$ to the
equation $\phi =\varepsilon_\Hgapnc + \kappa\prec \phi$. But we have just 
seen that $\psi_\prec \curvearrowleft K$ satisfies this equation.
\end{proof}

\begin{prop}
\label{prop:doubletensor}
Let us consider a family of scalars $(k_n)_{n\geq 1}$. 
We define the infinitesimal character $\kappa$ on $\Hgapnc$ by:
\begin{align*}
	\kappa(P)	&=\begin{cases}
		  k_n & \mbox{ if }P=I_n,\: n\geq 1,\\
		      0 & \mbox{ otherwise}.
\end{cases}
\end{align*}
Then Theorem~\ref{thm:left} gives, for any $P\in \latNP$, that 
$\Expl(\kappa)(P) = \prod_{\pi\in P} k_{\sharp \pi}$ and the moments
$$
	 (\Expl(\kappa)  \curvearrowleft \zeta) (J_n) 
	=\sum_{Q\in \latNCP(n)} \prod_{\pi\in Q}k_{\sharp \pi}.
$$
\end{prop}

\begin{proof}
For any $\kappa \in \g_\Hgapnc$, put $K := \overline{\Theta(\kappa)} \in M_\Bblnc$.  
For $P\in \latNP$ we compute, by Theorem~\ref{thm:left}:
\begin{align*}
	\Expl(\kappa)(P)&=(\psi_\prec \curvearrowleft \overline{\Theta(\kappa)})(P)
	=\sum_{P \leq Q} \overline{\Theta(\kappa)}(P/Q)
	=\Theta(\kappa)(P)+0
	=\prod_{\pi\in P} k_{\sharp \pi}.
\end{align*}
From this it follows that
$$
	\Expl(\kappa)  \curvearrowleft \zeta
	=(\psi_\prec \curvearrowleft \overline{\Theta(\kappa)}) \curvearrowleft \zeta 
	=\psi_\prec \curvearrowleft ( \overline{\Theta(\kappa)} * \zeta)  
$$
(by associativity of the action), and therefore
\allowdisplaybreaks
\begin{align*}
	 (\Expl(\kappa)  \curvearrowleft \zeta) (J_n) 
	 &= \sum_{Q\in \latNCP(n)} \psi_\prec(Q)( \overline{\Theta(\kappa)} * \zeta)(J_n/Q)\\
	 &= \sum\limits_{P=\{\pi_1,\dots,\pi_k\}\in \latNCP(n)}
	 ( \overline{\Theta(\kappa)} * \zeta)(J_{\sharp\pi_1} \cdots J_{\sharp\pi_k})\\
	&= \sum\limits_{P=\{\pi_1,\dots,\pi_k\}\in \latNCP(n)}	 
	( \overline{\Theta(\kappa)} * \zeta)(J_{\sharp\pi_1})  
	\cdots  
	( \overline{\Theta(\kappa)} * \zeta)(J_{\sharp\pi_k})\\	
	&= \sum_{P\in \latNCP(n)} \prod_{\pi\in P}k_{\sharp \pi}.
\end{align*}
\end{proof}

\begin{prop}
\label{prop:sumofall}
Let us consider a family of scalars $(k_n)_{n\geq 1}$. 
We define the infinitesimal character $\kappa$ on $\Hgapnc$ by:
\begin{align*}
\kappa(P)&=\begin{cases}
k_n 	& \mbox{ if }P=J_n=\{\{1\},\dots,\{n\}\},\: n\geq 1,\\
0 	& \mbox{ otherwise}.
\end{cases}
\end{align*}
Then for any $P\in \NCP(k,n)$:
\begin{align*}
\Expl(\kappa)(P) = (\psi_\prec \curvearrowleft \overline{\Theta(\kappa)})(P) 
&=\sum_{P\leq Q} \overline{\Theta(\kappa)}(P/Q)
=\begin{cases}
	0 											& \mbox{ if }k<n,\\
\displaystyle \sum_{Q\in \latNCP(n)} \prod_{\pi\in Q}k_{\sharp \pi} 	& \mbox{ if }k=n.
\end{cases}
\end{align*}
\end{prop}

\begin{proof}
  For any $P \in \NCP(k,n)$, $k<n$, it is clear that $P/Q$ contains
  blocks of size bigger than one, i.e., $P/Q$ is not a monomial of
  forests of sticks. This implies that $\overline{\Theta(\kappa)}(P/Q)
  =0$. For $P=J_n\in \NCP(n,n)$ \allowdisplaybreaks
\begin{align*}
	 \Expl(\kappa) (J_n) 
	 &= \sum_{Q\in \latNCP(n)} \psi_\prec(Q)\overline{\Theta(\kappa)}(J_n/Q)\\
	 &= \sum\limits_{P=\{\pi_1,\dots,\pi_k\}\in \latNCP(n)}
	 \overline{\Theta(\kappa)}(J_{\sharp\pi_1} \cdots J_{\sharp\pi_k})\\
	&= \sum\limits_{P=\{\pi_1,\dots,\pi_k\}\in \latNCP(n)}	 
	\overline{\Theta(\kappa)}(J_{\sharp\pi_1})  
	\cdots  
	\overline{\Theta(\kappa)} (J_{\sharp\pi_k})\\	
	&= \sum_{P\in \latNCP(n)} \prod_{\pi\in P}k_{\sharp \pi}.
\end{align*}
\end{proof}

\begin{rem}
  Proposition \ref{prop:sumofall} reveals a nice connection to the approach
  to free moment-cumulant relations presented in \cite{EFP15}. Indeed, let
  $(A,\varphi)$ be a noncommutative probability space and consider the sub
  Hopf algebra $\Hgapnc^{|A}$ of forests of sticks, as in 
  Corollary~\ref{coro:subHopfalg} but decorated by elements from $A$. Since 
  a monomial of $A$-decorated forests of sticks is the same thing as a list
  of $A$-words, it is easy to see that the Hopf algebra $\Hgapnc^{|A}$
  is isomorphic to the double tensor algebra over $A$ (with the coproduct defined in
  \cite{EFP15}). Proposition \ref{prop:sumofall} gives another proof of
  the free moment-cumulant relations deduced from a shuffle fixpoint
  equation.
\end{rem}

We now consider instead the right half shuffle:

\begin{prop}[Right half-shuffle exponential isomorphism]
  Let $\kappa \in \g_\Hgapnc$. There exists a unique $\phi\in \Hgapnc^*$, such that
  $\phi=\varepsilon_\Hgapnc+\phi\succ \kappa$; moreover, $\phi \in G_\Hgapnc$.
  Conversely, let $\phi\in G_\Hgapnc$. There exists a unique $\kappa\in \Hgapnc^*$,
  such that $\phi=\varepsilon_\Hgapnc+\phi\succ \kappa$; moreover, $\kappa \in
  \g_\Hgapnc$. The following map is therefore a bijection:
\begin{align*}
		\Expr:  &\left\{\begin{array}{rcl}
	\g_\Hgapnc &\longrightarrow&G_\Hgapnc\\
		\kappa&\longrightarrow&\phi\ \mbox{ such that }\ \phi=\varepsilon_\Hgapnc+\phi\succ \kappa.
\end{array}\right.
\end{align*}
\end{prop}

The right half-shuffle exponential isomorphism rewrites as follows. Let us put
$$
	\psi_\succ := \Expr(e).
$$
\begin{theo}
  
  For any $\kappa\in \g_\Hgapnc$, put 
  $K := \overline{\Theta(\kappa)} \in M_\Bblnc$ 
  as usual.
  Then
  for $\phi\in G_\Hgapnc$ we have:
\begin{align*}
	\phi=\varepsilon_\Hgapnc+\phi\succ \kappa
	&\Longleftrightarrow \phi=\psi_\succ  \curvearrowleft K.
\end{align*}  \end{theo}

Let us finally consider the classical exponential bijection from the Lie 
algebra $\g_\Hgapnc$ to the group $G_\Hgapnc$:
\begin{align*}
	\exp_\star:&\left\{\begin{array}{rcl} \g_\Hgapnc	&\longrightarrow	&G_\Hgapnc\\
	\kappa	&\longmapsto					&\displaystyle \exp_\star(\kappa)
										   =\sum_{k=0}^\infty \frac{\kappa^{\star n}}{n!}.
\end{array}\right.
\end{align*}

  Put
  $$
  \psi_\star := \exp_\star(e).  
  $$

\begin{prop}[Exponential isomorphism] 
    For any $\kappa\in \g_\Hgapnc$, put 
  $K := \overline{\Theta(\kappa)} \in M_\Bblnc$ 
  as usual.
  Then for  $\phi\in G_\Hgapnc$ we have:
\begin{align*}
\phi=\exp_\star(\kappa)&\ \Longleftrightarrow \ \phi=\psi_\star
\curvearrowleft K.
\end{align*} \end{prop}

\begin{proof}
  Indeed:
  \begin{align*}
  \psi_\ast \curvearrowleft K&=\exp_\star(e) \curvearrowleft 
  K=\exp_\star(e \curvearrowleft K)=\exp_\star(\kappa)=\psi_\star(\kappa),
  \end{align*}
  the third equality by Proposition~\ref{lemmae}.
\end{proof}


\subsection{Description of the universal maps $\psi_\prec$, $\psi_\succ$ and $\psi_\star$}

In the previous subsection, we have seen that the maps $\psi_\prec=\Expl(e)$,
$\psi_\succ=\Expr(e)$ and $\psi_\star=\exp_\star(e)$ encode the three shuffle exponentials
from $\g_\Hgapnc$ to $G_\Hgapnc$. In this subsection, we compute them explicitly.

\bigskip

The set of partitions (or noncrossing partitions) forms a monoid 
under ordinal sum, denoted $\ordsum$, with neutral element the
empty partition. On the underlying linear orders it is just
ordinal sum, $[n] \ordsum [m] = [n{+}m]$, and the partition
structure on $[n{+}m]$ is induced via the sum injections
$[n] \to [n{+}m] \leftarrow [m]$ in the canonical way.

A partition is called \emph{irreducible} if it cannot be
written as the ordinal sum of two non-empty partitions.
A noncrossing partition $P \in \latNCP(n)$ is irreducible if 
and only if $1$ and $n$ belong to the same block.  
In general, the block containing the element $1$ we call the \emph{base
block} and denote it $\pi_1$. Let $k$ denote the biggest element belonging
to $\pi_1$, then $\llbracket 1,k\rrbracket = \Conv(\pi_1)$ is a union of
blocks of $P$, and $P_{|\llbracket 1,k\rrbracket}$ is an irreducible
partition, which we call $E_1$. Now we have
$$
	P = E_1 \ordsum \tilde P ,
$$
where $\tilde P = {E_1}^c$. Proceeding now the same way with $\tilde P$ and
iterating, it is clear that every noncrossing partition $P$ splits uniquely
into an (iterated) ordinal sum of irreducible noncrossing partitions,
called \emph{irreducible components}
$$
	P = E_1 \ordsum \cdots \ordsum E_r .
$$

(For general partitions, the notion of irreducible is slightly more
complicated, as, for example, also the partition with crossing
\begin{tikzpicture}
  \draw (0,0.2)--(0,0)--(0.2,0)--(0.2,0.2);
  \draw (0.1,0.14)--(0.1,-0.06)--(0.3,-0.06)--(0.3,0.14);
\end{tikzpicture} 
is clearly irreducible. In general a partition is irreducible precisely 
when
its noncrossing closure is irreducible. The unique splitting of
an arbitrary partition into irreducible components is similar, but will 
not be needed here.)

We shall say that a noncrossing partition is \emph{boolean}
if its irreducible components are precisely its blocks.
In other words, it is an (iterated) ordinal sum of single-block
partitions.
In particular, in a boolean partition there is no nesting of blocks.

\begin{defi}[Monotone compositions] 
  Recall that a noncrossing composition is a noncrossing partition $P$
  equipped with a numbering of its blocks.  Such a numbering is called a
  \emph{heap ordering} if it preserves the order $\Pto$ (that is, the
  numbering map $(P,\Pto) \longrightarrow [k]$ is monotone ($k$ is the
  number of blocks)), meaning that if one block is nested inside another
  then it has a higher number. We denote by $\ho(P)$ the
  cardinality of the set of heap orderings of a noncrossing partition $P$.
  
  A \emph{monotone composition} is by definition a noncrossing
  partition with a chosen heap ordering. 
\end{defi}

\begin{theo}
  \label{thm:psiprec}
  \begin{enumerate}
\item For any $P\in \latNP$, we have $\psi_\prec(P)=1$. That is, 
$\psi_\prec$ corresponds to the zeta function $\zeta \in \Bblnc$.

\item For any $P\in \latNP$, we have $\psi_\succ(P)=\begin{cases}
1 & \mbox{ if $P$ is boolean},\\
0 & \mbox{ otherwise}.
\end{cases}$

\item Let $P\in \latNCP(k,n)$. 
 Then $\psi_\star(P)=\displaystyle \frac{\ho(P)}{k!}$.
\end{enumerate} \end{theo}

Recall from Definition~\ref{gap-monomial} that if $(L,U)$ is a cut of $P$, then
$\rmst{U}$ denotes the reduced gap monomial of the cut, the monomial 
obtained as the partitions belonging to $U$ that appear in the gaps of $L$.
\begin{proof}
  Let $P\in \latNCP(k,n)$. If $\pi_1$ is base block of $P$, write $\tilde 
  P := P_{|\Conv(\pi_1)^c}$ as in the preliminary discussion above. 
  In particular, if
  $k=1$, then $\tilde{P}=\emptyset$, and
  $\psi_\prec(P)=\psi_\succ(P)=\psi(\tilde{P})=1$.

  1. If $(L,U)$ is a cut of $P$ such that $e\big(\st{L}\big)\neq 0$, then
  $U$ contains all the blocks of $P$ but one. By definition of $\psi_\prec$:
\begin{align*}
	\psi_\prec(P)&=\sum_{\substack{(L,U) \in \cut(P),\\ 1\in L}} 
	e\big(\st{L}\big) \psi_\prec\big(\rmst{U}\big)=\psi_\prec(\tilde{P}).
\end{align*} 
The result follows by an easy induction on $k$.

2. If $U$ is an upperset of $P$ such that $e\big(\rmst{U}\big)\neq 0$, then,
by definition of $e$, $U$ contains only one block of $P$. 
The block $\pi_1$ of $P$ containing $1$ is an upperset of $P$ if and only if 
it is a irreducible component of $P$. Hence:
\begin{align*}
	\psi_\succ(P)&=\sum_{\substack{(L,U) \in \cut(P),\\ 1\in L}} 
	\psi_\succ(\st{L}) e(\rmst{U})
	=
	\begin{cases}
	\psi_\succ(\tilde{P}) 	& \mbox{ if $\pi_1$ is irreducible},\\
	0 				& \mbox{ otherwise.}
	\end{cases} 
\end{align*} 
The result follows by an easy induction on $k$.

3. An \emph{upper-block} is a block which is also an upperset;
in other words, it is a maximal element in the poset $(P,\Pto)$
(that is, a block that has no nested blocks inside it).
We write $(L,U) \in \cut'(P)$ for the situation where the
upperset $U$ consists of a single block (which is hence an upper-block).
If
$\sigma: (P, \Pto) \longrightarrow [k]$ is a heap ordering of $P$,
then the block $\sigma^{-1}(k)$ is necessarily an upper-block. This yields
a recursive calculation of the heap-order numbers:
\begin{align}\label{eq:ho}
	\ho(P)&=\sum_{(L,\pi) \in \cut'(P)} \ho(L),
\end{align}
which is used in the following proof of Item 3 by induction on $k$. If $k=1$, then
$\psi_\star(P)=e(P)+0=1$ and we are done. Assume the result
at rank $k-1$, and consider a partition $P$ with $k$ blocks.
With Sweedler notation for the iterated coproducts
($P\longmapsto P^{(1)}\otimes\dots\otimes P^{(l)}$): \allowdisplaybreaks
\begin{align*}
	\psi_\star(P)	
	&=\sum_{l=1}^\infty\frac{1}{l!} \, e(P^{(1)})\cdots e(P^{(l)})
	\\
	&=\frac{1}{k!} e(P^{(1)})\cdots e(P^{(k)})+0
	\\
	&=\sum_{(L,U) \in 
	\cut'(P)}\frac{1}{k!}e(\st{L}^{(1)})\cdots
		e(\st{L}^{(k-1)})e(\st{U})
		\\
	&=\sum_{(L,U) \in \cut'(P)}\frac{1}{k!}e(\st{L}^{(1)})
	\cdots e((\st{L})^{(k-1)})
	\\
	&=\frac{1}{k}\sum_{(L,U) \in \cut'(P)}\psi_\star(\st{L})
	\\
	&=\frac{1}{k}\sum_{(L,U) \in 
	\cut'(P)}\frac{\ho(\st{L})}{(k-1)!} 
	& \makebox[0pt]{\text{(by induction hyp.)}}
	\\
	&=\frac{\ho(P)}{k!}  & \makebox[0pt]{\text{(by 
	\eqref{eq:ho})\phantom{xxxxxxxx.}}}
\end{align*}
as required.
\end{proof}


\subsection{On the inverses of the universal maps}

By Lemma~\ref{lemmeinversible}, $\psi_\prec$, $\psi_\succ$ and $\psi_\star$
are invertible elements of the monoid $(M_\Bblnc, \ast)$. In this subsection, we
investigate the behaviour of these inverses. The following proposition
emphasizes once again the boolean character of $\psi_\succ$ (recall indeed 
from
\cite{EFP18} that the right half-shuffle exponential is associated to
boolean cumulants in free probability).

\begin{prop}
For any $P\in \latNCP(k,n)$:
\begin{align*}
\psi_\succ^{*-1}(P)&
=\begin{cases}
(-1)^{k+1} & \mbox{ if $P$ is boolean},\\
0 & \mbox{ otherwise};
\end{cases}\\
\psi_\succ^{*-1}*\psi_\prec(P)&
=\begin{cases}
1 & \mbox{ if $P$ is irreducible},\\
0 & \mbox{ otherwise},
\end{cases}\\
\psi_\prec^{*-1}*\psi_\succ(P)&
=\begin{cases}
(-1)^{k+1} & \mbox{ if $P$ is irreducible},\\
0 & \mbox{ otherwise}.
\end{cases}
\end{align*}
\end{prop}

\begin{proof}
Let us assume that $P$ is not boolean and let us prove that $\psi_\succ^{*-1}(P)=0$ by induction on $k$.
If $k=1$, there is nothing to prove. If $k\geq 2$, then for any $Q>P$, 
either $Q$ is not boolean with $l<k$ blocks, and hence
vanishes under $\psi_\succ^{*-1}$, or one of the components of $P/Q$ 
is not boolean, and hence vanishes under $\psi_\succ$. This shows that
\begin{align*}
	\psi_\succ^{*-1}*\psi_\succ(P)
	&=\psi_\succ^{*-1}(P)\psi_\succ(P/P)+0=\psi_\succ^{*-1}(P)=\varepsilon_\Bblnc(P)=0.
\end{align*}

Let us assume that $P$ is boolean and let us show that $\psi_\succ^{*-1}(P)=(-1)^{k+1}$ by induction on $k$. If $k=1$,
then $\psi_\succ^{*-1}(P)=\psi_\succ(P)^{-1}=1$. If $k\geq 2$:
\begin{align*}
	\psi_\succ^{*-1}*\psi_\succ(P)
	&=\sum_{Q\geq P} \psi_\succ^{*-1}(Q)\psi_\succ(P/Q)\\
	&=\sum_{Q\geq P} \psi_\succ^{*-1}(Q)\\
	&=\psi_\succ^{*-1}(P)+\sum_{l=1}^{k-1}(-1)^{l+1}\sharp\{(c_1,\ldots,c_l)\in \N_{>0}^l \mid c_1+\cdots+c_l=k\}\\
	&=\psi_\succ^{*-1}(P)+\sum_{l=1}^{k-1}(-1)^{l+1}\binom{k-1}{k-l}\\
	&=\psi_\succ^{*-1}(P)+(-1)^{k+1}\sum_{l=1}^{k-1} (-1)^l\binom{k-1}{l}\\
	&=\psi_\succ^{*-1}(P)-(-1)^{k+1}\\
	&=\varepsilon_\Bblnc(P)\\
	&=0.
\end{align*}
Let $E_1,\ldots,E_k$ be the irreducible components of $P$, 
$n_i$ their respective cardinality, and $P'$ the noncrossing partition 
whose blocks are
$\llbracket n_1+\cdots+n_{i-1}+1,n_1+\cdots+n_i\rrbracket$, $1\leq i \leq k$. Then:
\begin{align*}
	\psi_\succ^{*-1}*\psi_\prec(P)
	&=\sum_{P\leq Q} \psi_\succ^{*-1}(Q)\\
	&=\sum_{P\leq Q'} \psi_\succ^{*-1}(Q)+0\\
	&=\sum_{l=1}^k (-1)^{l+1}\sharp\{(c_1,\ldots,c_l)\in \N_{>0}^l 
	\mid  c_1+\cdots+c_l=k\}\\
	&=(-1)^{k+1}\sum_{l=0}^{k-1} (-1)^l\binom{k-1}{l}\\
	&=\begin{cases}
		1 & \mbox{ if }k=1,\\
		0 & \mbox{ otherwise.}
	\end{cases} 
\end{align*} 

Let $\alpha=\psi_\succ^{*-1}*\psi_\prec$,  and define $\beta \in M_\Bblnc$ by:
\begin{align*}
	\beta(P)&=\begin{cases}
	(-1)^{k+1} 		&\mbox{ if $P$ is irreducible},\\
			0 	&\mbox{ otherwise} 
			\end{cases}
			\qquad \text{ for } P\in \latNCP(k,n).
\end{align*}
For any $Q \in \latNP$ we denote by $\bl(Q)$ the number of blocks of $Q$. Let $P\in \latNCP(k,n)$. 
\begin{align*}
	\beta*\alpha(P)&=\sum_{\substack{Q\geq P,\\ \mbox{\scriptsize $Q$ irreducible,}\\
	\mbox{\scriptsize the factors of $P/Q$ irreducible}}} (-1)^{\bl(Q)+1}.
\end{align*}
If $P$ is not irreducible, then for any $Q\geq P$, either $Q$ is not irreducible or one of the factors of $P/Q$ is not irreducible. Hence, in this case,
$\beta*\alpha(P)=0=\varepsilon_\Bblnc(P)$. Let us now assume that $P$ is irreducible. If $P\in \latNCP(1,n)$, then $\beta*\alpha(P)=1=\varepsilon_\Bblnc(P)$.
Let us assume that $P\in \latNCP(k,n)$, with $n\geq 2$. As $P$ is irreducible, it can be written as:
\begin{align*}
	P&=I_n\diamond (\emptyset,P_1,\ldots,P_{n-1},\emptyset),
\end{align*}
where $P_1,\ldots,P_{n-1} \in \latNP_0$. We denote by $E_{i,j}$ the irreducible 
components of $P_i$. If some $P_i$ is the empty partition, then of course it 
has no irreducible components, so there are no $E_{i,j}$ for this index 
$i$. But since $k\geq 2$, at least one of the $P_i$ is non-empty.  We 
denote by $J$ the set of pairs $(i,j)$ such that $1 \leq i \leq n-1$ and 
such that $P_i$ has at least $j$ irreducible components.  
The set $J$ thus indexes the meaningful $P_{i,j}$. There 
is a bijection
\begin{align*}
	&\left\{\begin{array}{rcl} \{Q\in \latNP \mid Q\geq P\}&\longrightarrow&\displaystyle 
	\prod_{(i,j)\in J} \{Q_{i,j}\in \latNP \mid Q_{i,j}\geq P_{i,j}\}\times \{0,1\}\\
	Q&\longmapsto&(Q_{i,j},\epsilon_{i,j})_{i,j\in J},
\end{array}\right.
\end{align*}
where for any $(i,j)\in J$, $Q_{i,j}$ is the noncrossing partition
formed by the blocks of $Q$ included in $Q_{i,j}$, and
$\epsilon_{i,j}=0$ if the base block of $P$ and the base block of
$P_{i,j}$ are included in the same block of $Q$, and $1$ otherwise.
Hence:
\begin{align*}
	\beta*\alpha(P)
	&=\prod_{(i,j)\in J}\sum_{\substack{Q_{i,j}\geq P_{i,j},\\ \mbox{\scriptsize $Q_{i,j}$ irreducible,}\\
	\mbox{\scriptsize the factors of $P_{i,j}/Q_{i,j}$ irreducible}}}
	\left( (-1)^{\bl(Q_{i,j})+1}+(-1)^{\bl(Q_{i,j})}\right)=0=\varepsilon_\Bblnc(P).
\end{align*}
This shows that $\beta*\alpha(P)=0=\varepsilon_\Bblnc(P)$,
and therefore $\beta=\alpha^{*-1}=\psi_\prec^{*-1}*\psi_\succ$.
\end{proof}


\subsection{On the inverse of the ``free'' universal map $\psi_\prec$}

Lastly we investigate the behaviour of $\psi_\prec$, the universal map
associated to the left half-shuffle exponential --- the one relating free
cumulants and moments in free probability. Not surprisingly, its
combinatorics involves the Catalan numbers
$cat_{n+1}=\frac{1}{n+1}{2n\choose n}$. In fact the following lemma
\emph{characterizes} the Catalan number, as it allows to compute them
inductively.

\begin{lemma} \label{lemmacatalan}
For any $n\geq 1$:
\begin{align*}
	\sum_{k=0}^{\left\lfloor \frac{n}{2}\right\rfloor}(-1)^{n-k+1} cat_{n-k}\binom{n-k}{k}=\delta_{n,1}.
\end{align*}\end{lemma}

\begin{proof}
We denote by $\P$ the free non-symmetric operad on one binary generator: for any $n\geq 1$, $\P(n)$ is the vector space generated
by the set of plane binary trees with $n$ leaves (so has dimension 
$cat_n$), and the operad composition law is given by grafting on leaves.
For any $p\in \P(n)$, $q_1,\ldots, q_k \in \P$, we put:
\begin{align*}
	p\blacktriangleleft (q_1\ldots q_k)&
	=\sum_{1\leq i_1<\ldots<i_k\leq n} p\circ 
	(\underbrace{I,\ldots,I,p_1,I,\ldots,I,p_k,I,\ldots,I}
	_{\mbox{\scriptsize $p_j$ is position $i_j$}}).
\end{align*}
In particular, $p\blacktriangleleft(1)=p$ and, if $k>n$ then $p\blacktriangleleft (q_1\ldots q_k)=0$.
We denote by $I$ the unique plane binary tree with one leaf and by $Y$ the unique plane binary tree with two leaves. 
For all $n\geq 1$, we define inductively $X_n \in \P(n)$ by $X_1=1$ and:\\
\begin{align*}
	&\forall n\geq 2,&X_n&=\sum_{k=1}^{\left\lfloor \frac{n}{2}\right\rfloor} (-1)^{k+1}X_{n-k}\blacktriangleleft Y^k.
\end{align*}
We also put $X=\sum X_n$. Then $X$ is the unique solution in $\prod_{n\geq 1} \P(n)$ of the equation:
\begin{align*}
	X&=I+\sum_{k=1}^\infty (-1)^{k+1}X\blacktriangleleft Y^k,
\end{align*}
or equivalently:
\begin{align*}
	\sum_{k=0}^\infty (-1)^kX\blacktriangleleft Y^k&=I.
\end{align*}
Let us consider $X'=I+X\vee X$, where $\vee$ is the operator that
grafts two binary trees onto a new common root. Then:
\begin{align*}
	\sum_{k=0}^\infty (-1)^k X'\blacktriangleleft Y^k
	&=I-Y+0+\sum_{k=0}^\infty (X\vee X)\blacktriangleleft Y^k\\
	&=I-Y+\sum_{k=0}^\infty \sum_{i+j=k} (X\blacktriangleleft Y^i)\vee (X\blacktriangleleft Y^j)\\
	&=I-Y+\left(\sum_{i=0}^\infty (-1)^i X\blacktriangleleft Y^i\right) 
		\vee\left(\sum_{j=0}^\infty (-1)^j X\blacktriangleleft Y^j\right)\\
	&=I-Y+I\vee I\\
	&=I.
\end{align*}
Hence, $X'=X$, so $X=I+X\vee X$. An easy induction on  $n$ implies that $X_n$ is the sum of all plane binary trees with $n$ leaves.
Sending any plane binary tree to $1$, we get, for any $n\geq 2$:
\begin{align*}
	X_n(1)&
	=cat_n
	=\sum_{k=1}^{\left\lfloor \frac{n}{2}\right\rfloor} (-1)^{k+1}\binom{k}{n-k}X_{n-k}(1)
	=\sum_{k=1}^{\left\lfloor \frac{n}{2}\right\rfloor} (-1)^{k+1}\binom{k}{n-k}cat_{n-k}, 
\end{align*}
which gives the announced formula.
\end{proof}

For the following proposition, recall that $J_n$ denote the finest 
partitions and $I_n$ denote the coarsest partitions, for $n\geq 1$.

\begin{prop}\label{prop:psipreccat}
For any $n\geq 1$, we have $\psi_\prec^{*-1}(J_n)=(-1)^{n+1}cat_n$.
\end{prop}

\begin{proof} We put $\kappa=\Theta^{-1}(\psi_\prec^{*-1})$. 
  Since $\psi_\prec \curvearrowleft\overline{\Theta(\kappa)}=\varepsilon_\Bblnc$,
  we see that
$\varepsilon_\Bblnc=\varepsilon_\Delta+\kappa \prec \varepsilon_\Bblnc$. Hence, for all $n\geq 1$:
\begin{align*}
\delta_{n,1}&=(\kappa \otimes \varepsilon_\Bblnc)\circ \Delta_\prec(J_n).
\end{align*}
A study of uppersets of $J_n$ proves that:
\begin{align*}
\Delta(J_n)&=\sum_{l=0}^n 
\sum_{a_1+\cdots+a_l\leq n} \binom{n-a_1-\cdots-a_l+1}{l} 
J_{n-a_1-\cdots-a_l}\otimes J_{a_1}\cdots J_{a_l},\\
\Delta_\prec(J_n)&=\sum_{l=0}^n \sum_{a_1+\cdots+a_l\leq n} 
\binom{n-a_1-\cdots-a_l}{l} J_{n-a_1-\cdots-a_l}\otimes J_{a_1}\cdots J_{a_l}.
\end{align*}
For any $k\geq 1$, we have $\varepsilon_\Bblnc(J_k)=\delta_{k,1}$, so:
\begin{align*}
(\kappa \otimes \varepsilon_\Bblnc)\circ \Delta_\prec(J_n)
&=\sum_{l=0}^n \sum_{a_1+\cdots+a_l\leq n} \binom{n-a_1-\cdots-a_l}{l} 
\kappa(J_{n-a_1-\cdots-a_l})\otimes \delta_{a_1,1}\cdots \delta_{a_l,1}\\
&=\sum_{l=0}^n\binom{n-l}{l}\Expl^{*-1}(J_{n-l})\\
&=\delta_{n,1}.
\end{align*}
Hence, the sequence $(\psi_\prec^{*-1}(J_n))_{n\geq 1}$ is the unique sequence $(a_n)_{n\geq 1}$ such that for all $n\geq 1$:
\begin{align*}
\sum_{l=0}^n\binom{n-l}{l}a_{n-l}&=\delta_{n,1}.
\end{align*}
By Lemma \ref{lemmacatalan}, $a_n=(-1)^{n+1}cat_n$ for all $n$.
\end{proof}

For any $P\in \latNCP(k,n)$:
\begin{align*}
\psi_\prec^{*-1}*\psi_\prec(P)&=\sum_{Q\geq P} \psi_\prec^{*-1}(P)=\delta_{P,I_n}.
\end{align*} 
So $\psi_\prec^{*-1}(P)=\mob(P,I_n)$, where $\mob$ is the M\"obius function of the poset $\latNCP(n)$.
Proposition~\ref{prop:psipreccat} shows that the Euler characteristic of $\latNCP(n)$ is:
\begin{align*}
	\mob(J_n,I_n)&=(-1)^{n+1}cat_n.
\end{align*}
For any $m_1,\ldots,m_k \geq 1$, we put:
\begin{align*}
	I_{m_1,\ldots,m_k}&=\{\llbracket 1,m_1\rrbracket,\llbracket m_1
	+1,m_1+m_2\rrbracket,\ldots,
	\llbracket m_1+\cdots+m_{k-1}+1,m_1+\cdots+m_k\rrbracket\}.
\end{align*}
Let $P\in \latNP$. Then $P$ can be uniquely written as:
\begin{align*}
	P&=I_{n_1+1,\ldots,n_k+1}\diamond(\emptyset,P_{1,1},\ldots,P_{1,n_1},
	\emptyset,\ldots,\emptyset, 
	P_{k,1},\ldots,P_{k,n_k},\emptyset),
\end{align*}
where $k\geq 1$, $n_1,\ldots,n_k\geq 0$ and for all $i,j$, $P_{i,j}$ is a 
noncrossing partition, maybe empty.
Note that the irreducible components of $P$ are the noncrossing partitions 
$I_{n_i+1}\diamond(\emptyset,P_{i,1},\ldots,P_{i,n_i},\emptyset)$.
For any $Q \in \latNP$, we denote $\frontstick Q=I_1\diamond(\emptyset, Q)$. 
We have a poset isomorphism:
\begin{align*}
	\{Q\in \latNP \mid Q\geq P\}
	&\longrightarrow \latNCP(k)\times \prod_{i,j}\{Q\in \latNP \mid Q\geq 
	\frontstick P_{i,j}\}.
\end{align*}
The bijection sends $Q$ to $(Q_0,(Q_{i,j})_{i,j})$, where $Q_0$ is determined by identifying the base blocks of the irreducible components
of $P$ with the $k$ blocks of $J_k$, and $Q_{i,j}$ is obtained by identifying the base block of the $i$th irreducible component of $P$
with the block $\frontstick\ $ of $\frontstick P_{i,j}$, the blocks of $P_{i,j}$
being conserved. Hence:
\begin{align*}
\psi_\prec^{*-1}(Q)&=(-1)^{k+1}cat_k . \prod_{i,j} 
\psi_\prec^{*-1}(\frontstick P_{i,j}).
\end{align*}
This allows to compute $\psi_\prec^{*-1}(P)$ by induction on the number of blocks of cardinality $\geq 2$. For example:
\begin{align*}
\psi_\prec^{*-1}\left(\begin{tikzpicture}[line cap=round,line join=round,>=triangle 45,x=0.3cm,y=0.3cm]
\clip(0.8,1.8) rectangle (7.2,4.2);
\draw [line width=0.8pt] (1.,3.)-- (1.,2.);
\draw [line width=0.8pt] (1.,2.)-- (6.,2.);
\draw [line width=0.8pt] (6.,2.)-- (6.,3.);
\draw [line width=0.8pt] (2.,4.)-- (2.,3.);
\draw [line width=0.8pt] (2.,3.)-- (3.,3.);
\draw [line width=0.8pt] (3.,3.)-- (3.,4.);
\draw [line width=0.8pt] (4.,4.)-- (4.,3.);
\draw [line width=0.8pt] (4.,3.)-- (5.,3.);
\draw [line width=0.8pt] (5.,3.)-- (5.,4.);
\draw [line width=0.8pt] (7.,2.)-- (7.,3.);
\end{tikzpicture}\right)&=-cat_2.\psi_\prec^{*-1}\left(\begin{tikzpicture}[line cap=round,line join=round,>=triangle 45,x=0.3cm,y=0.3cm]
\clip(1.8,2.8) rectangle (6.2,4.2);
\draw [line width=0.8pt] (3.,4.)-- (3.,3.);
\draw [line width=0.8pt] (3.,3.)-- (4.,3.);
\draw [line width=0.8pt] (4.,3.)-- (4.,4.);
\draw [line width=0.8pt] (5.,4.)-- (5.,3.);
\draw [line width=0.8pt] (5.,3.)-- (6.,3.);
\draw [line width=0.8pt] (6.,3.)-- (6.,4.);
\draw [line width=0.8pt] (2.,3.)-- (2.,4.);
\end{tikzpicture}
\right)=-cat_2. cat_3,\\
\psi_\prec^{*-1}\left(
\begin{tikzpicture}[line cap=round,line join=round,>=triangle 45,x=0.3cm,y=0.3cm]
\clip(0.8,1.8) rectangle (8.2,4.2);
\draw [line width=0.8pt] (1.,3.)-- (1.,2.);
\draw [line width=0.8pt] (1.,2.)-- (7.,2.);
\draw [line width=0.8pt] (7.,2.)-- (7.,3.);
\draw [line width=0.8pt] (2.,4.)-- (2.,3.);
\draw [line width=0.8pt] (2.,3.)-- (3.,3.);
\draw [line width=0.8pt] (3.,3.)-- (3.,4.);
\draw [line width=0.8pt] (5.,4.)-- (5.,3.);
\draw [line width=0.8pt] (5.,3.)-- (6.,3.);
\draw [line width=0.8pt] (6.,3.)-- (6.,4.);
\draw [line width=0.8pt] (8.,2.)-- (8.,3.);
\draw [line width=0.8pt] (4.,2.)-- (4.,3.);
\end{tikzpicture}\right)&=-cat_2.
\psi_\prec^{*-1}\left(\begin{tikzpicture}[line cap=round,line join=round,>=triangle 45,x=0.3cm,y=0.3cm]
\clip(1.8,2.8) rectangle (4.2,4.2);
\draw [line width=0.8pt] (3.,4.)-- (3.,3.);
\draw [line width=0.8pt] (3.,3.)-- (4.,3.);
\draw [line width=0.8pt] (4.,3.)-- (4.,4.);
\draw [line width=0.8pt] (2.,3.)-- (2.,4.);
\end{tikzpicture}\right)^2=-cat_2^3.
\end{align*}


\footnotesize{
}

\begin{thebibliography}{ab}

\bibitem{Aguiar-Mahajan} 
M.~Aguiar, S.~Mahajan,
\newblock {\em Monoidal functors, species and {H}opf algebras}, vol.~29 of CRM Monograph Series.
\newblock American Mathematical Society, Providence, RI, 2010.
\newblock With forewords by K.~Brown, S.~Chase and A.~Joyal.

\bibitem{Baumeister-et.al}
B.~Baumeister, K.-U.~Bux, F.~G{\"o}tze, D.~Kielak, H.~Krause,
{\it{Non-crossing partitions}}, 
arxiv:1903.01146.

\bibitem{Biane0} 
P.~Biane,
{\it{Some properties of crossings and partitions}}, 
Discrete Math.~{\bf{175}} (1997), 41.
	
\bibitem{Biane} 
P.~Biane,
{\it{Free probability and combinatorics}},
Proceedings of the International Congress of Mathematicians, 
Vol.~II (Beijing, 2002), 765.

\bibitem{Bruned-Hairer-Zambotti}
Y.~Bruned, M.~Hairer, L.~Zambotti,
\newblock {\em Algebraic renormalisation of regularity structures},
\newblock Invent.~Math.~{\bf 215} (2019), 1039. 

\bibitem{CEFM} 
D.~Calaque, K.~Ebrahimi-Fard, D.~Manchon,
{\emph{Two interacting Hopf algebras of trees: A Hopf-algebraic approach to 
composition and substitution of B-series}},
Adv.~Appl.~Math.~{\bf{47}} (2011), 282.
  
\bibitem{chapotonMM} 
F.~Chapoton, 
{\it{Un th\'eor\`eme de Cartier--Milnor--Moore--Quillen pour les big\`ebres dendriformes et les alg\`ebres braces}}, 
J.~Pure Appl.~Algebra {\bf{168}} (2002), 1.

\bibitem{DC1}
G.~C.~Drummond-Cole,
{\emph{An operadic approach to operator-valued free cumulants}}, 
Higher Structures {\bf{2}} (2018), 42.

\bibitem{DC2}
G.~C.~Drummond-Cole,
{\emph{A non-crossing word cooperad for free homotopy probability theory}}, 
MATRIX Book Series {\bf{1}} (2018), 77.

\bibitem{EFM17}
K.~Ebrahimi-Fard, F.~Fauvet, D.~Manchon,
{\it{A comodule-bialgebra structure for word-series substitution and mould composition}},
J.~Algebra {\bf{489}} (2017), 552.

\bibitem{EFP15}
K.~Ebrahimi-Fard, F.~Patras,
{\emph{Cumulants, free cumulants and half-shuffles}}, 
Proc.~Royal Soc.~A {\bf{471}} 2176, (2015).

\bibitem{EFP16}
K.~Ebrahimi-Fard, F.~Patras,
{\emph{The splitting process in free probability theory}},
Int.~Math.~Res.~Notices {\bf 9} (2016), 2647.

\bibitem{EFP18}
K.~Ebrahimi-Fard, F.~Patras, 
{\it{Monotone, free, and boolean cumulants from a Hopf algebraic point of view}},
Adv.~Math.~{\bf{328}} (2018), 112.

\bibitem{EFP17}
K.~Ebrahimi-Fard, F.~Patras,
{\emph{Shuffle group laws. Applications in free probability}}, 
Proc.~London Math.~Soc.~{\bf{119}} (2019), 814.

\bibitem{Foissy:trees}
L.~Foissy,
\newblock {\em Les alg\`ebres de {H}opf des arbres enracin\'es d\'ecor\'es. {I}},
\newblock Bull.~Sci.~Math.~{\bf 126} (2002), 193. 

\bibitem{Foissy}
L.~Foissy,
{\it{Bidendriform bialgebras, trees, and free quasi-symmetric functions}}, 
J.~Pure Appl.~Algebra {\bf{209}} (2007), 439.

\bibitem{FoissyOperads}
L.~Foissy,
{\it{Algebraic structures associated to operads}},
arXiv:1702.05344.

\bibitem{Friedrich}
R.~M.~Friedrich, J.~McKay,
{\it{Homogeneous Lie Groups and Quantum Probability}},
arXiv:1506.07089. 
			
\bibitem{Gabriel15}
F.~Gabriel,
{\emph{Combinatorial theory of permutation-invariant random matrices I: partitions, geometry and renormalization}},
arXiv:1503.02792.

\bibitem{GKT1} 
I.~G\'alvez-Carrillo, J.~Kock, A.~Tonks,
{\it{Decomposition spaces, incidence algebras and M\"obius inversion I: basic theory}},
Adv.~Math.~{\bf{331}} (2018), 952.

\bibitem{GKT-comb} 
I.~G\'alvez-Carrillo, J.~Kock, A.~Tonks,
{\it{Decomposition spaces in combinatorics}},
arxiv:1612.09225.

\bibitem{gv} 
M.~Gerstenhaber, A.~A.~Voronov, 
{\emph{Homotopy $G$-algebras and moduli space operads}}, 
Int.~Math.~Res.~Notices {\bf{3}} (1995), 141.

\bibitem{HasLeh}
T.~Hasebe, F.~Lehner, 
{\emph{Cumulants, Spreadability and the Campbell--Baker--Hausdorff Series}}, 
arXiv:1711.00219.

\bibitem{Thibon}
M.~Josuat-Verg\`es, F.~Menous, J.-C.~Novelli, J.-Y.~Thibon,
{\it{Free cumulants, Schr{\"o}der trees, and operads}},
Adv.~Appl.~Math.~{\bf{88}} (2017), 92.

\bibitem{Kendall:vol1}
M.~G. Kendall.
\newblock {\em The {A}dvanced {T}heory of {S}tatistics}.
\newblock Vol. I. J. B. Lippincott Co., Philadelphia, 1944.

\bibitem{Kock-Weber} 
J.~Kock, M.~Weber, 
{\it{Fa\`a di Bruno for operads and internal algebras}}, 
J.~London Math.~Soc.~{\bf{99}} (2019), 919.

\bibitem{Kre} 
G.~Kreweras, 
{\emph{Sur les partitions non croisees d'un cycle}}, 
Discrete Math.~{\bf 1} (1972), 333.

\bibitem{Lehner}
F.~Lehner,
{\emph{Free cumulants and enumeration of connected partitions}}, 
Europ.~J.~Combinatorics {\bf{23}} (2002), 1025.

\bibitem{Male}
C.~Male, \emph{Traffic distributions and independence: permutation invariant
random matrices and the three notions of independence}.
To appear in Mem. Amer. Math. Soc. arXiv:1111.4662.

\bibitem{Manchon18}
D.~Manchon,
{\emph{A review on comodule-bialgebras}},
in the proceedings of the 2016 Abel Symposium
``Computation and Combinatorics in Dynamics, Stochastics and Control'',
Springer's Abel Symposia vol.~13, 2018. 

\bibitem{MSch}
S.~Manzel, M.~Sch\"urmann,
{\it{Non-Commutative Stochastic Independence and Cumulants}},
Infin.~Dimensional Anal.~Quantum Probab.~Relat.~Top.~{\bf{20}} (2017), 1750010.
 
\bibitem{MastNica}
M.~Mastnak, A.~Nica, 
{\it{Hopf algebras and the logarithm of the $S$-transform in free probability}}, 
Trans.~Amer.~Math.~Soc.~{\bf{362}} (2010), 3705.

\bibitem{MC}
J.~McCammond, 
{\emph{Noncrossing Partitions in Surprising Locations}}, 
Am.~Math.~Mon.~{\bf{113}} (2006), 598.

\bibitem{SpeicherNica}
A.~Nica, R.~Speicher, 
{\it Lectures on the combinatorics of free probability},
London Mathematical Society Lecture Note Series, {\bf{335}} Cambridge University Press, 2006.

\bibitem{Rota0} 
G.-C.~Rota, 
{\emph{On the Foundations of Combinatorial Theory I. Theory of M\"obius Functions}},
Z.~Wahrscheinlichkeitstheorie {\bf{2}} (1964), 340.

\bibitem{RotaWallstrom97} 
G.-C.~Rota, T.~C.~Wallstrom,
{\it{Stochastic Integrals: A Combinatorial Approach}},
Ann.~Probab.~{\bf{25}} (1997), 1257.

\bibitem{Schmitt:1994}
W.~R.~Schmitt,
\newblock {\em Incidence {H}opf algebras},
J. Pure Appl.~Algebra {\bf 96} (1994), 299.

\bibitem{Simion}
R.~Simion,
{\emph{Noncrossing partitions}},
Discrete Math.~{\bf{217}} (2000), 367.

\bibitem{Speed} 
T.~P.~Speed, 
{\emph{Cumulants and Partition Lattices}}, 
Austral.~J.~Statist.~{\bf 25} (1983), 378.

\bibitem{Speicher94} 
R.~Speicher,
{\emph{Multiplicative functions on the lattice of non-crossing partitions and free convolution}},
Math.~Ann.~{\bf{298}} (1994),  611.

\bibitem{Speicher} 
R.~Speicher,
{\emph{Combinatorial theory of the free product with amalgamation and operator-valued free probability theory}},
Memoirs of the AMS {\bf{627}} (1998).
        
\bibitem{DVoi} 
D.~Voiculescu, 
{\it{Free Probability Theory: Random Matrices and von Neumann Algebras}}, 
Proceedings of the International Congress of Mathematicians, Z\"urich, 
Switzerland 1994. Birkh\"auser Verlag, Basel, Switzerland, 1995.

\end{thebibliography}
\end{document}